\documentclass{amsart}

\usepackage{latexsym,amsthm,mathrsfs,amsmath,amscd,enumerate,enumitem, amscd,color,dsfont, cancel, mathtools, textcomp, float}

\usepackage[utf8x]{inputenc}

\usepackage[cspex,bbgreekl]{mathbbol}
\usepackage{amsfonts}
\usepackage{amssymb}             

\DeclareSymbolFontAlphabet{\mathbbl}{bbold}
\DeclareSymbolFontAlphabet{\mathbb}{AMSb}%

 \newtheorem{thm}{Theorem}[section]
 \newtheorem{cor}[thm]{Corollary}
 
 \newtheorem{prop}[thm]{Proposition}
\theoremstyle{definition}

 \theoremstyle{remark}
 \newtheorem{rem}[thm]{Remark}

\usepackage{hyperref}
\usepackage{tikz}
\usetikzlibrary{plotmarks}

\newcommand{\supp}{\mathop{\mathrm{supp}}}

\makeatletter
\DeclareRobustCommand\widecheck[1]{{\mathpalette\@widecheck{#1}}}
\def\@widecheck#1#2{%
    \setbox\z@\hbox{\m@th$#1#2$}%
    \setbox\tw@\hbox{\m@th$#1%
       \widehat{%
          \vrule\@width\z@\@height\ht\z@
          \vrule\@height\z@\@width\wd\z@}$}%
    \dp\tw@-\ht\z@
    \@tempdima\ht\z@ \advance\@tempdima2\ht\tw@ \divide\@tempdima\thr@@
    \setbox\tw@\hbox{%
       \raise\@tempdima\hbox{\scalebox{1}[-1]{\lower\@tempdima\box
\tw@}}}%
    {\ooalign{\box\tw@ \cr \box\z@}}}
\makeatother

\usepackage{caption}
\usepackage{subcaption}

\usepackage[left=2.2cm,top=2.5cm,right=2.2cm,bottom=2.5cm]{geometry} 

\numberwithin{equation}{section}
\allowdisplaybreaks

\begin{document}

\title{Variation and oscillation for harmonic operators in the inverse Gaussian setting}

\author[V. Almeida]{V\'ictor Almeida}
\address{V\'ictor Almeida, Jorge J. Betancor, \newline
	Departamento de An\'alisis Matem\'atico, Universidad de La Laguna,\newline
	Campus de Anchieta, Avda. Astrof\'isico S\'anchez, s/n,\newline
	38721 La Laguna (Sta. Cruz de Tenerife), Spain}
\email{valmeida@ull.edu.es, jbetanco@ull.es}

\author[J. J. Betancor]{Jorge J. Betancor}



\thanks{This paper is partially supported by grant PID2019-106093GB-I00 from the Spanish Government}

\subjclass[2010]{42B20, 42B25.}

\keywords{}

\date{\today}

\begin{abstract}
We prove variation and oscillation $L^p$-inequalities associated with fractional derivatives of certain semigroups of operators and with the family of truncations of Riesz transforms in the inverse Gaussian setting. We also study these variational $L^p$-inequalities in a Banach-valued context by considering Banach spaces with the UMD-property and whose martingale cotype is fewer than the variational exponent. We establish $L^p$-boundedness properties for weighted difference involving the semigroups under consideration.

\end{abstract}

\maketitle

\section{Introduction}
Variational inequalities have been studied in probability, harmonic analysis and ergodic theory in the past decades. The main interest of these variational  inequalities for a family of operators is that they give us some information about the speed of convergence of the family of operators considered and how this convergence occurs.

Suppose that $(\Omega,\mu)$ is a measure space and $\{T_t\}_{t>0}$ is a set of operators in $L^p(\Omega,\mu)$ with $1\leq p<\infty$. The usual procedure to prove, for instance, the existence of the pointwise limit $\displaystyle\lim_{t\rightarrow 0}T_t(f)(x)$ for every $f\in L^p(\Omega,\mu)$ is the following one. In a first step we define a subspace $F$ of $L^p(\Omega,\mu)$ consisting of smooth enough functions for which it is easier to see that the limit exists for almost all $x\in\Omega$. A typical election is $F=C_c^\infty(\Omega)$. Then, we consider the maximal operator $T_*$ associated with $\{T_t\}_{t>0}$ defined by
$$T_*(f)=\sup_{t>0}|T_t(f)|.$$
It is established that $T_*$ is bounded from $L^p(\Omega,\mu)$ into itself, when $1<p<\infty$, or from $L^1(\Omega,\mu)$ into $L^{1,\infty}(\Omega,\mu)$, when $p=1$. Finally, the Banach principle is used to extend the existence of limit to every function in $L^p(\Omega,\mu)$.

Variational inequalities allow us to get pointwise convergence associated with $\{T_t\}_{t>0}$ without using Banach principle. In this sense variational inequalities are better than maximal inequalities. The first variation inequality was established by Leplinge (\cite{Le}) for martingales improving the Doob maximal inequality. A simple proof of Lepingle's inequality was given by Pisier and Xu (\cite{PX}). Bourgain (\cite{Bou1}) proved the first variational inequality in the ergodic setting by considering ergodic averages for dinamical systems. This inequality was extended from $L^2(\Omega)$ to $L^p(\Omega)$, $1<p<\infty$, in \cite{JKRW}.

After \cite{Bou1} many variational inequalities in harmonic analysis and ergodic theory have been proved. Classical families of operators have been considered. For instance: semigroups of operators (\cite{BCT}, \cite{CT}, \cite{HMMT}, \cite{JW} and \cite{LeX}), differential operators (\cite{HM2} and  \cite{MTX}), singular integrals (\cite{CJRW1}, \cite{CJRW2}, \cite{HLM}, \cite{HLP}, \cite{MTX} and \cite{MT}), multipliers (\cite{DK}), and Fourier series (\cite{DL} and \cite{OSTTW}).

In this paper we study oscillation and variation inequalities, jump functions and differential transforms in the inverse Gaussian setting.

Next, we set out  the definitions of the operators, and the context, that we are going to consider. The Gaussian measure is usually denoted by $\gamma$ and it is the measure with density function $\displaystyle\frac{e^{-|x|^2}}{\pi^{n/2}}$, $x\in\mathbb R^n$, with respect to the Lebesgue measure in $\mathbb R^n$. Our setting will be $(\mathbb R^n,\gamma_{-1})$ where by $\gamma_{-1}$ we denote the inverse Gaussian measure whose density function with respect to the Lebesgue measure is $e^{|x|^2}\pi^{n/2}$, $x\in\mathbb R^n$. We consider a family $\{L_t\}_{t>0}$ of operators defined on $L^p(\mathbb R^n, \gamma_{-1})$, $1\leq p<\infty$.

Let $\rho >2$. We define the $\rho$-variation operator $V_{\rho}(\{L_t\}_{t>0})$ of $\{L_t\}_{t>0}$ by

$$V_{\rho}(\{L_t\}_{t>0})(f)=\sup_{0<t_k<...<t_1,\;k\in\mathbb N}\left(\sum_{j=1}^{k-1}|L_{t_j}(f)-L_{t_{j+1}}(f)|^\rho\right)^{1/\rho}.$$

It is not directly seen that $V_{\rho}(\{L_t\}_{t>0})(f)$ is a measurable function in $\mathbb R^n$. As it was commented after the statement of \cite[Theorem 1.2]{CJRW1} some continuity in $t\in (0,\infty)$ must be satisfied by $\{L_t\}_{t>0}$.

In order to obtain $L^p$-variation inequalities it is usually needed to impose that $\rho >2$. (see \cite{Qi} for the case of martingales, and \cite{JW} for the Poisson semigroup).

For every $\lambda >0$ we define $\lambda$-jump  operator as follows

\begin{align*}
\Lambda(\{L_t\}_{t>0},\lambda)(f)(x)= & \sup\left\{n\in\mathbb N\;:\;\mbox{such that there exists}\; s_1<t_1\leq s_2<t_2....\leq s_n<t_n \right. \\
& \left. \mbox{with the property}\; |L_{t_i}(f)(x)-L_{s_i}(f)(x)|>\lambda,\;i=1,...,n\right\}
\end{align*}
We have that
\begin{equation}\label{eq1.1}
\lambda\left[\Lambda(\{L_t\}_{t>0},\lambda)(f)\right]^{1/\rho}\leq 2^{1+1/q}V_{\rho}(\{L_t\}_{t>0})(f),\;\;\;\lambda >0.
\end{equation}
Under some conditions (see \cite[Lemma 2.1]{JSW}) variational inequalities can be deduced from $L^p$- inequalities involving the jump function.

Convergence properties for $\{L_t(f)\}_{t>0}$ can be deduce from the finiteness of both $V_{\rho}(\{L_t\}_{t>0})(f)$ and  $\Lambda(\{L_t\}_{t>0},\lambda)(f)$, $\lambda>0$. However variation inequalities are not consequence of convergence (\cite[Remark 2.1]{JW}).

We also consider the following operators when the exponent is $2$. For every $k\in\mathbb N$ we define

$$V_k(\{L_t\}_{t>0})(f)=\sup_{\begin{array}{c} \frac{1}{2^k}<t_1<...<t_m\leq\frac{1}{2^{k-1}} \\
 m\in\mathbb N
\end{array}}\left(\sum_{j=1}^{m-1}|L_{t_j}(f)-L_{t_{j+1}}(f)|^2\right)^{1/2}.$$
We call the short variation operator associated with $\{L_t\}_{t>0}$ to that defined by
$$S_V(\{L_t\}_{t>0})(f)=\left(\sum_{k=-\infty}^\infty\left(V_k(\{L_t\}_{t>0})(f)\right)^2\right)^{1/2}.$$

If $\{t_j\}_{j\in\mathbb N}\subset (0,\infty)$ is a decreasing sequence that converges to $0$, the oscillation operator $\mathcal O(\{L_t\}_{t>0},\{t_j\}_{j\in\mathbb N})(f)$ associated with $\{L_t\}_{t>0}$ defined by $\{t_j\}_{j\in\mathbb N}$ is given by

$$\mathcal O(\{L_t\}_{t>0},\{t_j\}_{j\in\mathbb N})(f)=\left(\sum_{i=1}^\infty\sup_{t_{i+1}\leq s_{i+1}<s_i<t_i}|L_{s_i}(f)-L_{s_{i+1}}(f)|^2\right)^{1/2}.$$

The short variation and the oscillation operators can be defined with exponent $\rho >2$ but the new operators are controlled by $S_V$ and $O$, respectively. Some example of $\rho$-oscillation operators with $\rho <2$ that are unbounded in $L^p$ can be found in \cite{AJS}.

For every $k\in\mathbb N$ by $H_k$ we denote the Hermite polynomial defined by

$$H_k(z)=(-1)^ke^{z^2}\frac{d^k}{dz^k}e^{-z^2},\;\;\;\;z\in \mathbb R.$$
If $k=(k_1,...,k_n)\in\mathbb N^n$ the Hermite polynomial $\mathcal H_k$ in $\mathbb R^n$ is given by

$$\mathcal H_k(x)=\prod_{i=1}^n H_{k_i}(x_i),\;\;\;\;x=(x_1,...,x_n)\in\mathbb R^n.$$
We consider $\widetilde{\mathcal H}_k(x)=\mathcal H_k(x)e^{-|x|^2}$, $x\in\mathbb R^n$ and $k\in\mathbb N^n$. The sequence $\{\widetilde{\mathcal H}_k\}_{k\in\mathbb N^n}$ is a base for $L^2(\mathbb R^n,\gamma_{-1})$.

We denote by $\mathcal A$ the differential operator

$$\mathcal A=-\frac12\Delta-x\cdot\nabla\;\;\;\;\mbox{in}\;\mathbb R^n.$$
We have that $\mathcal A\widetilde{\mathcal H}_k=(|k|+n)\widetilde{\mathcal H}_k$, $k\in\mathbb N$. Here, $|k|=k_1+...+k_n$, when $k=(k_1,...,k_n)\in\mathbb N^n$.

We define the operator $\overline{\mathcal A}$ by

$$\overline{\mathcal A}f=\sum_{k\in\mathbb N^n}c_k(f)(|k|+n)\widetilde{\mathcal H}_k, \;\;\;f\in D(\overline{\mathcal A}),$$
where the domain $ D(\overline{\mathcal A})$ of $\overline{\mathcal A}$ is given by

$$ D(\overline{\mathcal A})=\left\{f\in L^2(\mathbb R^n,\gamma_{-1})\;:\;\sum_{k\in\mathbb N^n}|c_k(f)|^2(|k|+n)^2\|\widetilde{\mathcal H}_k\|^2_{ L^2(\mathbb R^n,\gamma_{-1})}<\infty\right\}.$$
Here, for every $f\in L^2(\mathbb R^n,\gamma_{-1})$,

$$c_k(f)=\|\widetilde{\mathcal H}_k\|^{-2}_{ L^2(\mathbb R^n,\gamma_{-1})}\int_{\mathbb R^n}f(x)\widetilde{\mathcal H}_k(x)e^{|x|^2}\pi^{n/2}dx,\;\;\;k\in\mathbb N^n.$$
If $f\in C^\infty_c(\mathbb R^n)$, the space of smooth functions with compact support in $\mathbb R^n$, $\mathcal A f=\overline{\mathcal A}f$. $\overline{\mathcal A}$ is closed and selfadjoint. In the sequel we will write $\mathcal A$ to refer us also to $\overline{\mathcal A}$.

We define, for every $t>0$,

$$T_t^{\mathcal A}(f)=\sum_{k\in\mathbb N^n}c_k(f)e^{-(|k|+n)}\widetilde{\mathcal H}_k,\;\;\;\;\;f\in  L^2(\mathbb R^n,\gamma_{-1}).$$
$\{T_t^{\mathcal A}\}_{t>0}$ is the semigroup of operators generated by $-\mathcal A$. By using Mehler formula we can see that, for every $t>0$,

\begin{equation}\label{eq1.2}
T_t^{\mathcal A}(f)(x)=\int_{\mathbb R^n}T_t^{\mathcal A}(x,y)f(y)dy,\;\;\;\; f\in L^2(\mathbb R^n,\gamma_{-1}),
\end{equation}
where
$$T_t^{\mathcal A}(x,y)=\frac{e^{-nt}}{\pi^{n/2}}\frac{e^{-|x-e^{-t}y|^2/(1-e^{-2t})}}{(1-e^{-2t})^{n/2}},\;\;\;\;x,y\in\mathbb R^n.$$
$\{T_t^{\mathcal A}\}_{t>0}$, where $T_t^{\mathcal A}$ is defined by (\ref{eq1.2}) for every $t>0$, is a symmetric diffusion semigroup in $ L^p(\mathbb R^n,\gamma_{-1})$, $1\leq p<\infty$.

Harmonic analysis associated with the operator $\mathcal A$ was initiated by Salogni (\cite{Sa}). She studied the maximal operator $T_*^{\mathcal A}$ defined by $\{T_t^{\mathcal A}\}_{t>0}$ and some $\mathcal A$-spectral multipliers in $ L^p(\mathbb R^n,\gamma_{-1})$. Bruno (\cite{Br}) introduced Hardy type spaces adapted to inverse Gaussian setting. Recently, Bruno and Sjorgren (\cite{BrSj}) investigated Riesz transforms associated with $\mathcal A$ proving that they are bounded from $ L^1(\mathbb R^n,\gamma_{-1})$ into $ L^{1,\infty}(\mathbb R^n,\gamma_{-1})$ if  only if its order is less than three.  Higher order Riesz transforms in the $\mathcal A$-setting have been studied in \cite{B1}. In \cite{BCdL}, an study about maximal operator associated with $\{T_t^{\mathcal A}\}_{t>0}$ on K\"othe functions spaces is developed.

The subordinate Poisson semigroup $\{P_t^{\mathcal A}\}_{t>0}$ generated by $-\sqrt{\mathcal A}$ is given by

$$P_t^{\mathcal A}(f)=\frac{t}{2\sqrt{\pi}}\int_0^\infty\frac{e^{-t^2/4u}}{u^{3/2}}T_u^{\mathcal A}(f)(u)du.$$

If $m\in \mathbb N$, $m-1\leq\alpha < m$ and $g\in C^m(0,\infty)$ the Weyl derivative $D^\alpha g$ is defined by

$$D^\alpha g(t)=\frac{1}{\Gamma(m-\alpha)}\int_0^\infty g^{(m)}(t+s)s^{m-\alpha-1}ds,\;\;\;\;t\in(0,\infty).$$

Our next results are concerned with oscillation and variational operators associated with the uniparametric families of operators $\{t^\alpha\partial^\alpha_t T^{\mathcal A}_t\}_{t>0}$ and $\{t^\alpha\partial^\alpha_t P^{\mathcal A}_t\}_{t>0}$.

\begin{thm}\label{th1.1}
Suppose that $\alpha\geq 0$, $\rho >2$, $\lambda >0$ and $\{t_j\}_{j\in\mathbb N}$ is a decreasing sequence that converges to zero. The operators $\mathcal O(\{t^\alpha\partial^\alpha_t T^{\mathcal A}_t\}_{t>0},\{t_j\}_{j\in\mathbb N})$, $V_{\rho}(\{t^\alpha\partial^\alpha_t T^{\mathcal A}_t\}_{t>0})$, $S_V(\{t^\alpha\partial^\alpha_t T^{\mathcal A}_t\}_{t>0})$ and $\lambda\Lambda(\{t^\alpha\partial^\alpha_t T^{\mathcal A}_t\}_{t>0},\lambda)^{1/\rho}$ are bounded
\begin{enumerate}
    \item[(i)] from $ L^p(\mathbb R^n,\gamma_{-1})$ into itself, for every  $1< p<\infty$,
    \item[(ii)] from $ L^1(\mathbb R^n,\gamma_{-1})$ into $ L^{1,\infty}(\mathbb R^n,\gamma_{-1})$, when $0\leq\alpha\leq 1$.
\end{enumerate}
The operators $\mathcal O(\{t^\alpha\partial^\alpha_t P^{\mathcal A}_t\}_{t>0},\{t_j\}_{j\in\mathbb N})$, $V_{\rho}(\{t^\alpha\partial^\alpha_t P^{\mathcal A}_t\}_{t>0})$, $S_V(\{t^\alpha\partial^\alpha_t P^{\mathcal A}_t\}_{t>0})$ and $\lambda\Lambda(\{t^\alpha\partial^\alpha_t P^{\mathcal A}_t\}_{t>0},\lambda)^{1/\rho}$ are bounded from $ L^p(\mathbb R^n,\gamma_{-1})$, for every $1<p<\infty$, into itself and from $ L^1(\mathbb R^n,\gamma_{-1})$ into $ L^{1,\infty}(\mathbb R^n,\gamma_{-1})$. Furthermore the $L^p$-boundedeness properties of the $\lambda$-jump operator hold uniformly in $\lambda$.
\end{thm}

According to \cite[Corollary 4.5 and the comments in p. 2092]{LeX} the operators  $\mathcal O(\{t^k\partial^k_t S_t\}_{t>0},\{t_j\}_{j\in\mathbb N})$, \newline $V_{\rho}(\{t^k\partial^k_t S_t\}_{t>0})$ and $\lambda\Lambda(\{t^k\partial^k_t S_t\}_{t>0},\lambda)^{1/\rho}$ are bounded from $ L^p(\mathbb R^n,\gamma_{-1})$ into itself, for every $1<p<\infty$ and $k\in\mathbb N$, when $S_t=T_t^{\mathcal A}$ and $S_t=P_t^{\mathcal A}$, $t>0$. We prove our results by using a different procedure that allows us to obtain the strong and the weak type results. We consider local and global operators and we prove that the $L^p$-behaviour of the local ones are the same than those corresponding operators associated to the classical Euclidean heat and Poisson semigroups in $\mathbb R^n$. Our procedure also works in a vector valued setting.

Let $X$ be a Banach space. We denote by $ L^p_X(\mathbb R^n,\gamma_{-1})$, $1\leq p<\infty$, the $p$th-Bochner-Lebesgue space. If $t>0$ and $1\leq p<\infty$, since $T_t^{\mathcal A}$ and $P_t^{\mathcal A}$ are positive contractions in $L^p(\mathbb R^n,\gamma_{-1})$, they can be extended from $L^p(\mathbb R^n,\gamma_{-1})\otimes X$ to $L^p_X(\mathbb R^n,\gamma_{-1})$ as bounded contractions in $L^p_X(\mathbb R^n,\gamma_{-1})$. For every $f\in L^p_X(\mathbb R^n,\gamma_{-1})$, $1\leq p<\infty$, and $\rho >2$, the variation operators  $V_{\rho}^X(\{T^{\mathcal A}_t\}_{t>0})$ and $V_{\rho}^X(\{P^{\mathcal A}_t\}_{t>0})$ are defined as above by replacing absolute value by $X$-norms.

We recall that it is said that $X$ has $q$-martingale cotype (see \cite{Pi1}) with $q\geq 2$ when there exists $C>0$ such that for every martingale $(M_n)_{n\in\mathbb N}$ defined on some probability space with values in $X$ the following inequality holds

$$\sum_{n\in\mathbb N}\mathbb E\|M_n-M_{n-1}\|^q\leq C\sup_{n\in\mathbb N}\mathbb E\|M_n\|^q.$$
Here $\mathbb E$ denotes the expectation and $M_{-1}=0$.

\begin{thm}\label{th1.2}
Given $\alpha\geq 0$ and $2\leq q<\rho <\infty$. Suppose that $X$ is a Banach space with $q$-martingale cotype. Then,
\begin{enumerate}
    \item[(a)] The variation operator $V_{\rho}^X(\{t^\alpha\partial^\alpha_t T^{\mathcal A}_t\}_{t>0})$ is bounded from $ L^p_X(\mathbb R^n,\gamma_{-1})$ into itself, for every  $1< p<\infty$. Furthermore, $V_{\rho}^X(\{t^\alpha\partial^\alpha_t T^{\mathcal A}_t\}_{t>0})$ is bounded from $ L^1_X(\mathbb R^n,\gamma_{-1})$ into $ L^{1,\infty}_X(\mathbb R^n,\gamma_{-1})$ provided that $0\leq\alpha\leq 1$.

    \item[(b)] The variation operator $V_{\rho}^X(\{t^\alpha\partial^\alpha_t P^{\mathcal A}_t\}_{t>0})$ is bounded from $ L^p_X(\mathbb R^n,\gamma_{-1})$ into itself, for every  $1< p<\infty$, and  from $ L^1_X(\mathbb R^n,\gamma_{-1})$ into $ L^{1,\infty}_X(\mathbb R^n,\gamma_{-1})$.
\end{enumerate}
\end{thm}

Let $i=1,...,n$. The $i$-th Riesz transform $\mathcal R_i^{\mathcal A}$ associated with the operator $\mathcal A$ is defined, for every $f\in L^p(\mathbb R^n,\gamma_{-1})$, $1\leq p<\infty$,

$$\mathcal R_i^{\mathcal A}(f)(x)=\lim_{\epsilon\rightarrow 0^+}\int_{|x-y|>\epsilon}\mathcal R_i^{\mathcal A}(x,y)f(y)dy,\;\;\mbox{for almost all }x\in\mathbb R^n,$$
where

$$\mathcal R_i^{\mathcal A}(x,y)=\frac{1}{\sqrt{\pi}}\int_0^\infty \partial_{x_i}T_t^{\mathcal A}(x,y)\frac{dt}{\sqrt{t}}, \;\;\;x,y\in\mathbb R^n,\;x\neq y.$$
In \cite{Br1} and \cite{BrSj} it was proved that $\mathcal R_i^{\mathcal A}$ defines a bounded operator from  $L^p(\mathbb R^n,\gamma_{-1})$ into itself, for every $1<p<\infty$, and from $L^1(\mathbb R^n,\gamma_{-1})$ into $L^{1,\infty}(\mathbb R^n,\gamma_{-1})$. We define, for every $\epsilon >0$,

$$\mathcal R_{i,\epsilon}^{\mathcal A}(f)(x)=\int_{|x-y|>\epsilon}\mathcal R_i^{\mathcal A}(x,y)f(y)dy,\;\;\;\;x\in\mathbb R^n,$$

\begin{thm}\label{th1.3}
Let $i=1,...,n$, $\rho >2$, $\lambda >0$, and $\{t_j\}_{j\in\mathbb N}$ be a decreasing sequence in $(0,\infty)$ that converges to $0$. The operators $\mathcal O(\{\mathcal R_{i,\epsilon}^{\mathcal A}\}\}_{\epsilon>0},\{t_j\}_{j\in\mathbb N})$, $V_{\rho}(\{\mathcal R_{i,\epsilon}^{\mathcal A}\}_{\epsilon>0},)$, $S_V(\{\mathcal R_{i,\epsilon}^{\mathcal A}\}_{\epsilon>0},)$ and $\lambda\Lambda(\{\mathcal R_{i,\epsilon}^{\mathcal A}\}_{\epsilon>0},,\lambda)^{1/\rho}$ are bounded from $L^p(\mathbb R^n,\gamma_{-1})$ into itself, for every $1<p<\infty$, and from $L^1(\mathbb R^n,\gamma_{-1})$ into $L^{1,\infty}(\mathbb R^n,\gamma_{-1})$. Furthermore the $L^p$-boundedeness properties of the $\lambda$-jump operator hold uniformly in $\lambda$.
\end{thm}

Muckenhoupt (\cite{Mu}) introduced Cauchy-Riemann equations in the Gaussian context (see also \cite{AFS}). We can write when $n=1$

$$\mathcal A f=-\frac12\frac{d}{dx}\left(e^{-x^2}\frac{d}{dx}\left(e^{x^2}f\right)\right)+f,\;\;\;\;\mbox{in}\;\mathbb R^n.$$
This suggests to introduce in the $\mathcal A$-context the following Cauchy-Riemann equations

$$\left.\begin{array}{c}
\frac{\partial}{\partial x}u=-\frac{\partial v}{\partial t} \\
 \\
\frac{\partial}{\partial t}u=\frac12e^{-x^2}\frac{d}{dx}(e^{x^2}v)
\end{array}\right\}$$
and to say that $v$ is the conjugated of $u$ when the above equations hold. Motivated by this, we define the $\mathcal A$-conjugation operators as follows. For every $i=1,...,n$,

$$\mathcal C_{i,t}^{\mathcal A}(f)(x)=\int_t^\infty \partial_{x_i}P_s^{\mathcal A}(f)(x)ds,\;\;\;\;x\in\mathbb R^n\;\mbox{and}\;t>0.$$
We have that. for every $i=1,...,n$,
$$\frac{\partial}{\partial t}\mathcal C_{i,t}^{\mathcal A}(f)(x)=-\partial_{x_i}P_t^{\mathcal A}(f)(x)$$
and
\begin{align*}
\frac12\sum_{i=1}^n e^{-x_i^2}\partial_{x_i}e^{x_i^2}\mathcal C_{i,t}^{\mathcal A}(f)(x) & =\int_t^\infty \frac12\sum_{i=1}^n e^{-x_i^2}\partial_{x_i}e^{x_i^2}\partial_{x_i}P_s^{\mathcal A}(f)(x)ds \\
& =\int_t^\infty (-\mathcal A +n)P_s^{\mathcal A}(f)(x)ds \\
& =\frac{d}{dt}P_t^{\mathcal A}(f)(x)+n\int_t^\infty P_s^{\mathcal A}(f)(x)ds.
\end{align*}

\begin{thm}\label{th1.4}
 Let $i=1,...,n$, $\rho >0$, $\lambda >0$,  and $\{t_j\}_{j\in\mathbb N}$ be a decreasing sequence in $(0,\infty)$ that converges to $0$. The operators $\mathcal O(\{\mathcal C_{i,t}^{\mathcal A}\}_{t>0},\{t_j\}_{j\in\mathbb N})$, $V_{\rho}(\{\mathcal C_{i,t}^{\mathcal A}\}_{t>0})$, $S_V(\{\mathcal C_{i,t}^{\mathcal A}\}_{t>0})$ and $\lambda\Lambda(\{\mathcal C_{i,t}^{\mathcal A}\}_{t>0},\lambda)^{1/\rho}$ are bounded from $L^p(\mathbb R^n,\gamma_{-1})$ into itself, for every $1<p<\infty$. Furthermore the $L^p$-boundedeness properties of the $\lambda$-jump operator hold uniformly in $\lambda$.
 \end{thm}

 In order to prove this theorem we need to establish the corresponding result when we consider the classical Euclidean conjugation operators in $\mathbb R^n$. (see subsection 3.3).

 Let $X$ be a Banach space. Suppose that $\{M_n\}_{n=1}^k$ is a $X$-valued martingale. The sequence $\{d_n=M_n-M_{n-1}\}_{n=1}^k$, where we consider $M_0=0$, is named the martingale difference sequence associated with $\{M_n\}_{n=1}^k$. Let $1<p<\infty$. We say that $X$ has the UMD$_p$-property when there exists $C>0$ such that for every $L^p$-martingale difference $\{d_n\}_{n=1}^k$ we have

 $$\mathbb E\left\|\sum_{n=1}^k  \epsilon_nd_n\right\|^p\leq C\mathbb E\left\|\sum_{n=1}^k d_n\right\|^p.$$
 UMD is an abbreviation of unconditional martingale difference. $X$ is UMD$_p$ for some $1<p<\infty$, if and only if, $X$ is UMD$_p$ for every $1<p<\infty$. Then, we write UMD without any reference to $p$.

 Burkholder (\cite{Bu}) and Bourgain (\cite{Bou2}) characterized the Banach spaces $X$ with the UMD-property as those ones for which the Hilbert transform can be extended to $L^p(\mathbb R, dx)\otimes X$ as a bounded operator from $L^p_X(\mathbb R, dx)$ into itself, for every $1<p<\infty$. In \cite[Theorem 1.2]{B1} it was used the Riesz transforms $\mathcal R_i^{\mathcal A}$, $i=1,...,n$ to characterize the Banach spaces with the UMD property.

 UMD Banach spaces share many of the good properties of Hilbert spaces. This class of Banach spaces, as the experience has shown, can be seen as the right one for developing vector valued harmonic analysis.

 $L^p$-boundedness properties of the $\rho$-variation operator associated with the $\mathcal A$-conjugation in a Banach-valued setting is connected with UMD-property and the martingale cotype of $X$.

 \begin{thm}\label{th1.5}
 Let $X$ be a Banach space and $1<p<\infty$.
 \begin{enumerate}
     \item[(i)] Assume that $2\leq q_0<\rho<\infty$ and that $X$ is of martingale cotype $q_0$ and satisfies the UMD-property. Then, for every $i=1,...,n$,
     $$\left\|V_{\rho}(\{\mathcal C_{i,t}^{\mathcal A}\}_{t>0})(f)\right\|_{L^p(\mathbb R^n,\gamma_{-1})}\leq C\|f\|_{L^p_X(\mathbb R^n,\gamma_{-1})},\;\;\;f\in L^p_X(\mathbb R^n,\gamma_{-1}).$$

     \item[(ii)] Let $2\leq\rho<\infty$. If, for very $i=1,...,n$,
     $$\left\|V_{\rho}(\{\mathcal C_{i,t}^{\mathcal A}\}_{t>0})(f)\right\|_{L^p(\mathbb R^n,\gamma_{-1})}\leq C\|f\|_{L^p_X(\mathbb R^n,\gamma_{-1})},\;\;\;f\in L^p_X(\mathbb R^n,\gamma_{-1}),$$
     then $X$ is of martingale cotype $\rho$ and has the UMD-property.
 \end{enumerate}
 \end{thm}

 From the last theorem we deduce that if the Banach space $X$ has the UMD property and is of martingale cotype $q$, for some $q>2$, then for every $f\in L^p_X(\mathbb R^n,\gamma_{-1})$, $1<p<\infty$, the limit $\displaystyle\lim_{t\rightarrow 0^+}\mathcal C_{i,t}^{\mathcal A}(f)(x)$, exists, for almost all $x\in\mathbb R^n$, for each $i=1,...,n$. In particular, for every $1<p<\infty$, $f\in L^p(\mathbb R^n,\gamma_{-1})$, and $i=1,...,n$, the limit $\displaystyle\lim_{t\rightarrow 0^+}\mathcal C_{i,t}^{\mathcal A}(f)(x)$, exists, for almost all $x\in\mathbb R^n$.

 \begin{cor}\label{cor1.6}
 Let $1<p<\infty$ and $i=1,...,n$. For every $f\in L^p(\mathbb R^n,\gamma_{-1})$,
 $$\mathcal R_i^{\mathcal A}(f)(x)=\lim_{t\rightarrow 0^+}\mathcal C_{i,t}^{\mathcal A}(f)(x),\;\;\;\mbox{for almost all }x\in\mathbb R^n.$$
 \end{cor}

Let $1\leq p<\infty$ and $f\in L^p(\mathbb R^n,\gamma_{-1})$. We have that

$$\lim_{t\rightarrow 0^+}T_t^{\mathcal A}(f)(x)=f(x)\;\;\mbox{and}\;\;\lim_{t\rightarrow +\infty}T_t^{\mathcal A}(f)(x)=0,\;\;\;\mbox{for almost all }x\in\mathbb R^n.$$
Then, if $(t_j)_{j=-\infty}^\infty$ is a decreasing sequence in $(0,\infty)$ such that $\displaystyle\lim_{j\rightarrow -\infty}t_j=0$ and $\displaystyle\lim_{j\rightarrow +\infty}t_j=+\infty$,

 $$\sum_{j=-\infty}^\infty \left(T_{t_{j+1}}^{\mathcal A}(f)(x)-T_{t_j}^{\mathcal A}(f)(x)\right)=f(x),\;\;\;\mbox{for almost all }x\in\mathbb R^n.$$
Here the cancellation is important for the convergence. It is natural to ask about the convergence of the series

$$\sum_{j=-\infty}^\infty v_j\left(T_{t_{j+1}}^{\mathcal A}(f)(x)-T_{t_j}^{\mathcal A}(f)(x)\right)$$
where $\{v_j\}_{j=-\infty}^\infty$ is a bounded sequence of complex numbers. Special cases appear when $v_j\in\{+1,-1\}$ for every $j\in\mathbb Z$.

Assume that $\{t_j\}_{j=-\infty}^\infty$ is an increasing sequence in $(0,\infty)$ such that $\displaystyle\lim_{j\rightarrow -\infty}t_j=0$ and $\{v_j\}_{j=-\infty}^\infty$ is a bounded sequence of complex numbers. For every $N=(N_1,N_2)\in\mathbb Z^2$, $N_1<N_2$ and $\alpha \geq 0$, we define

$$T_N^{\mathcal A,\alpha}(f)=\sum_{j=N_1}^{N_2}v_j\left(t^\alpha\partial_t^\alpha T_t^{\mathcal A}(f)_{|t=t_{j+1}}-t^\alpha\partial_t^\alpha T_t^{\mathcal A}(f)_{|t=t_j}\right).$$
We consider the maximal operator

$$T_*^{\mathcal A,\alpha}(f)=\sup_{\begin{array}{c} N=(N_1,N_2)\in\mathbb Z^2 \\
 N_1<N_2
\end{array}}\left|T_N^{\mathcal A,\alpha}(f)\right|.$$
We define the operator $P_N^{\mathcal A,k}(f)$, $N=(N_1,N_2)\in\mathbb Z^2$, $N_1<N_2$, and $k\in\mathbb N$, and $P_*^{\mathcal A,k}(f)$, $k\in\mathbb N$, as above but considering the Poisson semigroup $\{P_t^{\mathcal A}\}_{t>0}$ instead of $\{T_t^{\mathcal A}\}_{t>0}$.

\begin{thm}\label{th1.7}
Assume that $\{v_j\}_{j=-\infty}^\infty$ is a bounded sequence of complex numbers and that $(t_j)_{j=-\infty}^\infty$ is an increasing sequence in $(0,\infty)$ which is $\lambda$-lacunary with $\lambda >1$ (that is, $t_{j+1}\geq \lambda t_j, \;\;j\in\mathbb Z$).
\begin{enumerate}
    \item[(a)] The operator $T_*^{\mathcal A,\alpha}$ is bounded from $L^p(\mathbb R^n,\gamma_{-1})$ into itself, for every $1<p<\infty$ and $\alpha\geq 1$ or $\alpha=0$.
    \item[(b)] $T_*^{\mathcal A,\alpha}$ is bounded from $L^1(\mathbb R^n,\gamma_{-1})$ into $L^{1,\infty}(\mathbb R^n,\gamma_{-1})$, when $\alpha=1$ or $\alpha=0$.
    \item[(c)] $P_*^{\mathcal A,\alpha}$ is bounded from $L^p(\mathbb R^n,\gamma_{-1})$ into itself, for every $1<p<\infty$, and from $L^1(\mathbb R^n,\gamma_{-1})$ into $L^{1,\infty}(\mathbb R^n,\gamma_{-1})$, for every $\alpha\geq 0$.
\end{enumerate}
\end{thm}

Chao and Torrea (\cite[Theorem 1.1]{CT}) proved a version of the last theorem for the classical heat semigroup in $\mathbb R^n$ when $\alpha =0$. In order to prove Theorem \ref{th1.7} we need to extend \cite[Theorem 1.1]{CT} by considering maximal operators involving fractional derivatives for Euclidean heat and Poisson semigroups in $\mathbb R^n$.

In the next sections we prove the results we have just stated. Throughout the paper by $c$ and $C$ we always denote positive constants that can change in each occurrence.

\section{Proof of Theorems \ref{th1.1} and \ref{th1.2}}

\subsection{Proof of Theorem \ref{th1.1}}
Firstly we study the operators associated with $\{T_t^{\mathcal A}\}_{t>0}$. According to \cite[Corollary 4.5 and the comments in the last lines of the section 5.2]{LeX} (see also, \cite[Theorem 3.3]{JR1}) we have that, for every $1<p<\infty$ and $k\in\mathbb N$, the operators $V_\rho(\{t^k\partial_t^k T_t^{\mathcal A}\}_{t>0})$ and  $O(\{t^k\partial_t^k T_t^{\mathcal A}\}_{t>0},\{t_j\}_{j\in\mathbb N})$,  are bounded from $L^p(\mathbb R^n,\gamma_{-1})$ into itself, and the family $\{\lambda\Lambda(\{t^k\partial_t^k T_t^{\mathcal A}\}_{t>0},\lambda)^{1/\rho}\}_{t>0}$ is uniformly bounded from  $L^p(\mathbb R^n,\gamma_{-1})$ into itself.

We define, for every $\delta >0$, the set

$$N_{\delta}=\left\{(x,y)\in\mathbb R^n\times\mathbb R^n\;:\;|x-y|<n\delta\min\{1,\frac{1}{|x|}\}\right\}$$
$N_\delta$ is called a local region and the complement $N_\delta^c$ of $N_\delta$ is named a global region.

For every $t>0$, we define

$$T_{t,loc}^{\mathcal A}(f)(x)=T_t^{\mathcal A}(\chi_{N_\delta}(x,\cdot)f)(x),\;\;\;\;x\in\mathbb R^n,$$
and
$$T_{t,glob}^{\mathcal A}(f)(x)=T_t^{\mathcal A}(\chi_{N_\delta^c}(x,\cdot)f)(x),\;\;\;\;x\in\mathbb R^n.$$
In the sequel local and global operators will be defined always in the above way where $\delta$ will be fixed in each case.

Let $\beta >0$. We firstly study the global operators. Suppose that $0<t_1<...<t_k$, $k\in\mathbb N$, and $f\in L^p(\mathbb R^n,\gamma_{-1})$, $1\leq p<\infty$. We have that

\begin{align*}
& \left(  \sum_{j=1}^{k-1}  \left|t^\beta\partial_t^\beta T_t^{\mathcal A}(f)(x)_{|t=t_{j+1}}-t^\beta\partial_t^\beta T_t^{\mathcal A}(f)(x)_{|t=t_j}\right|^\rho \right)^{1/\rho} \\
 & =\left(\sum_{j=1}^{k-1}  \left|\int_{t_j}^{t_{j+1}}\partial_t\left[t^\beta\partial_t^\beta T_t^{\mathcal A}(f)(x)\right]dt\right|^\rho\right)^{1/\rho} \\
 & \leq \sum_{j=1}^{k-1}  \left|\int_{t_j}^{t_{j+1}}\partial_t\left[t^\beta\partial_t^\beta T_t^{\mathcal A}(f)(x)\right]dt\right| \\
  & \leq \int_0^\infty \left|\partial_t\left[t^\beta\partial_t^\beta T_t^{\mathcal A}(f)(x)\right]\right|dt
\end{align*}
Then,

\begin{align*}
V_\rho(\{t^\beta\partial_t^\beta  T_t^{\mathcal A}\}_{t>0}) & (f)(x) \leq \int_0^\infty \left|\partial_t\left[t^\beta\partial_t^\beta T_t^{\mathcal A}(f)(x)\right]\right|dt \\
 & \leq \int_{\mathbb R^n}|f(y)|K(x,y)dy,\;\;\;\;x\in\mathbb R^n,
\end{align*}
where

$$K(x,y)=\int_0^\infty \left|\partial_t\left[t^\beta\partial_t^\beta T_t^{\mathcal A}(x,y)\right]\right|dt,\;\;\;\;(x,y)\in\mathbb R^n.$$
We choose $m\in\mathbb N$ such that $m-1\leq \beta <m$. We get

\begin{align}\label{2.2}
\partial_t & \left[t^\beta\partial_t^\beta T_t^{\mathcal A}(x,y)\right]=\frac{1}{\Gamma(m-\beta)}\partial_t\left[t^\beta\int_0^\infty\partial_u^m T_u^{\mathcal A}(x,y)_{|u=t+s}s^{m-\beta -1}ds\right] \nonumber \\
 & =\frac{1}{\Gamma(m-\beta)}\left(\beta t^{\beta -1}\int_0^\infty\partial_u^m T_u^{\mathcal A}(x,y)_{|u=t+s}s^{m-\beta -1}ds\right. \nonumber \\
  & \left.+t^\beta\int_0^\infty\partial_u^{m+1} T_u^{\mathcal A}(x,y)_{|u=t+s}s^{m-\beta -1}ds\right),\;\;\;\;x,y\in\mathbb R^n\;\mbox{and}\;t>0.
\end{align}

Teuwen (\cite[Thorem 5.1]{Teu}) give an expression for the $k$-th derivative $\partial_t^k \mathcal H_t^{ou}(x,y)$ of the integral kernel $ \mathcal H_t^{ou}(x,y)$ of Ornstein-Uhlenbeck semigroup given by

$$ \mathcal H_t^{ou}(x,y)=\frac{e^{-|y-xe^{-t}|^2/(1-e^{-2t})}}{(1-e^{-2t})^{n/2}},\;\;\;\;x,y\in\mathbb R^n\;\mbox{and}\;t>0.$$
The Teuwen's formula is not correct. Some signs have to be corrected as it is presented in the following one

\begin{align}\label{2.1}
& \partial_t^k  \mathcal H_t^{ou}(x,y)=(-1)^k\mathcal H_t^{ou}(x,y)\sum_{|m|=k}\left(\begin{array}{c} k \\ m_1,...,m_k \end{array}\right) \prod_{i=1}^n\sum_{s_i=0}^{m_i}\sum_{\ell_i=0}^{s_i}2^{-s_i}\left\{\begin{array}{c} m_i \\ s_i \end{array}\right\} \left(\begin{array}{c} s_i \\ \ell_i \end{array}\right) \nonumber \\
& (-1)^{s_i-\ell_i}\left(\frac{e^{-t}}{\sqrt{1-e^{-2t}}}\right)^{2s_i-\ell_i}H_{\ell_i}(x_i)H_{2s_i-\ell_i}\left(\frac{x_ie^{-t}-y_i}{\sqrt{1-e^{-2t}}}\right),\;\;\;\;x,y\in\mathbb R^n,\;k\in\mathbb N\;\mbox{and}\;t>0.
\end{align}
Here, for every $N,\ell\in\mathbb N$, $N\geq\ell$, $\left\{\begin{array}{c} N \\ \ell \end{array}\right\}$ denotes the Stirling numbers of the second kind, that is, the number of partitions of an $N$-set into $\ell$ nonempty subsets.

By using (\ref{2.2}) we get

\begin{align*}
& \partial_t^k T_t^{\mathcal A}(x,y)  = \partial_t^k\left[e^{-nt}\frac{e^{-|x-ye^{-t}|^2/(1-e^{-2t})}}{(1-e^{-2t})^{n/2}}\right] =e^{|y|^2-|x|^2}\partial_t^k\left[e^{-nt}\frac{e^{-|y-xe^{-t}|^2/(1-e^{-2t})}}{(1-e^{-2t})^{n/2}}\right] \\
& =e^{|y|^2-|x|^2}\sum_{j=0}^k \left(\begin{array}{c} k \\j \end{array}\right)(-n)^{k-j}e^{-nt}\partial_t^j\left[\frac{e^{-|y-xe^{-t}|^2/(1-e^{-2t})}}{(1-e^{-2t})^{n/2}}\right] \\
& =e^{|y|^2-|x|^2}(-1)^k\sum_{j=0}^k \left(\begin{array}{c} k \\j \end{array}\right)(-n)^{k-j}e^{-nt}\frac{e^{-|y-xe^{-t}|^2/(1-e^{-2t})}}{(1-e^{-2t})^{n/2}}\sum_{|a|=j}\left(\begin{array}{c} j \\ a_1,...,a_j \end{array}\right) \prod_{i=1}^n\sum_{s_i=0}^{a_i}\sum_{\ell_i=0}^{s_i}2^{-s_i} \\
& \left\{\begin{array}{c} a_i \\ s_i \end{array}\right\} \left(\begin{array}{c} s_i \\ \ell_i \end{array}\right)(-1)^{s_i-\ell_i}\left(\frac{e^{-t}}{\sqrt{1-e^{-2t}}}\right)^{2s_i-\ell_i}H_{\ell_i}(x_i)H_{2s_i-\ell_i}\left(\frac{x_ie^{-t}-y_i}{\sqrt{1-e^{-2t}}}\right),\;\;\;\;x,y\in\mathbb R^n,\;k\in\mathbb N\;\mbox{and}\;t>0.
\end{align*}
We have that

\begin{align}\label{2.2.1}
\int_0^\infty & t^{\beta -1} \left|\int_0^\infty\partial_u^m T_u^{\mathcal A}(x,y)_{|u=t+s}s^{m-\beta -1}ds\right|dt\leq  \int_0^\infty t^{\beta -1} \int_t^\infty\left|\partial_u^m T_u^{\mathcal A}(x,y)\right|(u-t)^{m-\beta -1}dudt\nonumber \\
& \leq \int_0^\infty\left|\partial_u^m T_u^{\mathcal A}(x,y)\right|\int_0^ut^{\beta -1}(u-t)^{m-\beta -1}dtdu \nonumber \\
& \leq C\int_0^\infty u^{m-1}\left|\partial_u^m T_u^{\mathcal A}(x,y)\right|du,\;\;\;\;x,y\in\mathbb R^n,
\end{align}
and

\begin{align}\label{2.2.2}
\int_0^\infty & t^{\beta} \left|\int_0^\infty\partial_u^{m+1} T_u^{\mathcal A}(x,y)_{|u=t+s}s^{m-\beta -1}ds\right|dt  \leq C\int_0^\infty u^{m}\left|\partial_u^{m+1} T_u^{\mathcal A}(x,y)\right|du,\;\;\;\;x,y\in\mathbb R^n.
\end{align}

From (\ref{2.2}) we deduce that

$$K(x,y)\leq C\left(\int_0^\infty u^{m-1}\left|\partial_u^m T_u^{\mathcal A}(x,y)\right|du+\int_0^\infty u^{m}\left|\partial_u^{m+1} T_u^{\mathcal A}(x,y)\right|du\right),\;\;\;\;x,y\in\mathbb R^n.$$

For every $j=0,...,m$, we define

\begin{align*}
& H_j(t,x,y)  = t^{m-1}e^{-nt}e^{|y|^2-|x|^2}(-1)^k\frac{e^{-|y-xe^{-t}|^2/(1-e^{-2t})}}{(1-e^{-2t})^{n/2}}\sum_{|\ell|=j}\left(\begin{array}{c} j \\ \ell_1,...,\ell_j \end{array}\right) \prod_{i=1}^n\sum_{s_i=0}^{\ell_i}\sum_{d_i=0}^{s_i}2^{-s_i} \\
& \left\{\begin{array}{c} \ell_i \\ s_i \end{array}\right\} \left(\begin{array}{c} s_i \\ d_i \end{array}\right)(-1)^{s_i-d_i}\left(\frac{e^{-t}}{\sqrt{1-e^{-2t}}}\right)^{2s_i-d_i}H_{d_i}(x_i)H_{2s_i-d_i}\left(\frac{x_ie^{-t}-y_i}{\sqrt{1-e^{-2t}}}\right),\;\;\;\;x,y\in\mathbb R^n\;\mbox{and}\;t>0.
\end{align*}
Let $j=0,...,m$. We get

\begin{align*}
|H_j(t,x,y)| & \leq C t^{m-1}e^{-nt}e^{|y|^2-|x|^2}\frac{ e^{-\epsilon|y-xe^{-t}|^2/(1-e^{-2t})}}{(1-e^{-2t})^{n/2}}\sum_{|\ell|=j} \prod_{i=1}^n\sum_{s_i=0}^{\ell_i}\sum_{d_i=0}^{s_i}  \\
& \cdot \left(\frac{e^{-t}}{\sqrt{1-e^{-2t}}}\right)^{2s_i-d_i}|H_{d_i}(x_i)|,\;\;\;\;x,y\in\mathbb R^n\;\mbox{and}\;t>0.
\end{align*}
Here $0<\epsilon <1$ will be fixed later.

For every $\ell\in\mathbb N$, $|x|^\ell\leq C(|x+y|^\ell +|x-y|^\ell)$, $x,y\in\mathbb R^n$. Also, by making the change of variable $\displaystyle t=\log\left(\frac{1+s}{1-s}\right)\in(0,\infty)$ we obtain
$$\frac{|y-xe^{-t}|^2}{1-e^{-2t}}=\frac{1}{4s}|x(1-s)-y(1+s)|^2=\frac14(s|x+y|^2+\frac1s|x-y|^2)+\frac{|y|^2-|x|^2}{2},\;\;\;\;x,y\in\mathbb R^n.$$
Then, for every $\ell\in\mathbb N$,

\begin{align*}
|x|^{\ell} e^{-\epsilon |y-xe^{-t}|^2/(1-e^{-2t})}  & \leq C  (|x+y|^\ell +|x-y|^\ell)e^{-\frac{\epsilon}{4}(s|x+y|^2+\frac1s|x-y|^2)-\epsilon\frac{|y|^2-|x|^2}{2}} \\
& \leq C\frac{1}{s^{\ell/2}}e^{-\frac{\eta}{4}(s|x+y|^2+\frac1s|x-y|^2)-\epsilon\frac{|y|^2-|x|^2}{2}} \\
& \leq C \frac{1}{(1-e^{-2t})^{\ell/2}}e^{-\eta |y-xe^{-t}|^2/(1-e^{-2t})}e^{-\frac{\epsilon -\eta}{2}(|y|^2-|x|^2)},\;\;\;\;x,y\in\mathbb R^n\;\mbox{and}\;t>0.
\end{align*}
Here $\eta\in(0,\epsilon)$.

It follows that,

\begin{align*}
|H_j(t,x,y)| & \leq C t^{m-1}e^{-nt}e^{|y|^2-|x|^2}\frac{ e^{-\epsilon|y-xe^{-t}|^2/(1-e^{-2t})}}{(1-e^{-2t})^{n/2}}\sum_{|\ell|=j} \prod_{i=1}^n\sum_{s_i=0}^{\ell_i}\sum_{d_i=0}^{s_i}  \left(\frac{e^{-t}}{\sqrt{1-e^{-2t}}}\right)^{2s_i-d_i}(1+|x|)^{d_i} \\
& \leq C t^{m-1}e^{-nt}\frac{1}{(1-e^{-2t})^{n/2}}e^{-\eta |y-xe^{-t}|^2/(1-e^{-2t})}e^{(1-\frac{\epsilon -\eta}{2})(|y|^2-|x|^2)}\sum_{s=0}^j\left(\frac{e^{-2t}}{1-e^{-2t}}\right)^s \\
 & \leq C  t^{m}e^{-nt}\frac{e^{-\eta |y-xe^{-t}|^2/(1-e^{-2t})}}{(1-e^{-2t})^{n/2+1}}e^{(1-\frac{\epsilon -\eta}{2})(|y|^2-|x|^2)}\sum_{s=0}^j\left(\frac{e^{-2t}}{1-e^{-2t}}\right)^s, \;\;\;\;x,y\in\mathbb R^n\;\mbox{and}\;t>0.
\end{align*}
We conclude that

$$|t^{m-1}\partial_t^m T_t^{\mathcal A}(x,y)|\leq C  t^{m}e^{-nt}\frac{e^{-\eta |y-xe^{-t}|^2/(1-e^{-2t})}}{(1-e^{-2t})^{n/2+1}}e^{(1-\frac{\epsilon -\eta}{2})(|y|^2-|x|^2)}\sum_{s=0}^m\left(\frac{e^{-2t}}{1-e^{-2t}}\right)^s, \;\;x,y\in\mathbb R^n\;\mbox{and}\;t>0.$$
We also have that

$$|t^{m}\partial_t^{m+1} T_t^{\mathcal A}(x,y)|\leq C  t^{m}e^{-nt}\frac{e^{-\eta |y-xe^{-t}|^2/(1-e^{-2t})}}{(1-e^{-2t})^{n/2+1}}e^{(1-\frac{\epsilon -\eta}{2})(|y|^2-|x|^2)}\sum_{s=0}^m\left(\frac{e^{-2t}}{1-e^{-2t}}\right)^s, \;\;x,y\in\mathbb R^n\;\mbox{and}\;t>0.$$
Here $0<\eta <\epsilon <1$.

Then
$$K(x,y)\leq C e^{(1-\frac{\epsilon -\eta}{2})(|y|^2-|x|^2)}\int_0^\infty t^{m}e^{-nt}\frac{e^{-\eta |y-xe^{-t}|^2/(1-e^{-2t})}}{(1-e^{-2t})^{n/2+1}}\sum_{s=0}^m\left(\frac{e^{-2t}}{1-e^{-2t}}\right)^sdt, \;\;x,y\in\mathbb R^n.$$
We now choose $\delta=\frac{1}{\eta}$. If $(x,y)\notin N_\delta$, then $(\sqrt{\eta}x,\sqrt{\eta}y)\notin N_1$. In order to estimate $K(x,y)$ in the global region $N_\delta^c$ we use some ideas developed in \cite{PS}.

By making the change of variable $e^{-t}=\sqrt{1-s}$, $t\in (0,\infty)$, we obtain

\begin{align*}
\int_0^\infty t^{m} & e^{-nt}\frac{e^{-\eta |y-xe^{-t}|^2/(1-e^{-2t})}}{(1-e^{-2t})^{n/2+1}}\sum_{\ell=0}^m\left(\frac{e^{-2t}}{1-e^{-2t}}\right)^\ell dt \\
& \leq C \int_0^1 (1-s)^{n/2}|\log(1-s)|^m\frac{e^{-\eta |y-\sqrt{1-s}x|^2/s}}{s^{n/2+1}}\sum_{\ell=0}^m\left(\frac{1-s}{s}\right)^\ell\frac{ds}{1-s}.
\end{align*}
Suppose that $n\geq 2$ and $(x,y)\notin N_\delta$. We get

$$K(x,y)\leq C e^{(1-\frac{\epsilon -\eta}{2})(|y|^2-|x|^2)}\int_0^1 \frac{e^{-\eta |y-\sqrt{1-s}x|^2/s}}{s^{n/2+1}}\frac{ds}{\sqrt{1-s}}, \;\;x,y\in\mathbb R^n.$$
Inspired by \cite{PS} we denote, for every $x,y\in\mathbb R^n$, $a=\eta(|x|^2+|y|^2)$, $b=2\eta \langle x,y\rangle$, $s_0=\frac{2\sqrt{a^2-b^2}}{a+\sqrt{a^2-b^2}}$, $u_0=\eta\frac{|y|^2-|x|^2+|x+y||x-y|}{2}$.

Assume that $b>0$. Since $(\sqrt{\eta} x,\sqrt{\eta} y)\notin N_1$, we can write

\begin{align*}
\int_0^1 & \frac{e^{-\eta |y-\sqrt{1-s}x|^2/s}}{s^{n/2+1}}\frac{ds}{\sqrt{1-s}}= \int_0^1  \frac{e^{-|\sqrt{\eta}y-\sqrt{1-s}\sqrt{\eta}x|^2/s}}{s^{n/2+1}}\frac{ds}{\sqrt{1-s}} \\
& =\int_0^1 \left( \frac{e^{-|\sqrt{\eta}y-\sqrt{1-s}\sqrt{\eta}x|^2/s}}{s^{n/2}}\right)^{\frac{n-1}{n}}\frac{e^{-|\sqrt{\eta}y-\sqrt{1-s}\sqrt{\eta}x|^2/sn}}{s^{3/2}\sqrt{1-s}}ds \\
 & \leq C\left(\frac{e^{-u_0}}{s_0^{n/2}}\right)^{\frac{n-1}{n}}\int_0^1 \frac{e^{-|\sqrt{\eta}y-\sqrt{1-s}\sqrt{\eta}x|^2/sn}}{s^{3/2}\sqrt{1-s}}ds.
\end{align*}
By proceeding as in the proof of \cite[Lemma 2.3]{PS} since $a^2-b^2>n$ we obtain

$$\int_0^1 \frac{e^{-|\sqrt{\eta}y-\sqrt{1-s}\sqrt{\eta}x|^2/sn}}{s^{3/2}\sqrt{1-s}}ds\leq C\frac{e^{-u_0/n}}{s_0^{1/2}}.$$
It follows that

$$\int_0^1 \frac{e^{-\eta |y-\sqrt{1-s}x|^2/s}}{s^{n/2+1}}\frac{ds}{\sqrt{1-s}}= \leq C\frac{e^{-u_0}}{s_0^{n/2}} $$
According to \cite[Proposition 2.1]{MPS} we get

\begin{align}\label{2.25}
K(x,y) &\leq C e^{(1-\frac{\epsilon -\eta}{2})(|y|^2-|x|^2)}\left(\frac{|x+y|}{|x-y|}\right)^{n/2}e^{-\frac{\eta}{2}(|y|^2-|x|^2)-\frac{\eta}{2}|x+y||x-y|} \nonumber \\
 & \leq C |x+y|^n e^{(1-\frac{\epsilon }{2})(|y|^2-|x|^2)-\frac{\eta}{2}|x+y||x-y|},
\end{align}
because $|x+y||x-y|\geq n\delta$.

Suppose now that $n=1$. We decompose the integral over $(0,\infty)$ in two integrals, one of them over $(0,1)$ and the other one over $(1,\infty)$. By proceeding as above we can  obtain that

\begin{align*}
\int_0^1 t^{m} & e^{-nt}\frac{e^{-\eta |y-xe^{-t}|^2/(1-e^{-2t})}}{(1-e^{-2t})^{n/2+1}}\sum_{\ell=0}^m\left(\frac{e^{-2t}}{1-e^{-2t}}\right)^\ell dt \;e^{(1-\frac{\epsilon -\eta}{2})(|y|^2-|x|^2)} \\
& \leq C |x+y|^n e^{(1-\frac{\epsilon }{2})(|y|^2-|x|^2)-\frac{\eta}{2}|x+y||x-y|}.
\end{align*}
On the other hand, by using \cite[Proposition 2.1]{MPS}, we get

\begin{align*}
\int_1^\infty t^{m} & e^{-nt}\frac{e^{-\eta |y-xe^{-t}|^2/(1-e^{-2t})}}{(1-e^{-2t})^{n/2+1}}\sum_{\ell=0}^m\left(\frac{e^{-2t}}{1-e^{-2t}}\right)^\ell dt\; e^{(1-\frac{\epsilon -\eta}{2})(|y|^2-|x|^2)} \\
 & \leq C \int_1^\infty t^{m}  e^{-nt}dt\;\sup_{t>0}\left[\frac{e^{-\eta |y-xe^{-t}|^2/(1-e^{-2t})}}{(1-e^{-2t})^{n/2+1}}\right]e^{(1-\frac{\epsilon -\eta}{2})(|y|^2-|x|^2)} \\
& \leq C |x+y|^n e^{(1-\frac{\epsilon }{2})(|y|^2-|x|^2)-\frac{\eta}{2}|x+y||x-y|}.
\end{align*}

We conclude that
$$K(x,y)\leq C |x+y|^n e^{(1-\frac{\epsilon }{2})(|y|^2-|x|^2)-\frac{\eta}{2}|x+y||x-y|}.$$

Assume that $(x,y)\notin N_\delta$ and $b\leq 0$. We make the change of variable $z= a(\frac1s -1)$ and by taking into account that $\frac{|y-x\sqrt{1-s}|^2}{s}\geq\frac{a}{s} -|x|^2$, $s\in (0,1)$, it follows that

\begin{align*}
\int_0^\infty t^{m} & e^{-nt}\frac{e^{-\eta |y-xe^{-t}|^2/(1-e^{-2t})}}{(1-e^{-2t})^{n/2+1}}\sum_{\ell=0}^m\left(\frac{e^{-2t}}{1-e^{-2t}}\right)^\ell dt\; e^{(1-\frac{\epsilon -\eta}{2})(|y|^2-|x|^2)} \\
 & \leq C \left(\int_0^1 \frac{e^{-\eta |y-x\sqrt{1-s}|^2/s}}{s^{n/2+1}}ds\;+\;\sup_{t>0}\left[\frac{e^{-\eta |y-xe^{-t}|^2/(1-e^{-2t})}}{(1-e^{-2t})^{n/2}}\right]\right)e^{(1-\frac{\epsilon -\eta}{2})(|y|^2-|x|^2)} \\
  & \leq Ce^{-\eta |y|^2}e^{(1-\frac{\epsilon -\eta}{2})(|y|^2-|x|^2)}\left(1+\int_0^\infty e^{-z}z^{n/2 -1}dz\right) \\
& \leq C e^{|y|^2(1-\frac{\epsilon +\eta }{2})-|x|^2(1-\frac{\epsilon -\eta}{2})},
\end{align*}
because $a>1/2$.

We obtain

$$K(x,y)\leq C e^{|y|^2(1-\frac{\epsilon +\eta }{2})-|x|^2(1-\frac{\epsilon -\eta}{2})}.$$

We have proved that, for every $(x,y)\notin N_\delta$,

$$K(x,y)\leq C \left\{\begin{array}{l}
e^{|y|^2(1-\frac{\epsilon +\eta }{2})-|x|^2(1-\frac{\epsilon -\eta}{2})},\;\;\;\langle x,y\rangle\leq 0 \\
 \\
|x+y|^n e^{(1-\frac{\epsilon }{2})(|y|^2-|x|^2)-\frac{\eta}{2}|x+y||x-y|},\;\;\;\langle x,y\rangle\geq 0
\end{array}
\right.$$

Let $1<p<\infty $. We choose $0<\eta <\epsilon <1$ such that $\epsilon -\eta   <2(1-1/p)<\epsilon +\eta $. We define the operator $\mathbb L_p$ by

$$\mathbb L_p(f)(x)=\int_{\mathbb R^n}K(x,y)\chi_{N_\delta^c}(x,y)f(y)dy.$$
Since $||y|^2-|x|^2|\leq|x+y||x-y|$, we have that

\begin{align*}
\int_{\mathbb R^n} & e^{\frac{|x|^2-|y|^2}{p}}K(x,y)\chi_{N_\delta^c}(x,y)dy\leq \int_{\mathbb R^n}e^{|y|^2(1-\frac{\epsilon +\eta }{2}-\frac{1}{p})}dy\;e^{-|x|^2(1-\frac{\epsilon -\eta}{2}-\frac{1}{p})} \\
 & +\int_{\mathbb R^n}|x+y|^n e^{-|x+y||x-y|(\frac{\eta}{2}-|1-\frac{\epsilon }{2}-\frac{1}{p}|)}dy\leq C,\;\;\;\;\;x\in\mathbb R^n.
\end{align*}

We also get

$$\sup_{y\in\mathbb R^n}\int_{\mathbb R^n} e^{\frac{|x|^2-|y|^2}{p}}K(x,y)\chi_{N_\delta^c}(x,y)dx<\infty.$$
It follows that the operator $\mathcal L_p$ defined by

$$\mathcal L_p(f)(x)=\int_{\mathbb R^n} e^{\frac{|x|^2-|y|^2}{p}}K(x,y)\chi_{N_\delta^c}(x,y)f(y)dy,$$
is bounded from $L^q(\mathbb R^n,dx)$ into itself for every $1<q<\infty$. In particular, $\mathcal L_p$ is bounded from $L^p(\mathbb R^n,dx)$ into itself, that is, $\mathbb L_p$ is bounded from $L^p(\mathbb R^n,\gamma_{-1})$ into itself.

We conclude that $V_{\rho,glob}(\{t^\beta\partial_t^\beta T_t^{\mathcal A}\}_{t>0})$ is bounded from $L^p(\mathbb R^n,\gamma_{-1})$ into itself.

When $\beta=0$ we can proceed in a similar way.

We now study the operator $V_{\rho,loc}(\{t^\beta\partial_t^\beta T_t^{\mathcal A}\}_{t>0})$ defined by $N_\delta$. Here we do not need to impose condition to $\delta$.

We have that

\begin{align*}
V_{\rho,loc}(\{t^\beta\partial_t^\beta T_t^{\mathcal A}\}_{t>0})(f) & =V_{\rho,loc}(\{t^\beta\partial_t^\beta T_t^{\mathcal A}\}_{t>0})(f)-V_{\rho,loc}(\{t^\beta\partial_t^\beta T_t\}_{t>0})(f)+V_{\rho,loc}(\{t^\beta\partial_t^\beta T_t\}_{t>0})(f) \\
 & \leq V_{\rho,loc}(\{t^\beta\partial_t^\beta (T_t^{\mathcal A}-T_t)\}_{t>0})(f)+V_{\rho,loc}(\{t^\beta\partial_t^\beta T_t\}_{t>0})(f).
\end{align*}
Here $\{T_t\}_{t>0}$ denotes the Euclidean heat semigroup.

We begin with the study of the operator $V_{\rho,loc}(\{t^\beta\partial_t^\beta T_t\}_{t>0})$. The heat kernel in $\mathbb R^n$ is given by

$$T_t(z)=\frac{1}{(2\pi t)^{n/2}}e^{-|z|^2/2t},\;\;\;\;z\in\mathbb R^n\;\mbox{and}\;t>0.$$
Let $k\in\mathbb N$. We have that, for certain $a_0.a_1,...,a_k\in\mathbb R$,

$$t^k\partial_t^k T_t(z)=T_t(z)(a_0+a_1\frac{|z|^2}{2t}+....+\frac{|z|^{2k}}{(2t)^k}),\;\;\;\;z\in\mathbb R^n\;\mbox{and}\;t>0.$$
This equality can be proved by induction. We define

$$\Phi_k(u)=e^{-u^2}(a_0+a_1u^2+....+a_ku^{2k}),\;\;\;\;u\in\mathbb R.$$

Suppose that $m\in\mathbb N$ and $m-1\leq \beta <m$, $\beta >0$. We can write

\begin{align}\label{2.3}
t^\beta\partial_t^\beta T_t(z) & =\frac{t^\beta}{\Gamma(m-\beta)}\int_t^\infty \partial_u^m T_u(z)(u-t)^{m-\beta -1}du \nonumber\\
 & = \frac{t^\beta}{(2\pi)^{n/2}\Gamma(m-\beta)}\int_t^\infty \frac{1}{u^{n/2}}\Phi_m\left(\frac{|z|}{\sqrt{2u}}\right)u^{-m}(u-t)^{m-\beta -1}du \nonumber\\
 & = \frac{t^{-n/2}}{(2\pi)^{n/2}\Gamma(m-\beta)}\int_1^\infty \Phi_m\left(\frac{|z|}{\sqrt{2tv}}\right)v^{-m-n/2}(v-1)^{m-\beta -1}dv \nonumber \\
 & = \frac{1}{t^{n/2}}\Psi_\beta\left(\frac{|z|}{\sqrt{t}}\right),\;\;\;\;z\in\mathbb R^n\;\mbox{and}\;t>0,
\end{align}
where

$$\Psi_\beta(u)=\frac{1}{(2\pi)^{n/2}\Gamma(m-\beta)}\int_1^\infty \Phi_m\left(\frac{u}{\sqrt{2v}}\right)v^{-m-n/2}(v-1)^{m-\beta -1}dv,\;\;\;\;u>0.$$
We have that $\displaystyle\lim_{u\rightarrow\infty}\Psi_\beta(u)=0$ and

\begin{align*}
\int_0^\infty & |\Psi'_\beta(u)|u^n du=\int_0^\infty \frac{u^n}{\sqrt{2}} \frac{1}{(2\pi)^{n/2}\Gamma(m-\beta)}\int_1^\infty \left|\Phi'_m\left(\frac{u}{\sqrt{2v}}\right)\right|\frac{(v-1)^{m-\beta -1}}{v^{m+n/2+1/2}}dvdu \\
 & = \frac{2^{n/2}}{(2\pi)^{n/2}\Gamma(m-\beta)}\int_1^\infty\frac{(v-1)^{m-\beta -1}}{v^m}\int_0^\infty|\Phi'_m(w)|w^n dwdv<\infty.
\end{align*}
We define $\Psi_{\beta ,t}(u)=\frac{1}{t^n}\Psi_\beta\left(\frac{u}{t}\right)$, $u,t>0$. We can write $\Psi_{\beta.\sqrt{t}}(|z|)=t^\beta\partial_t^\beta T_t(z)$, $z\in\mathbb R^n$ and $t>0$. We consider the operator

\begin{align}\label{2.15}
\varphi_{\beta,t}(f)=\Psi_{\beta,t}*f,\;\;\;\;t>0.
\end{align}
It follows that $V_{\rho}(\{t^\beta\partial_t^\beta T_t\}_{t>0})(f)=V_{\rho}(\{\varphi_{\beta,\sqrt{t}}\}_{t>0})(f)$, $\mathcal O(\{t^\beta\partial_t^\beta T_t\}_{t>0},\{t_j\}_{j=1}^\infty)(f)=\mathcal O(\{\varphi_{\beta,\sqrt{t}}\}_{t>0},\{t_j\}_{j=1}^\infty)(f)$, and, there exists $C>0$ such that

$$\frac1C S_V(\{t^\beta\partial_t^\beta T_t\}_{t>0})(f)\leq S_V(\{\varphi_{\beta,\sqrt{t}}\}_{t>0})(f)\leq CS_V(\{t^\beta\partial_t^\beta T_t\}_{t>0})(f).$$
According to \cite[Lemma 2.4]{CJRW1} the operators $V_{\rho}(\{t^\beta\partial_t^\beta T_t\}_{t>0})$, $\mathcal O(\{t^\beta\partial_t^\beta T_t\}_{t>0},\{t_j\}_{j=1}^\infty)$, $S_V(\{t^\beta\partial_t^\beta T_t\}_{t>0})$ and by (\ref{eq1.1}), $\lambda\Lambda(\{t^\beta\partial_t^\beta T_t\}_{t>0},\lambda)^{1/\rho}$ are bounded from $L^p(\mathbb R^n,dx)$ into itself, and from $L^1(\mathbb R^n,dx)$ into $L^{1,\infty}(\mathbb R^n,dx)$.

We consider the Banach space $F_\rho$ defined as follows, A function $g:(0,\infty)\rightarrow\mathbb C$ is in $F_\rho$ provided that

\begin{align}\label{2.10}
\|g\|_\rho:=\sup_{0<t_1<...<t_k,\;k\in\mathbb N}\left(\sum_{j=1}^{k-1}|g(t_{j+1})-g(t_j)|^\rho\right)^{1/\rho}<\infty
\end{align}
It is clear that $V_{\rho}(\{t^\beta\partial_t^\beta T_t\}_{t>0})(f)(x)=\|t^\beta\partial_t^\beta T_t(f)(x)\|_\rho$. We have that

\begin{align*}
\|t^\beta\partial_t^\beta T_t(x,y)\|_{F_\rho} & =\sup_{0<t_1<...<t_k,\;k\in\mathbb N}\left(\sum_{j=1}^{k-1}|t^\beta\partial_t^\beta T_t(x,y)_{|t=t_{j+1}}-t^\beta\partial_t^\beta T_t(x,y)_{|t=t_j}|^\rho\right)^{1/\rho} \\
 & \leq \sup_{0<t_1<...<t_k,\;k\in\mathbb N}\sum_{j=1}^{k-1}\int_{t_j}^{t_{j+1}}|\partial_t(t^\beta\partial_t^\beta T_t(x,y))|dt \\
& \leq \int_0^\infty|\partial_t(t^\beta\partial_t^\beta T_t(x,y))|dt,\;\;\;\;x,y\in\mathbb R^n.
\end{align*}
By (\ref{2.3}) we deduce that

\begin{align*}
\|t^\beta\partial_t^\beta T_t(x,y)\|_{F_\rho} & \leq \int_0^\infty\left|\partial_t\left[\frac{1}{t^{n/2}}\Psi_\beta\left(\frac{|x-y|)}{\sqrt{t}}\right)\right]\right|dt\\
 & \leq C \int_0^\infty\frac{1}{t^{n/2+1}}\left(\left|\Psi_\beta\left(\frac{|x-y|)}{\sqrt{t}}\right)\right|+\left|\Psi'_\beta\left(\frac{|x-y|)}{\sqrt{t}}\right)\right|\frac{|x-y|}{\sqrt{t}}\right)dt\\
& \leq C\int_0^\infty\left(\frac{u}{|x-y|}\right)^{n+2}\left(|\Psi_\beta(u)|+|\Psi'_\beta(u)|u\right)\frac{|x-y|^2}{u^3}du \\
& \leq \frac{C}{|x-y|^n},\;\;\;\;x,y\in\mathbb R^n,\;x\neq y.
\end{align*}
The arguments developed in the proof of \cite[Proposition 3.2.7]{Sa} (see also \cite[p. 10]{HTV}) allow us to conclude that $V_{\rho,loc}(\{t^\beta\partial_t^\beta T_t\}_{t>0})$ is bounded from $L^p(\mathbb R^n,dx)$ into itself, for every $1<p<\infty$, and from $L^1(\mathbb R^n,dx)$ into $L^{1,\infty}(\mathbb R^n,dx)$. By proposition \cite[Proposition 3.2.5]{Sa}, $V_{\rho,loc}(\{t^\beta\partial_t^\beta T_t\}_{t>0})$ is bounded from $L^p(\mathbb R^n,\gamma_{-1})$ into itself, for every $1<p<\infty$, and from $L^1(\mathbb R^n,\gamma_{-1})$ into $L^{1,\infty}(\mathbb R^n,\gamma_{-1})$.

When $\beta=0$ we can proceed in a similar way.

We now dedicate our attention to the operator $V_{\rho,loc}(\{t^\beta\partial_t^\beta (T_t^{\mathcal A}-T_t)\}_{t>0})$. Suppose that $\beta >0$, $m\in\mathbb N$ and $m-1\leq\beta <m$. When $\beta =0$ we can proceed similarly. We have that

\begin{align*}
 & V_{\rho,loc}(\{t^\beta\partial_t^\beta (T_t^{\mathcal A}-T_t)\}_{t>0})(f)(x)\leq \int_0^\infty |\partial_t(t^\beta\partial_t^\beta (T_t^{\mathcal A}-T_t)(f)(x)|dt  \\
 & \leq \int_{\mathbb R^n}|f(y)|\int_0^\infty |\partial_t(t^\beta\partial_t^\beta (T_t^{\mathcal A}(x,y)-T_t(x,y))|dtdy,\;\;\;\;x\in\mathbb R^n.
\end{align*}
We define

$$H(x,y)=\int_0^\infty |\partial_t(t^\beta\partial_t^\beta (T_t^{\mathcal A}(x,y)-T_t(x,y))|dt,\;\;\;\;x,y\in\mathbb R^n.$$
Let $\delta>0$. We consider the following operator

$$\mathcal H(f)(x)=\int_{\mathbb R^n}f(y)H(x,y)\chi_{N_\delta}(x,y)dy.$$
Our objective is to see that $\mathcal H$ is bounded from $L^p(\mathbb R^n,\gamma_{-1})$ into itself, for every $1<p<\infty$, and from $L^1(\mathbb R^n,\gamma_{-1})$ into $L^{1,\infty}(\mathbb R^n,\gamma_{-1})$.

We can write

\begin{align*}
  \partial_t(t^\beta & \partial_t^\beta (T_t^{\mathcal A}(x,y)-T_t(x,y))=\frac{\beta t^{\beta -1}}{\Gamma(m-\beta)} \int_0^\infty \partial_u^m(T_u^{\mathcal A}(x,y)-T_u(x,y))_{|u=t+s}s^{m-\beta -1}ds  \\
 & +\frac{t^\beta}{\Gamma(m-\beta)} \int_0^\infty \partial_u^{m+1}(T_u^{\mathcal A}(x,y)-T_u(x,y))_{|u=t+s}s^{m-\beta -1}ds,\;\;\;\;x,y\in\mathbb R^n\;\mbox{and}\;t>0.
\end{align*}
Then,

\begin{align*}
  & \int_0^\infty|\partial_t(t^\beta  \partial_t^\beta (T_t^{\mathcal A}(x,y) -T_t(x,y))|dt  \\
 & \leq C\left( \int_0^\infty u^{m-1}|\partial_u^{m}(T_u^{\mathcal A}(x,y)-T_u(x,y))|du + \int_0^\infty u^m |\partial_u^{m+1}(T_u^{\mathcal A}(x,y)-T_u(x,y))|du\right),\;\;\;\;x,y\in\mathbb R^n.
\end{align*}
Let $\ell\in\mathbb N$. We have that

\begin{align}\label{2.4}
\partial_t^\ell\left[\frac{e^{-u^2/2t}}{(2t)^{1/2}}\right]=\frac{1}{2^\ell}H_{2\ell}\left(\frac{u}{\sqrt{2t}}\right)\frac{e^{-u^2/2t}}{(2t)^{1/2+\ell}},\;\;\;\;u\in\mathbb R\;\mbox{and}\;t>0.
\end{align}
In order to prove (\ref{2.4}) we can proceed as follows. By recalling the definition of the Hermite polynomials we obtain

\begin{align*}
   &  \partial_t^\ell\left[\frac{e^{-u^2/2t}}{(2t)^{1/2}}\right]=\frac{1}{2^\ell}\partial_u^{2\ell}\left[\frac{e^{-u^2/2t}}{(2t)^{1/2}}\right]= \frac{1}{2^\ell}\frac{e^{-u^2/2t}}{(2t)^{1/2+\ell}}e^{z^2}\frac{d^{2\ell}}{dz^{2\ell}}e^{-z^2}_{|z=\frac{u}{\sqrt{2t}}} =\frac{1}{2^\ell}H_{2\ell}\left(\frac{u}{\sqrt{2t}}\right)\frac{e^{-u^2/2t}}{(2t)^{1/2+\ell}},\;\;\;\;u\in\mathbb R\;\mbox{and}\;t>0.
\end{align*}
Hence,

\begin{align*}
   &  \partial_t^\ell\left[\frac{e^{-|x-y|^2/2t}}{(2t)^{n/2}}\right]=\frac{e^{-|x-y|^2/2t}}{(2t)^{n/2}}\sum_{|k|=\ell}\left(\begin{array}{c} \ell \\
   k_1,...,k_n\end{array}\right)\prod_{i=1}^n\frac{1}{2^{k_i}}\frac{1}{(2t)^{k_i}}H_{2k_i}\left(\frac{x_i-y_i}{\sqrt{2t}}\right),\;\;\;\;x,y\in\mathbb R^n\;\mbox{and}\;t>0.
\end{align*}
By using (\ref{2.1}) we can write

\begin{align*}
  & \partial_t^{\ell}(T_t^{\mathcal A}(x,y)-T_t(x,y))=\partial_t^\ell\left[e^{-nt}\frac{e^{-|x-ye^{-t}|^2/(1-e^{-2t})}}{(1-e^{-2t})^{n/2}}-\frac{e^{-|x-y|^2/2t}}{(2t)^{n/2}}\right] \\
   & =\sum_{j=0}^{\ell -1}\left(\begin{array}{c} \ell \\ j \end{array}\right)(-n)^{\ell -j}e^{-nt} \partial_t^j\left[\frac{e^{-|x-ye^{-t}|^2/(1-e^{-2t})}}{(1-e^{-2t})^{n/2}}\right] + e^{-nt}\partial_t^{\ell}\left[\frac{e^{-|x-ye^{-t}|^2/(1-e^{-2t})}}{(1-e^{-2t})^{n/2}}\right]-\partial_t^{\ell}\left[\frac{e^{-|x-y|^2/2t}}{(2t)^{n/2}}\right] \\
    & =\sum_{j=0}^{\ell -1}\left(\begin{array}{c} \ell \\ j \end{array}\right)(-n)^{\ell -j}e^{-nt} \partial_t^j\left[\frac{e^{-|x-ye^{-t}|^2/(1-e^{-2t})}}{(1-e^{-2t})^{n/2}}\right] \\
     & \\
    & + e^{-nt} \frac{e^{-|x-ye^{-t}|^2/(1-e^{-2t})}}{(1-e^{-2t})^{n/2}}\sum_{|k|=\ell}\left(\begin{array}{c} \ell \\ k_1,...,k_n \end{array}\right) \prod_{i=1}^n\sum_{\begin{array}{c} m_i=0 \\ m_i\neq k_i\;\mbox{ for some } i=1,..,n \end{array}}^{k_i}2^{-m_i} \left\{\begin{array}{c} k_i \\ m_i \end{array}\right\} (-1)^{m_i}\\
& \cdot \left(\frac{e^{-t}}{\sqrt{1-e^{-2t}}}\right)^{2m_i}H_{2m_i}\left(\frac{y_ie^{-t}-x_i}{\sqrt{1-e^{-2t}}}\right) \\
 & \\
& + e^{-nt} \frac{e^{-|x-ye^{-t}|^2/(1-e^{-2t})}}{(1-e^{-2t})^{n/2}}\sum_{|k|=\ell}\left(\begin{array}{c} \ell \\ k_1,...,k_n \end{array}\right) \prod_{i=1}^n\sum_{m_i=0}^{k_i}\sum_{\begin{array}{c} d_i=0 \\ d_i\neq 0\;\mbox{ for some } i=1,..,n \end{array}}^{m_i}2^{-m_i} \left\{\begin{array}{c} k_i \\ m_i \end{array}\right\} \left(\begin{array}{c} m_i \\ d_i \end{array}\right)(-1)^{m_i-d_i}\\
& \cdot\left(\frac{e^{-t}}{\sqrt{1-e^{-2t}}}\right)^{2m_i-d_i}H_{d_i}(y_i)H_{2m_i-d_i}\left(\frac{y_ie^{-t}-x_i}{\sqrt{1-e^{-2t}}}\right) \\
 & \\
& + (e^{-nt}-1) \frac{e^{-|x-ye^{-t}|^2/(1-e^{-2t})}}{(1-e^{-2t})^{n/2}}\sum_{|k|=\ell}\left(\begin{array}{c} \ell \\ k_1,...,k_n \end{array}\right) \prod_{i=1}^n 2^{-k_i}\left(\frac{e^{-t}}{\sqrt{1-e^{-2t}}}\right)^{2k_i}H_{2k_i}\left(\frac{y_ie^{-t}-x_i}{\sqrt{1-e^{-2t}}}\right) \\
& +  \frac{e^{-|x-ye^{-t}|^2/(1-e^{-2t})}}{(1-e^{-2t})^{n/2}}\sum_{|k|=\ell}\left(\begin{array}{c} \ell \\ k_1,...,k_n \end{array}\right) \prod_{i=1}^n 2^{-k_i}\left(\frac{e^{-t}}{\sqrt{1-e^{-2t}}}\right)^{2k_i}H_{2k_i}\left(\frac{y_ie^{-t}-x_i}{\sqrt{1-e^{-2t}}}\right) \\
& - \frac{e^{-|x-y|^2/2t}}{(2t)^{n/2}}\sum_{|k|=\ell}\left(\begin{array}{c} \ell \\
   k_1,...,k_n\end{array}\right)\prod_{i=1}^n\frac{1}{2^{k_i}}\frac{1}{(2t)^{k_i}}H_{2k_i}\left(\frac{x_i-y_i}{\sqrt{2t}}\right),\;\;\;\;x,y\in\mathbb R^n\;\mbox{and}\;t>0.
\end{align*}

We have to do the estimations step by step. Assume that $(x,y)\in N_\delta$.
\newline

{\bf (a)} Let ${j=0,...,\ell -1}$. We can write

\begin{align*}
& \partial_t^j\left[\frac{e^{-|x-ye^{-t}|^2/(1-e^{-2t})}}{(1-e^{-2t})^{n/2}}\right]= \frac{e^{-|x-ye^{-t}|^2/(1-e^{-2t})}}{(1-e^{-2t})^{n/2}}\sum_{|k|=j}\left(\begin{array}{c} j \\ k_1,...,k_n \end{array}\right) \prod_{i=1}^n\sum_{m_i=0}^{k_i}\sum_{\ell_i=0}^{m_i}2^{-m_i} \left\{\begin{array}{c} k_i \\ m_i \end{array}\right\} (-1)^{m_i-\ell_i} \\
& \left(\begin{array}{c} m_i \\ \ell_i \end{array}\right)\left(\frac{e^{-t}}{\sqrt{1-e^{-2t}}}\right)^{2m_i-\ell_i}H_{\ell_i}(y_i)H_{2m_i-\ell_i}\left(\frac{y_ie^{-t}-x_i}{\sqrt{1-e^{-2t}}}\right),\;\;\;\;x,y\in\mathbb R^n,\;\mbox{and}\;t>0.
\end{align*}

We have to estimate
$$\int_0^\infty t^{\ell -1}e^{-nt}\frac{e^{-|x-ye^{-t}|^2/(1-e^{-2t})}}{(1-e^{-2t})^{n/2}}\left(\frac{e^{-t}}{\sqrt{1-e^{-2t}}}\right)^{2m-k}(1+|y|^k)\left(1+\left(\frac{|x-ye^{-t}|}{\sqrt{1-e^{-2t}}}\right)^{2m-k}\right)dt ,\;\;\;\;x,y\in\mathbb R^n.$$
where $m=0,...,j$ and $k=0,...,m$.

Let $m=0,...,j$ and $k=0,...,m$. By making the change of variables $\displaystyle t=\log\frac{1+s}{1-s}\in (0,\infty)$, we get

\begin{align*}
& \int_0^\infty t^{\ell -1}e^{-nt}\frac{e^{-|x-ye^{-t}|^2/(1-e^{-2t})}}{(1-e^{-2t})^{n/2}}\left(\frac{e^{-t}}{\sqrt{1-e^{-2t}}}\right)^{2m-k}(1+|y|^k)\left(1+\left(\frac{|x-ye^{-t}|}{\sqrt{1-e^{-2t}}}\right)^{2m-k}\right)dt \\
& \leq C\int_0^1 \left(\log\frac{1+s}{1-s}\right)^{\ell -1}(1-s)^{n-1}\frac{e^{-\frac{1}{8s}|(1+s)x-(1-s)y|^2}}{s^{n/2}}\left(\frac{1-s}{\sqrt{s}}\right)^{2m-k}(1+|y|)^kds,\;\;\;\;x,y\in\mathbb R^n.
\end{align*}
Since $\frac1s |(1+s)x-(1-s)y|^2=s|x+y|^2+\frac1s |x-y|^2+2(|x|^2-|y|^2)$ and $|y|\leq \frac12 (|x+y|+|x-y|)$, $x,y\in\mathbb R^n$ and $s\in(0,1)$, it follows that

$$e^{-\frac{1}{8s}|(1+s)x-(1-s)y|^2}(1+|y|)^k\leq C e^{-\frac{1}{4}(|x|^2-|y|^2)} e^{-c(s|x+y|^2+\frac1s |x-y|^2)}s^{-k/2},\;\;\;\;x,y\in\mathbb R^n\;\mbox{and}\;s\in(0,1).$$
Furthermore, since $(x,y)\in N_\delta$, $|x|^2-|y|^2\geq -C$ (see [line under (3.3.4)]{Sa}). We obtain

\begin{align*}
& \int_0^\infty t^{\ell -1}e^{-nt}\frac{e^{-|x-ye^{-t}|^2/(1-e^{-2t})}}{(1-e^{-2t})^{n/2}}\left(\frac{e^{-t}}{\sqrt{1-e^{-2t}}}\right)^{2m-k}(1+|y|^k)\left(1+\left(\frac{|x-ye^{-t}|}{\sqrt{1-e^{-2t}}}\right)^{2m-k}\right)dt \\
& \leq C\int_0^1 \left(\log\frac{1+s}{1-s}\right)^{\ell -1}\frac{(1-s)^{n-1+2m-k}}{s^{n/2+m}}e^{-c\frac{|x-y|^2}{s}}ds \\
 & \leq C\int_0^1\frac{1}{s^{n/2}}e^{-c\frac{|x-y|^2}{s}}ds\leq \frac{C}{|x-y|^{n-1}}.
\end{align*}
when $m=0=k$ and $n=1$ we split the integral in two, the first one from, $0$ to $1/2$, and the second one, from $1/2$ to $1$. We take into account that there exists $C>0$ such that,  $(\log(1-s))^{\ell -1}(1-s)^{d}\leq C$, with $0<d<1$ and $1/2<s<1$. \newline

{\bf (b)} Let $k=(k_1,...,k_n)\in\mathbb N^n$ such that $k_1+k_2+...+k_n=\ell$. We get

\begin{align*}
& e^{-nt} \frac{e^{-|x-ye^{-t}|^2/(1-e^{-2t})}}{(1-e^{-2t})^{n/2}} \prod_{i=1}^n\sum_{\begin{array}{c} m_i=0 \\ m_i\neq k_i\;\mbox{ for some } i=1,..,n \end{array}}^{k_i}\left(\frac{e^{-t}}{\sqrt{1-e^{-2t}}}\right)^{2m_i}\left|H_{2m_i}\left(\frac{y_ie^{-t}-x_i}{\sqrt{1-e^{-2t}}}\right)\right| \\
 & \leq C e^{-nt} \frac{e^{-|x-ye^{-t}|^2/(1-e^{-2t})}}{(1-e^{-2t})^{n/2}} \left(1+\left(\frac{e^{-t}}{\sqrt{1-e^{-2t}}}\right)^{2(\ell -1)}\right)\left(1+\left|\frac{x-ye^{-t}}{\sqrt{1-e^{-2t}}}\right|^{2(\ell -1)}\right),\;\;\;\;t>0.
\end{align*}
By proceeding as above we get

\begin{align*}
& \int_0^\infty t^{\ell -1}e^{-nt} \frac{e^{-|x-ye^{-t}|^2/(1-e^{-2t})}}{(1-e^{-2t})^{n/2}} \prod_{i=1}^n\sum_{\begin{array}{c} m_i=0 \\ m_i\neq k_i\;\mbox{ for some } i=1,..,n \end{array}}^{k_i}\left(\frac{e^{-t}}{\sqrt{1-e^{-2t}}}\right)^{2m_i}\left|H_{2m_i}\left(\frac{y_ie^{-t}-x_i}{\sqrt{1-e^{-2t}}}\right)\right|dt \\
 &\leq C \int_0^1 \left(\log\frac{1+s}{1-s}\right)^{\ell -1}\frac{(1-s)^{n-1}}{s^{n/2}}\left(1+\frac{(1-s)^{2(\ell -1)}}{s^{\ell -1}}\right)e^{-c\frac{|x-y|^2}{s}}ds \\
 & \leq C \int_0^1 \left(\log\frac{1+s}{1-s}\right)^{\ell -1}\frac{(1-s)^{n-1}}{s^{1/2}}\left(1+\frac{(1-s)^{2\ell}}{s^\ell}\right)e^{-c\frac{|x-y|^2}{s}}ds\frac{1}{|x-y|^{n-1}}\leq \frac{C}{|x-y|^{n-1}}.
\end{align*}
when $n=1$ we argue as in the last paragraph of {\bf (a)}. \newline

{\bf (c)} Let $k=(k_1,...,k_n)\in\mathbb N^n$ such that $k_1+k_2+...+k_n=\ell$. We get

\begin{align*}
&  \prod_{i=1}^n\sum_{m_i=0}^{k_i}\sum_{\begin{array}{c} d_i=0 \\ d_i\neq 0\;\mbox{ for some } i=1,..,n \end{array}}^{m_i}\left(\frac{e^{-t}}{\sqrt{1-e^{-2t}}}\right)^{2m_i-d_i}\left|H_{d_i}(y_i)\right|\left|H_{2m_i-d_i}\left(\frac{y_ie^{-t}-x_i}{\sqrt{1-e^{-2t}}}\right)\right| \\
 & \leq C \sum_{j=1}^\ell\sum_{i=1}^\ell \left(\frac{e^{-t}}{\sqrt{1-e^{-2t}}}+\left(\frac{e^{-t}}{\sqrt{1-e^{-2t}}}\right)^{2j-i}\right)(1+|y|)^i\left(1+\left(\frac{|x-ye^{-t}|}{\sqrt{1-e^{-2t}}}\right)^{2j-i}\right),\;\;\;\;t>0.
\end{align*}
Then,

\begin{align*}
& \int_0^\infty t^{\ell -1}e^{-nt} \frac{e^{-|x-ye^{-t}|^2/(1-e^{-2t})}}{(1-e^{-2t})^{n/2}} \prod_{i=1}^n\sum_{m_i=0}^{k_i}\sum_{\begin{array}{c} d_i=0 \\ d_i\neq 0\;\mbox{ for some } i=1,..,n \end{array}}^{m_i}\left(\frac{e^{-t}}{\sqrt{1-e^{-2t}}}\right)^{2m_i-d_i}  \\
& \cdot \left|H_{d_i}(y_i)\right|\left|H_{2m_i-d_i}\left(\frac{y_ie^{-t}-x_i}{\sqrt{1-e^{-2t}}}\right)\right|dt \\
 & \leq C \sum_{i,j=1}^\ell\int_0^1 \left(\log\frac{1+s}{1-s}\right)^{\ell -1}\frac{(1-s)^{n-1}}{s^{n/2}}\left(\frac{1-s}{\sqrt{s}}+\left(\frac{(1-s)}{\sqrt{s}}\right)^{2j-i}\right)\frac{1}{s^{i/2-1/4}}e^{-c\frac{|x-y|^2}{s}}ds\;(1+|y|)^{1/2} \\
 & \leq C \sum_{i,j=1}^\ell\int_0^1 \left(\log\frac{1+s}{1-s}\right)^{\ell -1}\frac{(1-s)^{n}}{s^{n/2}}\left(\frac{1}{s^{i/2+1/4}}+\frac{1}{s^{j-1/4}}\right)e^{-c\frac{|x-y|^2}{s}}ds\;(1+|y|)^{1/2} \\
 &  \leq C \int_0^1 \left(\log\frac{1+s}{1-s}\right)^{\ell -1}\frac{(1-s)^{n}}{s^{n/2+\ell-1/4}}e^{-c\frac{|x-y|^2}{s}}ds\;(1+|y|)^{1/2} \\
 &  \leq C \int_0^1 \frac{1}{s^{n/2+3/4}}e^{-c\frac{|x-y|^2}{s}}ds\;(1+|y|)^{1/2} \\
 & \leq C\frac{(1+|x|)^{1/2}}{|x-y|^{n-1/2}}.
\end{align*}\newline

{\bf (d)} Let $k=(k_1,...,k_n)\in\mathbb N^n$ such that $k_1+...+k_n=\ell$. For every $x\in\mathbb R^n$, we define $m(x)=\min\{1,\frac{1}{|x|^2}\}$. We have that

\begin{align*}
& \int_{m(x)}^\infty t^{\ell -1}e^{-nt} \frac{e^{-|x-ye^{-t}|^2/(1-e^{-2t})}}{(1-e^{-2t})^{n/2}} \sum_{|k|=\ell}\left(\begin{array}{c} \ell \\ k_1,...,k_n \end{array}\right)\prod_{i=1}^n2^{-k_i}\left(\frac{e^{-t}}{\sqrt{1-e^{-2t}}}\right)^{2k_i} \left|H_{2k_i}\left(\frac{y_ie^{-t}-x_i}{\sqrt{1-e^{-2t}}}\right)\right|dt \\
 & \leq C\int_{m(x)}^\infty t^{\ell -1} \frac{e^{-nt}}{(1-e^{-2t})^{n/2+\ell}}dt\leq\frac{C}{m(x)^{n/2}}\leq \frac{C}{m(x)^{1/4}|x-y|^{n-1/2}}  \\
 & \leq\frac{(1+|x|)^{1/2}}{|x-y|^{n-1/2}},
\end{align*}\newline
and

\begin{align*}
& \int_{m(x)}^\infty t^{\ell -1} \frac{e^{-|x-y|^2/2t}}{(2t)^{n/2}} \sum_{|k|=\ell}\left(\begin{array}{c} \ell \\ k_1,...,k_n \end{array}\right)\prod_{i=1}^n2^{-k_i}(2t)^{-k_i} \left|H_{2k_i}\left(\frac{y_i-x_i}{\sqrt{2t}}\right)\right|dt \\
 & \leq C\int_{m(x)}^\infty \frac{1}{t^{n/2+1}}dt\leq C\frac{(1+|x|)^{1/2}}{|x-y|^{n-1/2}},
\end{align*}\newline

{\bf (e)} Since $|e^{-nt}-1|\leq Ct$, $t\in (0,1)$, the same arguments than in {\bf (a)} allow us to obtain

\begin{align*}
\int_0^{m(x)} t^{\ell -1}(1-e^{-nt}) & \frac{e^{-|x-ye^{-t}|^2/(1-e^{-2t})}}{(1-e^{-2t})^{n/2}} \sum_{|k|=\ell}\left(\begin{array}{c} \ell \\ k_1,...,k_n \end{array}\right)\prod_{i=1}^n2^{-k_i}\left(\frac{e^{-t}}{\sqrt{1-e^{-2t}}}\right)^{2k_i}\left| H_{2k_i}\left(\frac{y_ie^{-t}-x_i}{\sqrt{1-e^{-2t}}}\right)\right|dt \\
  & \leq\frac{C}{|x-y|^{n-1}}.
\end{align*}\newline

{\bf (f)} By taking into account that
$$\left|\frac{1}{(1-e^{-2t})^{n/2}}-\frac{1}{(2t)^{n/2}}\right|\leq\frac{C}{t^{n/2-1}},\;\;\;\;0<t<1,$$
as in {\bf (a)} we get

\begin{align*}
\int_0^{m(x)} t^{\ell -1}\left|\frac{1}{(1-e^{-2t})^{n/2}}-\frac{1}{(2t)^{n/2}}\right| & e^{-|x-ye^{-t}|^2/(1-e^{-2t})} \sum_{|k|=\ell}\left(\begin{array}{c} \ell \\ k_1,...,k_n \end{array}\right)\prod_{i=1}^n2^{-k_i}\left(\frac{e^{-t}}{\sqrt{1-e^{-2t}}}\right)^{2k_i} \\
 & \cdot \left| H_{2k_i}\left(\frac{y_ie^{-t}-x_i}{\sqrt{1-e^{-2t}}}\right)\right|dt  \leq\frac{C}{|x-y|^{n-1}}.
\end{align*}\newline

{\bf (g)} We can write

$$|x-ye^{-t}|^2-|x-y|^2=|y|^2(1-e^{-t})^2+2\langle x-y,y\rangle(1-e^{-t}),\;\;\;\;t>0.$$
Then,

$$\frac{|x-ye^{-t}|^2-|x-y|^2}{t}\leq\frac{C}{t}(|y|^2t^2+t|x-y||y|)\leq C((|y|^2t+|x-y||y|),\;\;\;\;t\in(0,1).$$
We obtain

\begin{align*}
& \left|e^{-|x-ye^{-t}|^2/(1-e^{-2t})}-e^{-|x-y|^2/2t}\right|\leq C\left(\left|e^{-|x-ye^{-t}|^2/(1-e^{-2t})}-e^{-|x-ye^{-t}|^2/2t}\right|+\left|e^{-|x-ye^{-t}|^2/2t}-e^{-|x-y|^2/2t}\right|\right)    \\
 & \leq C\left(e^{-c|x-ye^{-t}|^2/t}|x-ye^{-t}|^2+e^{-c\min\{|x-ye^{-t}|^2,|x-y|^2\}/t}((|y|^2t+|x-y||y|)\right),\;\;\;t\in(0,1).
\end{align*}
By making $t=\log\frac{1+s}{1-s}\in(0,1)$ it follows that

\begin{align*}
& \left|e^{-|x-ye^{-t}|^2/(1-e^{-2t})}-e^{-|x-y|^2/2t}\right|_{|t=\log\frac{1+s}{1-s}}\leq Ce^{-c|x-y|^2/s}(s+|y|^2m(x)^{3/4}s^{1/4}+s^{1/4}m(x)^{1/4}|y|)   \\
 & \leq C e^{-c|x-y|^2/s}s^{1/4}(1+|y|)^{1/2},\;\;\;0<t<m(x),\;0<s<\frac{e-1}{e+1}.
\end{align*}
We can also get,

$$\frac{|ye^{-t}-x|}{\sqrt{1-e^{-2t}}}\leq C\frac{|y|(1-e^{-t})+|x-y|}{\sqrt{t}}\leq  C\left(|y|\sqrt{t}+\frac{|x-y|}{\sqrt{t}}\right)\leq C\left(1+\frac{|x-y|}{\sqrt{t}}\right),\;\;\;0<t<m(x).$$
We obtain

\begin{align*}
& \int_0^{m(x)} \left|t^{\ell -1}\frac{e^{-|x-ye^{-t}|^2/(1-e^{-2t})}-e^{-|x-y|^2/2t}}{(2t)^{n/2}}\sum_{|k|=\ell}\left(\begin{array}{c} \ell \\ k_1,...,k_n \end{array}\right)\prod_{i=1}^n2^{-k_i}\left(\frac{e^{-t}}{\sqrt{1-e^{-2t}}}\right)^{2k_i}  H_{2k_i}\left(\frac{y_ie^{-t}-x_i}{\sqrt{1-e^{-2t}}}\right) \right| dt \\
& \leq C \int_0^{\frac{e-1}{e+1}} \left(\log\frac{1+s}{1-s}\right)^{\ell -1}\frac{e^{-c\frac{|x-y|^2}{s}}}{s^{n/2+\ell-1/4}}\left(1+\frac{|x-y|}{\sqrt{s}}\right)^{2\ell}ds\;(1+|y|)^{1/2}    \\
 & \leq C \int_0^\infty \frac{e^{-c\frac{|x-y|^2}{s}}}{s^{n/2+3/4}}ds\;(1+|y|)^{1/2}\leq  C\frac{(1+|x|)^{1/2}}{|x-y|^{n-1/2}}.
\end{align*}\newline

{\bf (h)} We can write

$$\left|\left(\frac{e^{-t}}{\sqrt{1-e^{-2t}}}\right)^{2\ell}-\frac{1}{(2t)^{\ell}}\right|\leq C\frac{1}{t^{\ell -1/2}}\left|\frac{e^{-t}}{\sqrt{1-e^{-2t}}}-\frac{1}{\sqrt{2t}}\right|\leq\frac{1}{t^{\ell -1}},\;\;\;\;t\in (0,1).$$
It follows that

\begin{align*}
& \int_0^{m(x)} \left|t^{\ell -1}\frac{e^{-|x-y|^2/2t}}{(2t)^{n/2}}\left(\left(\frac{e^{-t}}{\sqrt{1-e^{-2t}}}\right)^{2\ell}-\frac{1}{(2t)^{\ell}}\right)\sum_{|k|=\ell}\left(\begin{array}{c} \ell \\ k_1,...,k_n \end{array}\right)\prod_{i=1}^n2^{-k_i}H_{2k_i}\left(\frac{y_ie^{-t}-x_i}{\sqrt{1-e^{-2t}}}\right) \right| dt \\
& \leq C \int_0^{\frac{e-1}{e+1}} \frac{e^{-c\frac{|x-y|^2}{s}}}{s^{n/2}}ds \leq \frac{C}{|x-y|^{n-1}}.
\end{align*}\newline

{\bf (i)} Let $r\in\mathbb N$. We have that

$$|(ae^{-u}-b)^r-(a-b)^r|\leq C\sum_{j=0}^{r-1}(|a|u)^{r-j}|a-b|^j,\;\;\;\;a,b\in\mathbb R\;\mbox{and}\;u\in (0,1).$$
We get

\begin{align*}
& e^{-|x-y|^2/2t}|(y_ie^{-t}-x_i)^r-(y_i-x_i)^r|\leq Ce^{-|x-y|^2/2t}\sum_{j=0}^{r-1}(|y_i|t)^{r-j}|x_i-y_i|^j   \\
 & \leq Ce^{-c|x-y|^2/t}\sum_{j=0}^{r-1}(|y|t)^{r-j}\leq Ce^{-c|x-y|^2/t}t^{3/4}|y|^{1/2} ,\;\;\;t\in(0,m(x)),
\end{align*}
and

\begin{align*}
& \left|\prod_{i=1}^nH_{2k_i}\left(\frac{e^{-t}y_i-x_i}{\sqrt{1-e^{-2t}}}\right)-\prod_{i=1}^nH_{2k_i}\left(\frac{x_i-y_i}{\sqrt{2t}}\right)\right|e^{-|x-y|^2/2t} \\
 & \leq Ce^{-|x-y|^2/2t}\sum_{j=1}^{n}\prod_{i=1}^{j-1}\left|H_{2k_i}\left(\frac{x_i-y_i}{\sqrt{2t}}\right)\right|\left|H_{2k_j}\left(\frac{e^{-t}y_i-x_i}{\sqrt{1-e^{-2t}}}\right)-H_{2k_j}\left(\frac{x_i-y_i}{\sqrt{2t}}\right)\right|\prod_{i=j+1}^n\left|H_{2k_i}\left(\frac{e^{-t}y_i-x_i}{\sqrt{1-e^{-2t}}}\right)\right| \\
 & \leq C e^{-c|x-y|^2/t}t^{1/4}|y|^{1/2} ,\;\;\;t\in(0,m(x)),
\end{align*}
where we understand that $\displaystyle\prod_{i=1}^0H_{2k_i}\left(\frac{x_i-y_i}{\sqrt{2t}}\right)=1$. Then,

\begin{align*}
& \int_0^{m(x)} \left|t^{\ell -1}\frac{e^{-|x-y|^2/2t}}{(2t)^{n/2+\ell}}\sum_{|k|=\ell}\left(\begin{array}{c} \ell \\ k_1...k_n \end{array}\right)2^{-\ell}\left(\prod_{i=1}^nH_{2k_i}\left(\frac{y_ie^{-t}-x_i}{\sqrt{1-e^{-2t}}}\right) -\prod_{i=1}^nH_{2k_i}\left(\frac{x_i-y_i}{\sqrt{2t}}\right)\right)\right| dt \\
& \leq C \int_0^{m(x)} \frac{e^{-c\frac{|x-y|^2}{t}}}{t^{n/2+3/4}}dt |y|^{1/2} \leq C \frac{(1+|x|)^2}{|x-y|^{n-1/2}}.
\end{align*}
By putting together all the above estimates we conclude that

$$\int_0^\infty |t^{\ell -1}\partial_t^{\ell}(T_t^{\mathcal A}(x,y)-T_t(x,y))|dt\leq C\frac{(1+|x|)^{1/2}}{|x-y|^{n-1/2}},\;\;\;\;(x,y)\in N_\delta.$$
Hence, we obtain

$$H(x,y)\leq C\frac{(1+|x|)^{1/2}}{|x-y|^{n-1/2}},\;\;\;\;(x,y)\in N_\delta.$$
We have that (see \cite[p. 19]{HTV})

\begin{align*}
 & \int_{\mathbb R^n} H(x,y)  \chi_{N_\delta}(x,y)dy\leq C\int_{|x-y|\leq \sqrt{m(x)}}\frac{(1+|x|)^{1/2}}{|x-y|^{n-1/2}}dy\leq C\int_0^{\sqrt{m(x)}}\frac{(1+|x|)^{1/2}}{u^{1/2}}du \\
& \leq C m(x)^{1/4}(1+|x|)^{1/2}\leq C,\;\;\;\;x\in\mathbb R^n.
\end{align*}
If $(x,y)\in N_\delta$, then $1+|x|\sim 1+|y|$ and $m(x)\sim m(y)$. By using this we get

$$\sup_{y\in\mathbb R^n}\int_{\mathbb R^n}\frac{(1+|x|)^{1/2}}{|x-y|^{n-1/2}}\chi_{N_\delta}(x,y)dx<\infty.$$
It follows that the operator $\mathcal H$ is bounded from $L^p(\mathbb R^n,dx)$ into itself, for every $1\leq p<\infty$. By \cite[Proposition 3.2.5]{Sa} $\mathcal H$ is bounded from $L^p(\mathbb R^n,\gamma_{-1})$ into itself, for every $1\leq p<\infty$.

We conclude That the operator $V_{\rho,loc}(\{t^\beta\partial_t^\beta(T_t^{\mathcal A}-T_t)\}_{t>0})$ is bounded from $L^p(\mathbb R^n,\gamma_{-1})$ into itself, for every $1\leq p<\infty$.

Thus we have proved that  $V_{\rho}(\{t^\beta\partial_t^\beta T_t^{\mathcal A} \}_{t>0})$ is bounded from $L^p(\mathbb R^n,\gamma_{-1})$ into itself, for every $1< p<\infty$. Furthermore  $V_{\rho,loc}(\{t^\beta\partial_t^\beta T_t^{\mathcal A}\}_{t>0})$ is bounded from $L^1(\mathbb R^n,\gamma_{-1})$ into $L^{1,\infty}(\mathbb R^n,\gamma_{-1})$, we are going to see that  $V_{\rho,glob}(\{t^\beta\partial_t^\beta T_t^{\mathcal A}\}_{t>0})$ is bounded from $L^1(\mathbb R^n,\gamma_{-1})$ into $L^{1,\infty}(\mathbb R^n,\gamma_{-1})$, when $\beta\in[0,1]$.

We firstly consider the case $\beta =0$. We have that

$$V_{\rho,glob}(\{T_t^{\mathcal A}\}_{t>0})(f)(x)\leq C\int_{\mathbb R^n}\chi_{N_1^c}(x,y)H(x,y)f(y)dy,\;\;\;\;x\in\mathbb R^n,$$
where

$$H(x,y)=\int_0^\infty|\partial_t T_t^{\mathcal A}(x,y)|dt,\;\;\;\;x,y\in\mathbb R^n.$$
We get

\begin{align*}
 \partial_t T_t^{\mathcal A}(x,y)= & -\frac{1}{\pi^{n/2}}\left[n+2e^{-t}\sum_{i=1}^ny_i(x_i-y_ie^{-t})-\frac{2e^{-2t}}{1-e^{-2t}}|x-ye^{-t}|^2\right] \\
 & \cdot \frac{e^{-nt}}{(1-e^{-2t})^{n/2 +1}} e^{-|x-e^{-t}y|^2/(1-e^{-2t})},\;\;\;\;x,y\in\mathbb R^n\;\mbox{and}\;t>0.
\end{align*}
Taking $e^{-t}=r$, we obtain

\begin{align*}
 \partial_t T_t^{\mathcal A}(x,y)_{|t=-\log r}= & -\frac{1}{\pi^{n/2}}\left[n(1-r^2)+2r\sum_{i=1}^ny_i(x_i-y_ir)(1-r^2)-2r^2|x-yr|^2\right] \\
 & \cdot \frac{r^n}{(1-r^2)^{n/2 +2}}e^{-|x-ry|^2/(1-r^2)},\;\;\;\;x,y\in\mathbb R^n\;\mbox{and}\;r\in (0,1).
\end{align*}
Let $x=(x_1,...,x_n),\; y=(y_1,...,y_n)\in\mathbb R^n$. The function

$$P_{x,y}(r)=n(1-r^2)+2r(1-r^2)\sum_{i=1}^ny_i(x_i-ry_i)-2r^2\sum_{i=1}^n(x_i-ry_i)^2,\;\;\;\;r\in(0,1),$$
is a polynomial of, at most, fourth degree. Then, the sign of $P_{x,y}(r)$ change at most four times. Following the idea given in \cite[p. 286]{GCMST2}, we have that

\begin{align}\label{2.8}
 H(x,y)= & \frac{1}{\pi^{n/2}}\int_0^1|P_{x,y}(r)|\frac{r^{n-1}}{(1-r^2)^{n/2+2}}e^{\frac{-|x-yr|^2}{1-r^2}}dr \nonumber\\
 & \leq C\sup_{t>0}|T_t^{\mathcal A}(x,y)|,\;\;\;\;x,y\in\mathbb R^n.
\end{align}
By sing \cite[Lemma 3.3.3]{Sa} we obtain

$$\sup_{t>0}|T_t^{\mathcal A}(x,y)|\leq C e^{|y|^2-|x|^2}[(1+|x|)^n\wedge (|x|\sin\theta(x,y))^{-n}],\;\;\;(x,y)\in N_1^c,$$
where $\theta(x,y)\in[0,\pi]$ is the angle between $x$ and $y$ being $x,y\in\mathbb R^n\setminus\{0\}$, and $n>1$ and $\theta(x,y)=0$, in other cases.

From \cite[Lemma 3.3.4]{Sa} we deduce that the operator $V_{\rho,glob}(\{T_t^{\mathcal A}\}_{t>0})$ is bounded from $L^1(\mathbb R^n,\gamma_{-1})$ into $L^{1,\infty}(\mathbb R^n,\gamma_{-1})$.

Hence,  $V_{\rho}(\{T_t^{\mathcal A}\}_{t>0})$ is bounded from $L^1(\mathbb R^n,\gamma_{-1})$ into $L^{1,\infty}(\mathbb R^n,\gamma_{-1})$.

Suppose now that $\beta\in(0,1]$. According to (\ref{2.2.1}) and  (\ref{2.2.2}) we get

$$\int_0^\infty|\partial_t(t^\beta\partial_t^\beta T_t^{\mathcal A}(x,y))|dt\leq C\left(\int_0^\infty|\partial_t T_t^{\mathcal A}(x,y)|dt+\int_0^\infty t|\partial_t^2 T_t^{\mathcal A}(x,y))|dt\right),\;\;\;\;x,y\in\mathbb R^n.$$

We need to estimate

$$H(x,y)=\int_0^\infty t|\partial_t^2 T_t^{\mathcal A}(x,y))|dt,\;\;\;\;x,y\in N_1^c.$$
By (\ref{2.1}) we obtain

\begin{align*}
& \partial_t^2 T_t^{\mathcal A}(x,y)  = e^{|y|^2-|x|^2}e^{-nt}\frac{e^{-|y-xe^{-t}|^2/(1-e^{-2t})}}{(1-e^{-2t})^{n/2}}\sum_{j=0}^2 \left(\begin{array}{c} 2 \\j \end{array}\right)(-n)^{2-j}\sum_{|\ell|=j}\left(\begin{array}{c} j \\ \ell_1,...,\ell_n \end{array}\right) \prod_{i=1}^n\sum_{m_i=0}^{\ell_i}\sum_{d_i=0}^{m_i}2^{-m_i} \\
& \cdot \left\{\begin{array}{c} \ell_i \\ m_i \end{array}\right\} \left(\begin{array}{c} m_i \\ d_i \end{array}\right)(-1)^{m_i-d_i}\left(\frac{e^{-t}}{\sqrt{1-e^{-2t}}}\right)^{2m_i-d_i}H_{d_i}(x_i)H_{2m_i-d_i}\left(\frac{x_ie^{-t}-y_i}{\sqrt{1-e^{-2t}}}\right),\;\;\;\;x,y\in\mathbb R^n,\;\mbox{and}\;t>0.
\end{align*}
We have to estimate, for $j=1,2,...,6$, the functions

$$H_j(x,y)=\int_0^\infty te^{|y|^2-|x|^2}e^{-nt}\frac{e^{-|y-xe^{-t}|^2/(1-e^{-2t})}}{(1-e^{-2t})^{n/2}}h_j(t,x,y)dt,\;\;\;\;x,y\in N_1^c,$$
where, for every $x,y\in\mathbb R^n$ and $t>0$,

\begin{align*}
& h_1(t,x,y)=1,\;\;h_2(t,x,y)=\left(\frac{e^{-t}}{\sqrt{1-e^{-2t}}}\right)^2\left(1+\left|\frac{xe^{-t}-y}{\sqrt{1-e^{-2t}}}\right|^2\right), \\
 & h_3(t,x,y)=\frac{e^{-t}}{\sqrt{1-e^{-2t}}}|x|\left(1+\left|\frac{xe^{-t}-y}{\sqrt{1-e^{-2t}}}\right|\right), \\
 & h_4(t,x,y)=\left(\frac{e^{-t}}{\sqrt{1-e^{-2t}}}\right)^4\left(1+\left|\frac{xe^{-t}-y}{\sqrt{1-e^{-2t}}}\right|^4\right), \\
 & h_5(t,x,y)=\left(\frac{e^{-t}}{\sqrt{1-e^{-2t}}}\right)^3|x|\left(1+\left|\frac{xe^{-t}-y}{\sqrt{1-e^{-2t}}}\right|^3\right), \\
 & h_6(t,x,y)=\left(\frac{e^{-t}}{\sqrt{1-e^{-2t}}}\right)^2|x|^2\left(1+\left|\frac{xe^{-t}-y}{\sqrt{1-e^{-2t}}}\right|^2\right).
\end{align*} \newline

{\bf (i)} By making the change of variable $t=\log\frac{1+s}{1-s}\in(0,\infty)$ we get

$$H_4(x,y)\leq C e^{|y|^2-|x|^2}\int_0^1\log\frac{1+s}{1-s}(1-s)^{n+3}\frac{e^{-c|y(1+s)-x(1-s)|^2/s}}{s^{n/2 +2}}ds.$$
As in the proof of \cite[Lemma 3.3.3]{Sa}, it follows that

\begin{align*}
 \int_0^{\frac{1}{8(1+|x|)^2}} & \log\frac{1+s}{1-s}(1-s)^{n+3}\frac{e^{-c|y(1+s)-x(1-s)|^2/s}}{s^{n/2 +2}}ds\\
 & \leq C \int_0^{\frac{1}{8(1+|x|)^2}} \frac{e^{-c|y(1+s)-x(1-s)|^2/s}}{s^{n/2 +1}}ds \\
 & \leq C \int_0^{\frac{1}{8(1+|x|)^2}} \frac{e^{-c/[s(1+|x|)^2]}}{s^{n/2 +1}}ds \leq C (1+|x|)^n,\;\;\;\;(x,y)\notin N_1.
\end{align*}
On the other hand

\begin{align*}
 \int_{\frac{1}{8(1+|x|)^2}}^1 & \log\frac{1+s}{1-s}(1-s)^{n+3}\frac{e^{-c|y(1+s)-x(1-s)|^2/s}}{s^{n/2 +2}}ds\\
 & \leq C \int_{\frac{1}{8(1+|x|)^2}}^1 \log\frac{1+s}{1-s}\frac{(1-s)^{n+3}}{s^{n/2 +2}}ds \\
 & \leq C \int_{\frac{1}{8(1+|x|)^2}}^1 \frac{1}{s^{n/2 +1}}ds \leq C (1+|x|)^n,\;\;\;\;(x,y)\notin N_1.
\end{align*}
Hence,

$$H_4(x,y)\leq C e^{|y|^2-|x|^2}(1+|x|)^n,\;\;\;\;(x,y)\notin N_1.$$
Since $|(1+s)y-(1-s)x|\geq(1-s)|x|\sin\theta(x,y)$, $x,y\in\mathbb R^n$ and $s\in (0,1)$, we obtain

\begin{align*}
& H_4(x.y)\leq C e^{|y|^2-|x|^2}\int_0^1  \log\frac{1+s}{1-s}(1-s)^{n+3}\frac{e^{-c|y(1+s)-x(1-s)|^2/s}}{s^{n/2 +2}}ds \\
 & \leq C e^{|y|^2-|x|^2}\left( \int_0^{1/2}  \frac{e^{-c(|x|(1-s)\sin(\theta(x,y)))^2/s}}{s^{n/2 +1}}ds+\int_{1/2}^1  e^{-c((1-s)|x|\sin\theta(x,y))^2}(1-s)^{n+2}ds\right) \\
 & \leq C e^{|y|^2-|x|^2}(|x|\sin\theta(x,y))^{-n},\;\;\;\;(x,y)\notin N_1.
\end{align*}
We conclude that

$$H_4(x,y)\leq C e^{|y|^2-|x|^2}((1+|x|)^n\wedge (|x|\sin\theta(x,y))^{-n}),\;\;\;\;(x,y)\notin N_1.$$ \newline

{\bf (ii)} By proceeding as in {\bf (i)}  we can see that

$$H_2(x,y)\leq C e^{|y|^2-|x|^2}((1+|x|)^n\wedge (|x|\sin\theta(x,y))^{-n}),\;\;\;\;(x,y)\notin N_1.$$ \newline

{\bf (iii)} We have that

$$H_6(x,y)\leq C e^{|y|^2-|x|^2}\int_0^\infty t e^{-(n+2)t}\frac{e^{-c|y-xe^{-t}|^2/(1-e^{-2t})}}{(1-e^{-2t})^{n/2+1}}dt\;|x|^2\;\;\;\;(x,y)\in\mathbb R^n.$$
We use some ideas developed in \cite{BrSj}. Firstly we study the integral extended over $(\log 2,\infty)$. We make the change of variable $r=e^{-t}\in(0,1/2)$. We get, as in \cite[p. 11]{BrSj},

\begin{align*}
\int_{\log 2}^\infty t & e^{-(n+2)t}\frac{e^{-c|y-xe^{-t}|^2/(1-e^{-2t})}}{(1-e^{-2t})^{n/2+1}}dt\;|x|^2 \\
 & \leq C \int_0^{1/2} (-\log r) r^{n+1} e^{-c|y-xr|^2/(1-r^2)}dr\;|x|^2 \\
 & \leq C \int_0^{1/2}  r^{n-1} e^{-c|y-xr|^2}dr\;|x|^2  \\
 & \leq C \left\{\begin{array}{l} |x|^{1-n},\;\;\;\;|y|\geq2|x|, \\  \\ e^{-|y_{\perp}|^2}|x|\left(\frac{|y|}{|x|}\right)^{n-1}+|x|^{2-n},\;\;\;\;|y|<2|x|. \end{array}\right.
\end{align*}
Here, if $x,y\in\mathbb R^n$ we decompose $y=y_\perp +y_x$, where $y_x$ is parallel to $x$ and $y_\perp$ is orthogonal to $x$.

We now study the integral extended to $(0,\log 2)$ adapting the arguments in \cite[pages 12 and 13]{BrSj}.

For every $x\in\mathbb R^n\setminus\{0\}$, $y\in\mathbb R^n$, we define $r_0=\frac{|y|}{|x|}\cos\theta(x,y)$. Suppose that $r_0\leq 1/3$. We have that

\begin{align*}
\int_0^{\log 2} t & e^{-(n+2)t}\frac{e^{-c|y-xe^{-t}|^2/(1-e^{-2t})}}{(1-e^{-2t})^{n/2+1}}dt\;|x|^2 \leq C \int_{1/2}^1 (-\log r)  \frac{e^{-c|y-xr|^2/(1-r)}}{(1-r)^{n/2+1}}dr\;|x|^2 \\
 & \leq C (1+|r_0|)^2|x|^2\int_{1/2}^1 \frac{e^{-c((1+|r_0|)^2|x|^2+|y_\perp|^2)/(1-r)}}{(1-r)^{n/2}}dr \leq C|x|^{-n},\;\;\;\;(x,y)\in N_1^c,
\end{align*}
because $|y-rx|^2=|r-r_0|^2|x|^2+|y_\perp|^2$ and $|r-r_0|=r-r_0\sim 1+|r_0|$.

If $r_0\geq 2$ we can proceed in a similar way and we obtain

$$\int_0^{\log 2} t  e^{-(n+2)t}\frac{e^{-c|y-xe^{-t}|^2/(1-e^{-2t})}}{(1-e^{-2t})^{n/2+1}}dt\;|x|^2\leq C|x|^{-n},\;\;\;\;(x,y)\in N_1^c.$$

Assume now that $1/3<r_0<2$. Let $(x,y)\notin N_1$. We define $I_1=\{r\in(1/2,1):\;|r-r_0|<(1-r_0)/2\}$, where $r_0<1$. Since
 $1-r\sim 1-r_0$, we can write

 \begin{align*}
\int_{I_1}  & (-\log r)  \frac{e^{-c|y-xr|^2/(1-r)}}{(1-r)^{n/2+1}}dr\;|x|^2\leq C(1-r_0)^{-n/2}|x|^2\int_{I_1} e^{-c|x|^2\frac{|r-r_0|^2}{1-r_0}-c\frac{|y_\perp|^2}{1-r_0}} dr \\
 & \leq C ((1+|x|)^n\wedge (|x|\sin\theta(x,y))^{-n}),
\end{align*}
(see \cite[2.3.1, with a=2]{BrSj}).

We consider now $I_2=\{r\in(1/2,1):\;1-r>3|r-r_0|/2\}$. Since $|r-r_0|\sim 1-r$, we get

 \begin{align*}
\int_{I_2}  & (-\log r)  \frac{e^{-c|y-xr|^2/(1-r)}}{(1-r)^{n/2+1}}dr\;|x|^2\leq C|x|^2\int_{I_2} \frac{e^{-c(|x|^2(1-r)+\frac{|y_\perp|^2}{1-r})}}{(1-r)^{n/2}} dr \\
 & \leq C\int_{I_2} \frac{e^{-c(|x|^2(1-r)+\frac{|y_\perp|^2}{1-r})}}{(1-r)^{n/2+1}} dr\leq C ((1+|x|)^n\wedge (|x|\sin\theta(x,y))^{-n}),
\end{align*}
(see \cite[2.3.2]{BrSj}).

Finally, we define $I_3=\{r\in(1/2,1):\;1-r\leq |r-r_0|/2\vee 3(r_0 -1)/2\}$. We have that $1-r\leq C|1-r_0|$ and $|r-r_0|\sim |1-r_0|$. Since $|x-y_x|=|x||1-r_0|$, it follows that

 \begin{align*}
& \int_{I_3}   (-\log r)  \frac{e^{-c|y-xr|^2/(1-r^2)}}{(1-r^2)^{n/2+1}}dr\;|x|^2\leq C|x|^2\int_{I_3} \frac{e^{-c(|x-y_x|^2+|y_\perp|^2)/(1-r)}}{(1-r)^{n/2}} dr \\
 & \leq C(|x-y_x|+|y_\perp|)^2\int_{I_3} \frac{e^{-c(|x-y_x|^2+|y_\perp|^2)/(1-r)}}{(1-r)^{n/2+2}} dr\leq C ((1+|x|)^n\wedge (|x|\sin\theta(x,y))^{-n}),
\end{align*}
(see \cite[2.3.3, with a=2]{BrSj}).

We conclude that

\begin{align*}
  & H_6(x,y)\leq C[|x|^{1-n}\chi_{\{|y|\geq 2|x|\}}(x,y)+\left(e^{-|y_{\perp}|^2}|x|\left(\frac{|y|}{|x|}\right)^{n-1}+|x|^{2-n}\right)\chi_{\{|y|<2|x|\}}(x,y) \\
 &  +|x|^{-n} +((1+|x|)^n\wedge (|x|\sin\theta(x,y))^{-n})]e^{|y|^2-|x|^2},\;\;\;\;(x,y)\notin N_1.
\end{align*} \newline

{\bf (iv)} By arguing as in {\bf (iii)} we obtain

\begin{align*}
  & H_1(x,y)+H_3(x,y)+H_5(x,y) \\
  & \leq C[|x|^{1-n}\chi_{\{|y|\geq 2|x|\}}(x,y)+\sum_{j=0}^1\left(e^{-|y_{\perp}|^2}|x|^j\left(\frac{|y|}{|x|}\right)^{n-1}+|x|^{1+j-n}\right)\chi_{\{|y|<2|x|\}}(x,y) \\
 &  +|x|^{-n} +((1+|x|)^n\wedge (|x|\sin\theta(x,y))^{-n})]e^{|y|^2-|x|^2},\;\;\;\;(x,y)\notin N_1.
\end{align*}

Putting together all the above estimates we obtain

\begin{align*}
  & \int_0^\infty|\partial_t(t^\beta\partial_t^\beta  T_t^{\mathcal A}(x,y)|dt \\
  & \leq C\Big(|x|^{1-n}\chi_{\{|y|\geq 2|x|\}}(x,y)+\sum_{j=0}^1\left(e^{-|y_{\perp}|^2}|x|^j\left(\frac{|y|}{|x|}\right)^{n-1}+|x|^{1+j-n}\right)\chi_{\{|y|<2|x|\}}(x,y) \\
 & +|x|^{-n} +((1+|x|)^n\wedge (|x|\sin\theta(x,y))^{-n})\Big) e^{|y|^2-|x|^2},\;\;\;\;(x,y)\in N_1^c.
\end{align*}

According to \cite[Lemma 3.3.4]{Sa} and \cite[Lemmas 4.2, 4.3 and 4.4]{BrSj} we conclude that $V_{\rho,glob}(\{t^\beta\partial_t^\beta  T_t^{\mathcal A}\}_{t>0})$ is bounded from $L^1(\mathbb R^n, \gamma_{-1})$ into $L^{1,\infty}(\mathbb R^n, \gamma_{-1})$.

It follows that $V_{\rho}(\{t^\beta\partial_t^\beta  T_t^{\mathcal A}\}_{t>0})$ is bounded from $L^1(\mathbb R^n, \gamma_{-1})$ into $L^{1,\infty}(\mathbb R^n, \gamma_{-1})$.

According to (\ref{eq1.1}) from the $L^p$-boundedness properties of $V_{\rho}(\{t^\beta\partial_t^\beta  T_t^{\mathcal A}\}_{t>0})$, that we have just proved, follow the corresponding ones for the family  $\{\lambda\Lambda(\{t^\beta\partial_t^\beta  T_t^{\mathcal A}\}_{t>0},\lambda)^{1/\rho}\}_{\lambda >0}$.

On the other hand, we have that

\begin{align*}
  & \mathcal O(\{t^\beta\partial_t^\beta  T_t^{\mathcal A}\}_{t>0},\{t_j\}_{j=1}^\infty)(f)(x)=\left(\sum_{i=1}^\infty\sup_{t_{i+1}\leq\epsilon_{i+1}<\epsilon_i\leq t_i}|t^\beta\partial_t^\beta  T_t^{\mathcal A}(f)(x)_{|t=\epsilon_i}-t^\beta\partial_t^\beta  T_t^{\mathcal A}(f)(x)_{|t=\epsilon_{i+1}}|^2\right)^{1/2} \\
  & \leq\int_{\mathbb R^n}|f(y)|\int_0^\infty|\partial_t(t^\beta\partial_t^\beta  T_t^{\mathcal A}(x,y))|dtdy.
\end{align*}
Also, for every $k\in\mathbb N$,

$$V_k(\{t^\beta\partial_t^\beta  T_t^{\mathcal A}\}_{t>0})(f)(x)\leq \int_{\mathbb R^n}|f(y)|\int_{2^{-k}}^{2^{-k+1}}|\partial_t(t^\beta\partial_t^\beta  T_t^{\mathcal A}(x,y))|dtdy,$$
and then

$$S_V(\{t^\beta\partial_t^\beta  T_t^{\mathcal A}\}_{t>0})(f)(x)\leq \int_{\mathbb R^n}|f(y)|\int_0^\infty|\partial_t(t^\beta\partial_t^\beta  T_t^{\mathcal A}(x,y))|dtdy.$$
From our previous results and following the same strategy we can obtain the $L^p$- boundedness properties for the oscillation and the short variation operators associated to $\{t^\beta\partial_t^\beta  T_t^{\mathcal A}\}_{t>0}$ stated in Theorem \ref{th1.1}.

\begin{rem}
The procedure we have used to prove that the variation operators associated with $\{t^\beta\partial_t^\beta  T_t^{\mathcal A}\}_{t>0}$ are bounded from $L^1(\mathbb R^n, \gamma_{-1})$ into $L^{1,\infty}(\mathbb R^n, \gamma_{-1})$ does not work when $\beta >1$. The reason is that there appears terms with $|x|^a$, $a\geq 3$. We have not been able to get examples that allow us to prove that the mentioned operators are not bounded from $L^1(\mathbb R^n, \gamma_{-1})$ into $L^{1,\infty}(\mathbb R^n, \gamma_{-1})$. These operators are more involved than the higher order Riesz transform associated with the operator $\mathcal A$ considered by Bruno and Sjogren (\cite{BrSj}).

On the other hand in the Gaussian context, that is, when the semigroup is $\{T_t^{ou}\}_{t>0}$ as far as we know the results , as the ones we have proved, remain still unproved. In \cite{HMMT} the subordinated Poisson semigroup to $\{T_t^{ou}\}_{t>0}$ is considered only with $\beta =0$.
\end{rem}

We now study the operators associated with the Poisson semigroup $\{P_t^{\mathcal A}\}_{t>0}$ associated to $\{T_t^{\mathcal A}\}_{t>0}$.

We recall that

$$P_t^{\mathcal A}(f)=\frac{t}{2\sqrt{\pi}}\int_0^\infty\frac{e^{-t^2/4s}}{s^{3/2}}T_s^{\mathcal A}(f)ds.$$
As it is shown in \cite[Corollary 6.2]{LeX} by using the subordinate representation we can see that

$$V_\rho(\{P_t^{\mathcal A}\}_{t>0})\leq V_\rho(\{T_t^{\mathcal A}\}_{t>0}).$$
Then, it follows that $V_\rho(\{P_t^{\mathcal A}\}_{t>0})$ and $\lambda\Lambda(\{P_t^{\mathcal A}\}_{t>0},\lambda)^{1/\rho}$, $\lambda >0$, are bounded from $L^p(\mathbb R^n, \gamma_{-1})$ into itself, $1<p<\infty$, and from $L^1(\mathbb R^n, \gamma_{-1})$ into $L^{1,\infty}(\mathbb R^n, \gamma_{-1})$. We are going to prove our results that contain these ones as particular cases.

The local and global operators are defined in the usual way.

If $0<t_1<...<t_k$, $k\in\mathbb N$, we have that

\begin{align*}
  & \left(\sum_{j=1}^{k-1}\left|t^\beta\partial_t^\beta  P_t^{\mathcal A}(f)(x)_{|t=t_{j+1}}-t^\beta\partial_t^\beta  P_t^{\mathcal A}(f)(x)_{|t=t_j}\right|^{\rho}\right)^{1/\rho} \\
  & \leq\int_{\mathbb R^n}|f(y)|\int_0^\infty|\partial_t(t^\beta\partial_t^\beta  P_t^{\mathcal A}(x,y))|dtdy,\;\;\;x\in\mathbb R^n.
\end{align*}
Then,

$$V_\rho(\{t^\beta\partial_t^\beta  P_t^{\mathcal A}\}_{t>0})(f)(x)\leq \int_{\mathbb R^n}|f(y)|\int_0^\infty|\partial_t(t^\beta\partial_t^\beta  P_t^{\mathcal A}(x,y))|dtdy,\;\;\;x\in\mathbb R^n.$$
Firstly we study the global operator $V_{\rho,glob}(\{t^\beta\partial_t^\beta  P_t^{\mathcal A}\}_{t>0})$. In order to do so we need to estimate

$$H(x,y)=\int_0^\infty|\partial_t(t^\beta\partial_t^\beta  P_t^{\mathcal A}(x,y))|dt,\;\;\;\;(x,y)\in N_1^c.$$
As in the $\{T_t^{\mathcal A}\}_{t>0}$-case we get

$$H(x,y)\leq C\left(\int_0^\infty t^{m-1}|\partial_t^m P_t^{\mathcal A}(x,y))|dt+\int_0^\infty t^{m}|\partial_t^{m+1} P_t^{\mathcal A}(x,y))|dt\right),\;\;\;\;(x,y)\in \mathbb R^n,$$
where $m\in\mathbb N$ and $m-1\leq\beta <m$.

We can write

$$P_t^{\mathcal A}(f)=\frac{1}{\sqrt{\pi}}\int_0^\infty\frac{e^{-u}}{\sqrt{u}}T_{t^2/4u}^{\mathcal A}(f)du.$$
We get

\begin{align*}
 \partial_t P_t^{\mathcal A}(f)(x)= & \frac{1}{\sqrt{\pi}}\int_0^\infty \frac{e^{-u}}{\sqrt{u}}\partial_sT_s^{\mathcal A}(f)(x)_{|s=t^2/4u}\frac{t}{2u}du=\frac{1}{\sqrt{\pi}}\int_0^\infty \frac{e^{-t^2/4s}}{\sqrt{s}}\partial_sT_s^{\mathcal A}(f)(x)ds\\
  & =\int_{\mathbb R^n}f(y)\frac{1}{\sqrt{\pi}}\int_0^\infty \frac{e^{-t^2/4s}}{\sqrt{s}}\partial_sT_s^{\mathcal A}(x,y)dsdy,\;\;\;x,y\in\mathbb R^n.
\end{align*}
We have that

$$\partial_t P_t^{\mathcal A}(x,y)=\frac{1}{\sqrt{\pi}}\int_0^\infty \frac{e^{-t^2/4s}}{\sqrt{s}}\partial_sT_s^{\mathcal A}(x,y)ds,\;\;\;x,y\in\mathbb R^n.$$
By using \cite[Lemma 4]{BCCFR} we obtain

\begin{align*}
 |\partial_t^m P_t^{\mathcal A}(x,y)| & \leq C\int_0^\infty \frac{\left|\partial_t^{m-1}\left(e^{-t^2/4s}\right)\right|}{\sqrt{s}}|\partial_sT_s^{\mathcal A}(x,y)|ds\\
  & \leq C\int_0^\infty \frac{e^{-t^2/8s}}{s^{m/2}}|\partial_sT_s^{\mathcal A}(x,y)|ds,\;\;\;x,y\in\mathbb R^n.
\end{align*}
All the derivatives under the integral sign can be justified provided that $f\in L^q(\mathbb R^n,\gamma_{-1})$, $1\leq q<\infty$.

By (\ref{2.8}) and \cite[Lemma 3.3.3]{Sa} we get

\begin{align*}
 \int_0^\infty t^{m-1} & |\partial_t^m P_t^{\mathcal A}(x,y)|dt  \leq C\int_0^\infty |\partial_sT_s^{\mathcal A}(x,y)| \int_0^\infty \frac{e^{-t^2/8s}}{s^{m/2}}t^{m-1}dtds\\
  & \leq C\int_0^\infty |\partial_sT_s^{\mathcal A}(x,y)|ds \\
  & \leq C \sup_{s>0}|T_s^{\mathcal A}(x,y)|\leq  Ce^{|y|^2-|x|^2}((1+|x|)^n\wedge (|x|\sin\theta(x,y))^{-n}),\;\;\;(x,y)\in N_1^c.
\end{align*}
According to \cite[Proposition 2.1]{MPS} we also obtain

\begin{align*}
 \int_0^\infty t^{m-1} & |\partial_t^m P_t^{\mathcal A}(x,y)|dt  \leq C e^{|y|^2-|x|^2}\sup_{s>0}\frac{e^{-|y-e^{-s}x|^2/(1-e^{-2s})}}{(1-e^{-2s})^{n/2}} \\
  & \leq C \left\{\begin{array}{l} e^{-|x1^2},\;\;\;\;\;\;\;\;\mbox{if}\;\langle x,y\rangle\leq 0, \\
   \\
   |x+y|^n e^{\frac{|y|^2-|x|^2}{2}-\frac{|x-y||x+y|}{2}},\;\;\;\;\mbox{if}\;\langle x,y\rangle\geq 0,
  \end{array}\right.,\;\;\;(x,y)\in N_1^c.
\end{align*}
We conclude that

$$H(x,y)\leq C((1+|x|)^n\wedge (|x|\sin\theta(x,y))^{-n}),\;\;\;\;x\in N_1^c,$$
and

\begin{align*}
 & H(x,y) \leq C \left\{\begin{array}{l} e^{-|x1^2},\;\;\;\;\;\;\;\;\mbox{if}\;\langle x,y\rangle\leq 0, \\
   \\
   |x+y|^n e^{\frac{|y|^2-|x|^2}{2}-\frac{|x-y||x+y|}{2}},\;\;\;\;\mbox{if}\;\langle x,y\rangle\geq 0,
  \end{array}\right.,\;\;\;(x,y)\in N_1^c.
\end{align*}
It follows that $V_{\rho,glob}(\{t^\beta\partial_t^\beta P_t^{\mathcal A}\}_{t>0})$ is bounded from $L^p(\mathbb R^n,\gamma_{-1})$ into itself, for every $1<p<\infty$, and  from $L^1(\mathbb R^n,\gamma_{-1})$ into $L^{1,\infty}(\mathbb R^n,\gamma_{-1})$.

We can write

\begin{align}\label{2.11}
V_{\rho,loc}(\{t^\beta\partial_t^\beta P_t^{\mathcal A}\}_{t>0})(f)\leq V_{\rho,loc}(\{t^\beta\partial_t^\beta (P_t^{\mathcal A}-P_t)\}_{t>0})(f)+V_{\rho,loc}(\{t^\beta\partial_t^\beta P_t\}_{t>0})(f),
\end{align}
where
$$P_t(z)=\frac{\Gamma((n+1)/2)}{\pi^{(n+1)/2}}\frac{t}{(t^2+|z|^2)^{(n+1)/2}},\;\;\;\;z\in\mathbb R^n\;\mbox{and}\;t>0.$$
We have that, for certain $a_0,a_1,...,a_m\in\mathbb R$,

$$t^m\partial_t^m P_t(z)=P_t(z)\frac{a_0+a_1\left(\frac{|z|}{t}\right)^2+...+a_m\left(\frac{|z|}{t}\right)^{2m}}{\left(1+\frac{|z|^2}{t^2}\right)^{m}},\;\;\;\;z\in\mathbb R^n\;\mbox{and}\;t>0.$$
This equality can be proved by induction. As in (\ref{2.3}), but adapting $\Psi$ and $\Phi$, we obtain

$$t^\beta \partial_t^\beta P_t(z)=\frac{1}{t^n}\Psi_\beta\left(\frac{|z|}{t}\right),\;\;\;\;z\in\mathbb R^n\;\mbox{and}\;t>0.$$
where

$$\Psi_\beta(u)=\frac{1}{\Gamma(m-\beta)}\int_1^\infty \Phi_m\left(\frac{u}{v}\right)v^{-m-n}(1-v)^{m-\beta -1}dv,\;\;\;\;u>0,$$
and

$$\Phi_m(u)=\frac{1}{(1+u^2)^{(n+1)/2}}\frac{a_0+a_1u^2+...+a_mu^{2m}}{(1+u^2)^m},\;\;\;\;u>0.$$

We get

\begin{align*}
 & \int_0^\infty\Psi'_\beta(u)u^n du=\frac{1}{\Gamma(m-\beta)}\int_0^\infty u^n\int_1^\infty \Phi'_m\left(\frac{u}{v}\right)v^{-m-n-1}(1-v)^{m-\beta -1}dvdu   \\
  & =\frac{2^{n/2}}{\Gamma(m-\alpha)}\int_1^\infty\frac{(1-v)^{m-\beta -1}}{v^m}\int_0^\infty|\Phi'_m(\omega)|\omega^nd\omega dv>\infty,
\end{align*}
and $\displaystyle\lim_{u\rightarrow +\infty}\Psi_\beta(u)=0.$

We define $\displaystyle \Psi_{\beta,t}(z)=\frac{1}{t^n}\Psi_\beta\left(\frac{z}{t}\right)$, $z\in\mathbb R^n$ and $t>0$, and the convolution operator $\varphi_{\beta,t}$ by $\varphi_{\beta,t}(f)=\Psi_{\beta,t}\ast f$, $t>0$. We have that $\varphi_{\beta,t}=t^\beta\partial_t^\beta P_t$, $t>0$. According to \cite[Lemma 2.4]{CJRW1}, the operators $V_\rho(\{t^\beta\partial_t^\beta P_t\}_{t>0})$, $\mathcal O(\{t^\beta\partial_t^\beta P_t\}_{t>0},\{t_j\}_{j=1}^\infty)$, $\lambda\Lambda(\{t^\beta\partial_t^\beta P_t\}_{t>0},\lambda)^{1/\rho}$ and $S_V(\{t^\beta\partial_t^\beta P_t\}_{t>0})$ are bounded from $L^p(\mathbb R^n,dx)$ into itself, for every $1<p<\infty$, and from $L^1(\mathbb R^n,dx)$ into $L^{1,\infty}(\mathbb R^n,dx)$.

By $(F_\rho,\|\;\|_\rho)$ we continue denoting the space defined by (\ref{2.10}). We obtain

\begin{align*}
 & \|t^\beta\partial_t^\beta P_t(x-y)\|_{F_\rho}\leq C\left(\int_0^\infty t^{m-1}|\partial_t^m P_t(x-y)|dt+\int_0^\infty t^{m}|\partial_t^{m+1} P_t(x-y)|dt\right)  \\
  & \leq C\int_0^\infty |\partial_t T_t(x-y)|dt \leq C\int_0^\infty\frac{e^{-c|x-y|^2/t}}{t^{(n+2)/2}}dt\leq\frac{C}{|x-y|^n},\;\;\;\;x,y\in\mathbb R^n.
\end{align*}

We can prove (see \cite[Propositions 3.2.5 and 3.2.7]{Sa} and \cite[p. 10]{HTV}) that the operator $V_{\rho,loc}(\{t^\beta\partial_t^\beta P_t\}_{t>0})$ is bounded from $L^p(\mathbb R^n,\gamma_{-1})$ into itself, for every $1<p<\infty$, and from $L^1(\mathbb R^n,\gamma_{-1})$ into $L^{1,\infty}(\mathbb R^n,\gamma_{-1})$.

We now study the operator $V_{\rho,loc}(\{t^\beta\partial_t^\beta (P_t^{\mathcal A}-P_t)\}_{t>0})$. We get

$$V_{\rho,loc}(\{t^\beta\partial_t^\beta (P_t^{\mathcal A}-P_t)\}_{t>0})(f)(x)\leq \int_{\mathbb R^n} f(y)K(x,y) dy, \;\;\;\;x\in\mathbb R^n,$$
where

$$K(x,y)=\int_0^\infty|\partial_t(t^\beta\partial_t^\beta (P_t^{\mathcal A}-P_t))(x,y)|dt,\;\;\;\;x,y\in\mathbb R^n.$$
We have that

$$K(x,y)\leq C\left(\int_0^\infty t^{m-1}|\partial_t^m (P_t^{\mathcal A}-P_t))(x,y)|dt+\int_0^\infty t^{m}|\partial_t^{m+1} (P_t^{\mathcal A}-P_t))(x,y)|dt\right),\;\;\;\;x,y\in\mathbb R^n.$$
Let $\ell\in\mathbb N$. By using \cite[Lemma 4]{BCCFR} we obtain

\begin{align*}
 & \int_0^\infty t^{\ell-1}|\partial_t^\ell (P_t^{\mathcal A}- P_t)(x-y)|dt\leq C \int_0^\infty t^{\ell -1}\int_0^\infty\frac{|\partial_t^{\ell}e^{-t^2/4s}|}{s^{3/2}}|T_s^{\mathcal A}(x,y)-T_s(x,y)|dsdt  \\
  & \leq C\int_0^\infty \frac{| T_s^{\mathcal A}(x.y)-T_s(x,y)|}{s^{3/2}}\int_0^\infty t^{\ell-1}\frac{e^{-t^2/8s}}{s^{(\ell -1)/2}} dtds \\
  & \leq C\int_0^\infty\frac{| T_s^{\mathcal A}(x.y)-T_s(x,y)|}{s}ds,\;\;\;\;x,y\in\mathbb R^n.
\end{align*}
Then,

$$K(x,y)\leq C\int_0^\infty\frac{| T_s^{\mathcal A}(x.y)-T_s(x,y)|}{s}ds,\;\;\;\;x,y\in\mathbb R^n.$$
For $x\in\mathbb R^n$, recall that $m(x)=\min\{1,1/|x|^2\}$. We have that

\begin{align*}
 & \int_{m(x)}^\infty \frac{| T_t^{\mathcal A}(x.y)-T_t(x,y)|}{t}dt\leq C\left( \int_{m(x)}^\infty e^{-nt}\frac{e^{-|x-ye^{-t}|^2/(1-e^{-2t})}}{(1-e^{-2t})^{n/2}}\frac{dt}{t}+\int_{m(x)}^\infty\frac{e^{-|x-y|^2/2t}}{t^{n/2+1}}dt\right)  \\
  & \leq C\int_{m(x)}^\infty \frac{1}{t^{n/2+1}}dt \leq C\frac{1}{m(x)^{n/2}}\leq C\frac{(1+|x|)^{1/2}}{|x-y|^{n-1/2}},\;\;\;\;(x,y)\in N_1.
\end{align*}

On the other hand, we can write

\begin{align*}
 & T_t^{\mathcal A}(x.y)-T_t(x,y)= (e^{-nt}-1)\frac{e^{-|x-ye^{-t}|^2/(1-e^{-2t})}}{(1-e^{-2t})^{n/2}}  +\left(\frac{1}{(1-e^{-2t})^{n/2}}-\frac{1}{(2t)^{n/2}}\right)e^{-|x-ye^{-t}|^2/(1-e^{-2t})}  \\
 & +\frac{1}{(2t)^{n/2}}\left(e^{-|x-ye^{-t}|^2/(1-e^{-2t})}-e^{-|x-y|^2/2t}\right)  ,\;\;\;\;x,y\in \mathbb R^n\;\mbox{and}\;t>0.
\end{align*}
By proceeding as in the first part of this proof we get

$$\left|T_t^{\mathcal A}(x.y)-T_t(x,y)\right|_{|t=\log\frac{1+s}{1-s}}\leq C e^{-c|x-y|^2/s}\left(\frac{1}{s^{n/2-1}}+\frac{(1+|x|)^{1/2}}{s^{n/2-1/4}}\right),\;\;\;\;(x,y)\in N_1\;\mbox{and}\; 0<t<m(x).$$
Then,

\begin{align*}
  \int_0^{m(x)} & \frac{| T_t^{\mathcal A}(x.y)-T_t(x,y)|}{t}dt\leq C \int_0^1 e^{-c|x-y|^2/s}\left(\frac{1}{s^{n/2-1}}+\frac{(1+|x|)^{1/2}}{s^{n/2-1/4}}\right)ds  \\
  & \leq C\frac{(1+|x|)^{1/2}}{|x-y|^{n-1/2}},\;\;\;\;(x,y)\in N_1.
\end{align*}
We conclude that

$$K(x,y)\leq C \frac{(1+|x|)^{1/2}}{|x-y|^{n-1/2}},\;\;\;\;(x,y)\in N_1,$$
and

$$V_{\rho,loc}(\{t^\beta\partial_t^\beta(P_t^{\mathcal A}-P_t)\}_{t>0})(f)(x)\leq C\int_{\mathbb R^n}|f(y)|\chi_{N_1}(x,y)\frac{(1+|x|)^{1/2}}{|x-y|^{n-1/2}} dy,\;\;\;\;x\in\mathbb R^n.$$
It follows that $V_{\rho,loc}(\{t^\beta\partial_t^\beta(P_t^{\mathcal A}-P_t)\}_{t>0})$ is bounded from $L^p(\mathbb R^n,\gamma_{-1})$ into itself, for every $1\leq p<\infty$.

From (\ref{2.11}) we deduce that $V_{\rho,loc}(\{t^\beta\partial_t^\beta P_t^{\mathcal A})\}_{t>0})$ is bounded from $L^p(\mathbb R^n,\gamma_{-1})$ into itself, for every $1<p<\infty$, and from $L^1(\mathbb R^n,\gamma_{-1})$ into $L^{1,\infty}(\mathbb R^n,\gamma_{-1})$.

Thus, we have proved that $V_{\rho}(\{t^\beta\partial_t^\beta P_t^{\mathcal A})\}_{t>0})$ is bounded from $L^p(\mathbb R^n,\gamma_{-1})$ into itself, for every $1<p<\infty$, and from $L^1(\mathbb R^n,\gamma_{-1})$ into $L^{1,\infty}(\mathbb R^n,\gamma_{-1})$.

The boundedness of the other operators can be done as in the $\{T_t^{\mathcal A}\}_{t>0}$-case.

\subsection{Proof of the Theorem 1.2} We are going to use the same arguments developed in the proof of Theorem \ref{th1.1}, together with some results in \cite{CJRW1} and \cite{HM2}.

\noindent{\bf (a)} According to \cite[Lemma 4.2 and Remark 4.3]{HM2} the operator $V_\rho^X(\{T_t\}_{t>0})$ is bounded from $L_X^p(\mathbb R^n,dx)$ into  $L^p(\mathbb R^n,dx)$ for every $1<p<\infty$. We define the Banach space (modulus constant functions) $F_\rho^X$ as follows. $F_\rho^X$ consists of all those functions $g:(0,\infty)\rightarrow X$   such that

$$\|g\|_{F_\rho^X}:=\sup_{\begin{array}{c} 0<t_1<...<t_k \\ k\in\mathbb N \end{array}}\left(\sum_{j=1}^{k-1}\|g(t_{j+1})-g(t_j)\|^\rho\right)^{1/\rho}<\infty.$$

We have that

$$V_\rho^X(\{T_t\}_{t>0})(f)(x)=\|T_{\centerdot}(f)(x)\|_{F_\rho^X},\;\;\;\;x\in\mathbb R^n.$$
Reproducing the proof of \cite[Lemma 2.4]{CJRW1} by (\ref{2.3}) we can see that $V_\rho^X(\{t^\alpha\partial_t^\alpha T_t\}_{t>0})$ is bounded from $L_X^1(\mathbb R^n,dx)$ into  $L^{1,\infty}(\mathbb R^n,dx)$ and from $L_X^p(\mathbb R^n,dx)$ into  $L^p(\mathbb R^n,dx)$ for every $1<p<\infty$, by using $\{\varphi_{\beta,t}\}_{t>0}$  (see (\ref{2.15})).

We deduce from the proof of Theorem \ref{th1.1} that there exists, for every $1\leq p<\infty$, measurable functions $K_1$ and $K_2$ defined in $\mathbb R^n\times\mathbb R^n$ such that, for every $f\in L^p(\mathbb R^n,\gamma_{-1})$, $1\leq p<\infty$,

$$V_{\rho,glob}^X(\{t^\alpha\partial_t^\alpha T_t^{\mathcal A}\}_{t>0})(f)(x)\leq\int_{\mathbb R^n}\|f(y)\|K_1(x,y)\chi_{N_1^c}(x,y)dy,\;\;\;\;x\in\mathbb R^n,$$
and

$$V_{\rho,loc}^X(\{t^\alpha\partial_t^\alpha(T_t^{\mathcal A}- T_t\}_{t>0})(f)(x)\leq\int_{\mathbb R^n}\|f(y)\|K_2(x,y)\chi_{N_1}(x,y)dy,\;\;\;\;x\in\mathbb R^n,$$
being, for $j=1,2$, the operator $\mathbb K_j$ defined by

$$\mathbb K_j(\mathfrak{g})(x)=\int_{\mathbb R^n}\mathfrak{g}(y)K_j(x,y)dy,\;\;\;\;x\in\mathbb R^n,$$
bounded from $L^p(\mathbb R^n,\gamma_{-1})$ into itself. Furthermore, $\mathbb K_1$ (respectively $\mathbb K_2$) is bounded from $L^1(\mathbb R^n,\gamma_{-1})$ into $L^{1,\infty}(\mathbb R^n,\gamma_{-1})$, for every $\beta\in [0,1]$ (respectively $\beta\geq 0$).

By arguing as in the proof of Theorem \ref{th1.1} we conclude that $V_{\rho}^X(\{t^\beta\partial_t^\beta T_t^{\mathcal A}\}_{t>0})$ is bounded from $L^p_X(\mathbb R^n,\gamma_{-1})$ into $L^p(\mathbb R^n,\gamma_{-1})$ for every $1<p<\infty$ and $\beta\geq 0$, and from $L^1_X(\mathbb R^n,\gamma_{-1})$ to $L^{1,\infty}(\mathbb R^n,\gamma_{-1})$, when $\beta\in[0,1]$.\newline

\noindent{\bf (b)} The Poisson case can be established in a similar way.

\begin{rem}
The arguments used in the proof of Theorem \ref{th1.2} allow us to prove analogous results for the vector valued oscillation, short variation and jump operators associated with $\{t^\beta\partial_t^\beta T_t^{\mathcal A}\}_{t>0}$ and $\{t^\beta\partial_t^\beta P_t^{\mathcal A}\}_{t>0}$ when $\beta\geq 0$.
\end{rem}

\section{Proofs of Theorems \ref{th1.3}, \ref{th1.4} and \ref{th1.5} and Corollary \ref{cor1.6}}

\subsection{Proof of Theorem \ref{th1.3}} Let $i=1,...,n.$ If $0<\epsilon_1<\epsilon_2<...<\epsilon_k$, $k\in\mathbb N$, we have that

\begin{align}\label{3.1}
&  \left(\sum_{j=1}^{k-1}\left|\mathcal R_{i,\epsilon_{j+1}}^{\mathcal A}(f)(x)-\mathcal R_{i,\epsilon_{j}}^{\mathcal A}(f)(x)\right|^\rho\right)^{1/\rho}=\left(\sum_{j=1}^{k-1}\left|\int_{\epsilon_j<|x-y|<\epsilon_{j+1}}\mathcal R_i^{\mathcal A}(x,y)f(y)dy\right|^\rho\right)^{1/\rho}  \nonumber\\
  & \leq \sum_{j=1}^{k-1}\left(\left|\int_{\epsilon_j<|x-y|<\epsilon_{j+1}}\mathcal R_i^{\mathcal A}(x,y)f(y)dy\right|^\rho\right)^{1/\rho}\leq \int_{\epsilon_1<|x-y|<\epsilon_k}|\mathcal R_i^{\mathcal A}(x,y)||f(y)|dy,\;\;\;\;x\in\mathbb R^n.
\end{align}

We understand local and global operators in the usual way. By (\ref{3.1}) we deduce that

$$V_{\rho,glob}(\{\mathcal R_{i,\epsilon}^{\mathcal A}\}_{\epsilon>0})(f)(x)\leq  \int_{\mathbb R^n}|\mathcal R_i^{\mathcal A}(x,y)|\chi_{N_\delta}^c(x,y)|f(y)|dy\;\;\;\;x\in\mathbb R^n.$$
Here $\delta >0$, and it will be specified later.

We define the classical $i$-th Riesz transform $\mathcal R_i$ by

$$\mathcal R_i(f)(x)=\lim_{\epsilon\rightarrow 0}\int_{|x-y|>\epsilon}\mathcal R_i(x,y)f(y)dy,\;\mbox{for almost all}\;\;x\in\mathbb R^n,$$
for every $f\in L^p(\mathbb R^n,dx)$, $1\leq p<\infty$, where $\mathcal R_i(x,y)$ can be written as

$$\mathcal R_i(x,y)=\frac{-2}{\pi^{\frac{n+1}{2}}}\int_0^\infty\frac{x_i-y_i}{(2t)^{\frac{n+2}{2}}}e^{-|x-y|^2/2t}dt,\;\;\;\;x,y\in\mathbb R^n.$$

As in (\ref{3.1}) we again get

\begin{align*}
   \left|\left(V_{\rho,loc}(\{\mathcal R_{i,\epsilon}^{\mathcal A}\}_{\epsilon>0})-V_{\rho,loc}(\{\mathcal R_{i,\epsilon}\}_{\epsilon>0})\right)(f)(x)\right|  \leq \int_{\mathbb R^n}|\mathcal R_i^{\mathcal A}(x,y)-\mathcal R_i(x,y)|\chi_{N_\delta}(x,y)|f(y)|dy\;\;\;\;x\in\mathbb R^n.
\end{align*}
Let $1<p<\infty$. In \cite[(3.18) ]{B1} it was proved that

$$|\mathcal R_i^{\mathcal A}(x,y)\leq C\left\{\begin{array}{ccc} e^{-\eta|x|^2}, & \langle x,y \rangle \leq 0 & \\ & & (x,y)\in N_\delta^c, \\ |x+y|^n\exp(\eta\frac{|y|^2-|x|^2-|x-y||x+y|}{2}) & \langle x,y\rangle >0, &   \end{array}\right.$$
where $1/p<\eta<1$ and $\delta=\eta^{-1/2}$. Since

$$\sup_{x\in\mathbb R^n}\int_{\mathbb R^n}e^{\frac{|x|^2-|y|^2}{p}}|\mathcal R_i^{\mathcal A}(x,y)|\chi_{N_\delta}^c(x,y)dy<\infty,$$
and

$$\sup_{y\in\mathbb R^n}\int_{\mathbb R^n}e^{\frac{|x|^2-|y|^2}{p}}|\mathcal R_i^{\mathcal A}(x,y)|\chi_{N_\delta}^c(x,y)dx<\infty,$$
it follows that $V_{\rho,glob}(\{\mathcal R_{i,\epsilon}^{\mathcal A}\}_{\epsilon>0})$ is bounded from $L^p(\mathbb R^n,\gamma_{-1})$ into itself.

It was also proved in \cite[(3.20)]{B1} that

\begin{align}\label{3.2}
|\mathcal R_i^{\mathcal A}(x,y)-\mathcal R_i(x,y)|\leq C\frac{(1+|x|)^{1/2}}{|x-y|^{n-1/2}},\;\;\;\;(x,y)\in N_\delta.
\end{align}
Then,  $V_{\rho,loc}(\{\mathcal R_{i,\epsilon}^{\mathcal A}\}_{\epsilon>0})-V_{\rho,loc}(\{\mathcal R_{i,\epsilon}\}_{\epsilon>0})$ is bounded from $L^p(\mathbb R^n,\gamma_{-1})$ into itself.

Bruno and Sjogren proved in \cite{BrSj} that there exists positive function $K(x,y)$, $x,y\in\mathbb R^n$, such that

$$|\mathcal R_i^{\mathcal A}(x,y)|\leq K(x,y), \;\;\;\;(x,y)\in N_\delta^c,$$
and that the operator $\mathbb K$ defined by

$$\mathbb K(f)(x)=\int_{\mathbb R^n}K(x,y)\chi_{N_1^c}(x,y)f(y)dy,\;\;\;\;x\in\mathbb R^n,$$
is bounded from $L^1(\mathbb R^n,\gamma_{-1})$ into $L^{1,\infty}(\mathbb R^n,\gamma_{-1})$ then, so is $V_{\rho,glob}(\{\mathcal R_{i,\epsilon}^{\mathcal A}\}_{\epsilon>0})$.

Given that the inequality in (\ref{3.2}) holds for every $(x,y)\in N_\delta$, it follows that $V_{\rho,loc}(\{\mathcal R_{i,\epsilon}^{\mathcal A}\}_{\epsilon>0})-V_{\rho,loc}(\{\mathcal R_{i,\epsilon}\}_{\epsilon>0})$ is bounded from $L^1(\mathbb R^n,\gamma_{-1})$ into $L^{1,\infty}(\mathbb R^n,\gamma_{-1})$.

It is well-known that $\mathcal R_i$ is a Calder\'on-Zygmund operator. The results in \cite{CJRW2} (see also \cite[Theorem 1]{MTX}) established that $V_\rho(\{\mathcal R_{i,\epsilon}\}_{\epsilon>0})$ is bounded from $L^q(\mathbb R^n,dx)$ into itself, for every $1<q<\infty$, and from $L^1(\mathbb R^n,dx)$ into $L^{1,\infty}(\mathbb R^n,dx)$.

We define $\mathcal R_{i,\epsilon}(x,y)=\mathcal R_i(x,y)\chi_{\{|x-y|>\epsilon\}}(y)$, $x,y\in\mathbb R^n$, and $\epsilon>0$. If $0<\epsilon_1<\epsilon_2<...<\epsilon_k$, $k\in\mathbb N$, we have that

$$\left(\sum_{j=1}^{k-1}\left|\mathcal R_{i,\epsilon_{j+1}}(x,y)-\mathcal R_{i,\epsilon_{j}}(x,y)\right|^\rho\right)^{1/\rho}\leq\sum_{j=1}^{k-1}\left|\mathcal R_{i,\epsilon_{j+1}}(x,y)-\mathcal R_{i,\epsilon_{j}}(x,y)\right|\leq  |\mathcal R_i(x,y)|,\;\;\;\;x,y\in\mathbb R^n.$$
Then,

$$V_\rho(\{\mathcal R_{i,\epsilon}(x,y)\}_{\epsilon>0})\leq |\mathcal R_i(x,y)|\leq\frac{C}{|x-y|^n},\;\;\;\;x,y\in\mathbb R^n.$$
As above, once again, the arguments in the proof of \cite[Proposistions 3.2.5 and 3.2.7]{Sa} and \cite[p. 10]{HTV} lead to the establishment that $V_{\rho,loc}(\{\mathcal R_{i,\epsilon}\}_{\epsilon>0})$ is bounded from $L^1(\mathbb R^n,\gamma_{-1})$ into $L^{1,\infty}(\mathbb R^n,\gamma_{-1})$ and from $L^p(\mathbb R^n,\gamma_{-1})$ into itself, so  we finally have established that the same holds for $V_{\rho}(\{\mathcal R_{i,\epsilon}^{\mathcal A}\}_{\epsilon>0})$.

In the similar way we can prove the $L^p(\mathbb R^n,\gamma_{-1})$-properties for the operators  $\mathcal O(\{\mathcal R_{i,\epsilon}^{\mathcal A}\}_{\epsilon>0},\{t_j\}_{j=1}^\infty)$, $S_V(\{\mathcal R_{i,\epsilon}^{\mathcal A}\}_{\epsilon>0})$ and $\lambda\Lambda(\{\mathcal R_{i,\epsilon}^{\mathcal A}\}_{\epsilon>0},\lambda)^{1/\rho}$, $\lambda >0$.

\subsection{A Banach valued version of Theorem \ref{th1.3}}
In \cite[Theorem 1.2]{HLM} it was established that if $X$ is a Banach space with the UMD-property and of $q$-martingale cotype being $2\leq q<\rho<\infty$, then $V_\rho^X(\{H_\epsilon\}_{\epsilon >0})$ is bounded from $L_X^p(\mathbb R,dx)$ into $L^p(\mathbb R,dx)$, for every $1<p<\infty$, where for every $\epsilon >0$, $H_\epsilon$ is the $\epsilon$-truncated Hilbert transform given by

$$H_\epsilon(f)(x)=\int_{|x-y|>\epsilon}\frac{f(y)}{x-y}dy,\;\;\;\;x\in\mathbb R.$$
By proceeding as in the proof of \cite[Theorem 2.1]{CJRW2} we can deduce that $V_\rho^X(\{\mathcal R_{i,\epsilon}\}_{\epsilon >0})$ is bounded from $L_X^p(\mathbb R^n,dx)$ into $L^p(\mathbb R^n,dx)$, for every $i=1,...,n$ and  $1<p<\infty$, provided that $X$ satisfies the above properties.

We now, following the same strategy than the one used in the proof of Theorem \ref{th1.3}, can concluded that the operator  $V_\rho^X(\{\mathcal R_{i,\epsilon}^{\mathcal A}\}_{\epsilon >0})$ is bounded from $L_X^p(\mathbb R^n,\gamma_{-1})$ into $L^p(\mathbb R^n,\gamma_{-1})$, for every $i=1,...,n$ and  $1<p<\infty$, when $X$ is a Banach space with the UMD-property and $q$-martingale cotype being $2\leq q<\rho <\infty$.

\subsection{Proof of Theorem \ref{th1.4}} Let $i=1,...,n$ and $k\in\mathbb N^n$. We have that

$$\mathcal R_i^{\mathcal A}(\tilde{\mathcal H}_k)=-\frac{1}{\sqrt{n+|k|}}\tilde{\mathcal H}_{k+e_i},$$

$$\partial_{x_i}\tilde{\mathcal H}_k=-\tilde{\mathcal H}_{k+e_i},$$
and

$$P_t^{\mathcal A}(\tilde{\mathcal H}_k)=e^{-t\sqrt{n+|k|}}\tilde{\mathcal H}_k.$$

Then,

$$P_t^{\mathcal A}(\mathcal R_i^{\mathcal A}(\tilde{\mathcal H}_k))=-\frac{e^{-t\sqrt{n+|k|+1}}}{\sqrt{n+|k|}}\tilde{\mathcal H}_{k+e_i},$$
and

$$\mathcal C_{i,t}^{\mathcal A}(\tilde{\mathcal H}_k)=-\int_t^\infty e^{-s\sqrt{n+|k|}}\tilde{\mathcal H}_{k+e_i}ds=-\frac{e^{-t\sqrt{n+|k|}}}{\sqrt{n+|k|}}\tilde{\mathcal H}_{k+e_i}.$$
Note that $\mathcal C_{i,t}^{\mathcal A}(\tilde{\mathcal H}_k)\neq P_t^{\mathcal A}(\mathcal R_i^{\mathcal A}(\tilde{\mathcal H}_k))$. Then, we can not proceed as in the proof of \cite[Corollary 2.8]{CJRW1}.

We consider the Euclidean conjugation operator $Q_{t,i}$ defined by

$$Q_{i,t}(f)(x)=\frac{\Gamma(\frac{n+1}{2})}{\pi^{\frac{n+1}{2}}}\int_{\mathbb R^n}\frac{x_i-y_i}{(t^2+|x-y|^2)^{\frac{n+1}{2}}}f(y)dy,\;\;\;\;x\in\mathbb R^n,$$
for every $i=1,...,n$.

Let $i=1,...,n$, since $Q_{i,t}(f)=P_t(\mathcal R_i)(f)$, $f\in L^p(\mathbb R^n,dx)$, $1\leq p<\infty$, by using \cite[Corollary 2.8]{CJRW1} we deduce that the operators $V_\rho(\{Q_{i,t}\}_{t>0})$, $\mathcal O(\{Q_{i,t}\}_{t>0},\{t_j\}_{j=1}^\infty)$, $S_V(\{Q_{i,t}\}_{t>0})$ and $\lambda\Lambda(\{Q_{i,t}\}_{t>0,\lambda})^{1/\rho}$ are bounded from $L^p(\mathbb R^n,dx)$ into itself, for every $1<p<\infty$.

Let $1<p<\infty$ and $f\in L^p(\mathbb R^n,\gamma_{-1})$. By using subordination formula we can write

\begin{align*}
    \mathcal C_{i,t}^{\mathcal A}(f)(x)= & \int_t^\infty\partial_{x_i}\frac{s}{2\sqrt{\pi}}\int_0^\infty\frac{e^{-s^2/4u}}{u^{3/2}}T_u^{\mathcal A}(f)(x)duds=\frac{1}{2\sqrt{\pi}} \int_0^\infty\partial_{x_i}T_u^{\mathcal A}(f)(x)\frac{1}{u^{3/2}}\int_t^\infty e^{-s^2/4u}sdsdu \\
   & = \frac{1}{\sqrt{\pi}}\int_0^\infty\frac{e^{-t^2/4u}}{u^{1/2}}\partial_{x_i}T_u^{\mathcal A}(f)(x)du,\;\;\;\;x\in\mathbb R^n\;\mbox{and}\;t>0,
\end{align*}
and

$$Q_{i,t}(f)(x)=\frac{1}{\sqrt{\pi}}\int_0^\infty\frac{e^{-t^2/4u}}{u^{1/2}}\partial_{x_i}T_u(f)(x)du,\;\;\;\;x\in\mathbb R^n\;\mbox{and}\;t>0.$$

We are going to see that $V_\rho(\{\mathcal C_{i,t}^{\mathcal A}\}_{t>0})$ is bounded from $L^p(\mathbb R^n,\gamma_{-1})$ into itself. For the other operators one can proceed similarly.

We define the local and global operators  in the usual way. We have that

\begin{align*}
   V_\rho( \{\mathcal C_{i,t}^{\mathcal A}\}_{t>0})(f)\leq  V_{\rho,loc}( \{\mathcal C_{i,t}^{\mathcal A}-Q_{i,t}\}_{t>0})(f)+  V_{\rho,loc}( \{Q_{i,t}^{\mathcal A}\}_{t>0})(f)+ V_{\rho,glob}( \{\mathcal C_{i,t}^{\mathcal A}\}_{t>0})(f).
\end{align*}
We can write

$$\mathcal C_{i,t}^{\mathcal A}(f)(x)= \int_{\mathbb R^n}f(y)\;\mathcal C_{i,t}^{\mathcal A}(x,y)dy,\;\;\;\;x\in\mathbb R^n,$$
where

$$\mathcal C_{i,t}^{\mathcal A}(x,y)=\frac{1}{\sqrt{\pi}}\int_0^\infty\frac{e^{-t^2/4u}}{u^{1/2}}\partial_{x_i}T_u^{\mathcal A}(x,y)du,\;\;\;\;x,y\in\mathbb R^n\;\mbox{and}\;t>0.$$

We study firstly $V_{\rho,glob}( \{\mathcal C_{i,t}^{\mathcal A}\}_{t>0})$. We have that

$$V_{\rho,glob}( \{\mathcal C_{i,t}^{\mathcal A}\}_{t>0})(f)(x)\leq \int_{\mathbb R^n}K(x,y)\chi_{N_\delta^c}(x,y)f(y)dy,\;\;\;\;x\in\mathbb R^n,$$
where

$$K(x,y)=\int_0^\infty|\partial_t\mathcal C_{i,t}^{\mathcal A}(x,y)|dt,\;\;\;\;x,y\in\mathbb R^n.$$
We get

\begin{align*}
    K(x,y))\leq & C\int_0^\infty\frac{|\partial_{x_i}T_u^{\mathcal A}(x,y)|}{\sqrt{u}}\int_0^\infty\left|\partial_t\left[e^{-t^2/4u}\right]\right|dtdu\leq C\int_0^\infty\frac{|\partial_{x_i}T_u^{\mathcal A}(x,y)|}{\sqrt{u}}du \\
   & \leq  C \int_0^\infty \frac{e^{-nu}}{(1-e^{-2u})^{n/2}}e^{-|x-ye^{-u}|^2/(1-e^{-2u})}\frac{|x_i-y_ie^{-u}|}{1-e^{-2u}}\frac{du}{\sqrt{u}} \\
  & \leq C \int_0^\infty \frac{e^{-nu}}{(1-e^{-2u})^{(n+2)/2}}e^{-\eta|x-ye^{-u}|^2/(1-e^{-2u})}du \\
  & \leq Ce^{-\eta(|x|^2-|y|^2)}\int_0^\infty \frac{e^{-nu}}{(1-e^{-2u})^{(n+2)/2}}e^{-\eta|y-xe^{-u}|^2/(1-e^{-2u})}du,\;\;\;\;x,y\in\mathbb R^n,
\end{align*}
where $0<\eta<1$ will be chosen later.

We choose $\delta=1/\sqrt{\eta}$. proceeding as in the proof of (\ref{2.25}) we obtain

\begin{align}\label{2.29}
    K(x,y))\leq & C\left\{\begin{array}{lll}
       e^{-\eta|x|^2},  &  \mbox{if}\;\langle x,y\rangle \leq 0 & \\
         & & (x,y)\notin N_\delta. \\
        |x+y|^n exp\left(-\frac{\eta}{2}|x+y||x-y|-\frac{\eta}{2}(|x|^2-|y|^2)\right) & \mbox{if}\;\langle x,y\rangle \geq 0 &
    \end{array}  \right.
\end{align}
We have that
$$\sup_{x\in\mathbb R^n}\int_{\mathbb R^n}e^{\frac{|x|^2-|y|^2}{p}}K(x,y)\chi_{N_\delta^c}(x,y)dy<\infty\;\;\mbox{and}\;\;\sup_{y\in\mathbb R^n}\int_{\mathbb R^n}e^{\frac{|x|^2-|y|^2}{p}}K(x,y)\chi_{N_\delta^c}(x,y)dx<\infty,$$
provided that $1/p<\eta<1$.

Then, the operator $\mathbb K$ defined by

$$\mathbb K(f)(x)=\int_{\mathbb R^n}K(x,y)f(y)dy,\;\;\;\;x\in\mathbb R^n,$$
is bounded from $L^p(\mathbb R^n,\gamma_{-1})$ into itself and then, so is $V_{\rho,glob}( \{\mathcal C_{i,t}^{\mathcal A}\}_{t>0})$.

On the other hand, we can write

$$V_{\rho,loc}( \{\mathcal C_{i,t}^{\mathcal A}-Q_{i,t}\}_{t>0})(f)(x)\leq\int_{\mathbb R^n}|f(y)|H(x,y)\chi_{N_\delta}(x,y)dy,\;\;\;\;x\in\mathbb R^n,$$
where

$$H(x,y)=\int_0^\infty|\partial_t(\mathcal C_{i,t}^{\mathcal A}-Q_{i,t})(x,y)|dt,\;\;\;\;x,y\in\mathbb R^n.$$

We get

$$H(x,y)\leq C\int_0^\infty\frac{|\partial_{x_i}(T_t^{\mathcal A}-T_t)(x,y)|}{\sqrt{t}}dt,\;\;\;\;x,y\in\mathbb R^n.$$

We have that

\begin{align*}
   & \partial_{x_i}(T_t^{\mathcal A}-T_t)(x,y)=\frac{2e^{-nt}}{\pi^{n/2}}\frac{x_i-y_ie^{-t}}{(1-e^{-2t})^{\frac{n}{2}+1}}e^{-\frac{|x-ye^{-t}|^2}{1-e^{-2t}}} - \frac{2}{\pi^{n/2}}\frac{x_i-y_i}{(2t)^{\frac{n}{2}+1}}e^{-\frac{|x-y|^2}{2t}} \\
  & =\frac{2}{\pi^{n/2}}\left(\frac{(e^{-nt}-1)(x_i-y_ie^{-t})}{(1-e^{-2t})^{\frac{n}{2}+1}}e^{-\frac{|x-ye^{-t}|^2}{1-e^{-2t}}}+\left(\frac{1}{(1-e^{-2t})^{\frac{n}{2}+1}}-\frac{1}{(2t)^{\frac{n}{2}+1}}\right)(x_i-y_ie^{-t})e^{-\frac{|x-ye^{-t}|^2}{1-e^{-2t}}}\right. \\
  & \left.+ \frac{x_i-y_ie^{-t}-x_i+y_i}{(2t)^{\frac{n}{2}+1}}e^{-\frac{|x-ye^{-t}|^2}{1-e^{-2t}}}+ \frac{x_i-y_i}{(2t)^{\frac{n}{2}+1}}\left(e^{-\frac{|x-ye^{-t}|^2}{1-e^{-2t}}}-e^{-\frac{|x-y|^2}{2t}} \right) \right),\;\;\;\;x,y\in\mathbb R^n\;\mbox{and}\;t>0.
\end{align*}
Using again the definition of $m(x)=\min\{1,1/|x|^2\}$. We obtain

\begin{align*}
    \int_{m(x)}^\infty|\partial_{x_i}(T_t^{\mathcal A}-T_t)(x,y)|\frac{dt}{\sqrt{t}}\leq C\int_{m(x)}^\infty\frac{1}{t^{\frac{n}{2}+1}}dt\leq\frac{C}{m(x)^{n/2}} \leq C\frac{(1+|x|)^{1/2}}{|x-y|^{n-1/2}},\;\;\;\;(x,y)\in N_\delta.
\end{align*}

By making the change of variables $\displaystyle t=\log\frac{1+s}{1-s}\in(0,\infty)$ and by proceeding as in the proof of Theorem \ref{th1.1} we get

\begin{align*}
    |\partial_{x_i}(T_t^{\mathcal A}-T_t)(x,y)|_{|t=\log\frac{1+s}{1-s}}\leq C\left(\frac{1}{s^{\frac{n-1}{2}}}+\frac{(1+|y|)^{1/2}}{s^{n/2+1/4}}\right)e^{-c|x-y|^2/s},\;\;\;\;(x,y)\in N_\delta\;\mbox{and}\;0<t<m(x).
\end{align*}
Then,

$$\int_0^{m(x)}|\partial_{x_i}(T_t^{\mathcal A}-T_t)(x,y)|\frac{dt}{\sqrt{t}}\leq C\frac{(1+|y|)^{1/2}}{|x-y|^{n-1/2}},\;\;\;\;(x,y)\in N_\delta.$$
We conclude that

\begin{align}\label{2.30}
   H(x,y)\leq C\frac{(1+|x|)^{1/2}}{|x-y|^{n-1/2}},\;\;\;\;(x,y)\in N_\delta.
\end{align}
It follows that the operator $V_{\rho,loc}( \{\mathcal C_{i,t}^{\mathcal A}-Q_{i,t}\}_{t>0})$ is bounded from $L^p(\mathbb R^n,\gamma_{-1})$ into itself.

We define

$$Q_{i,t}(x,y)=\frac{x_i-y_i}{(t^2+|x-y|^2)^{\frac{n+1}{2}}},\;\;\;\;(x,y)\in\mathbb R^n\;\mbox{and}\;t>0.$$
We have that

\begin{align*}
    V_{\rho}( \{Q_{i,t}\}_{t>0}) & \leq \int_0^\infty|\partial_tQ_{i,t}(x,y)|dt\leq C|x_i-y_i|\int_0^\infty\frac{t}{(t^2+|x-y|^2)^{\frac{n+3}{2}}}dt \\
   & \leq \frac{C}{|x-y|^n},\;\;\;\;x,y\in \mathbb R^n.
\end{align*}
Since $V_{\rho}( \{Q_{i,t}\}_{t>0})$ is bounded from $L^p(\mathbb R^n,dx)$ into itself, the arguments in the proof of \cite[Proposition 3.2.7]{Sa} allow us to prove that $V_{\rho,loc}( \{Q_{i,t}\}_{t>0})$  is bounded from $L^p(\mathbb R^n,dx)$ into itself. Then, \cite[Proposition 3.2.5]{Sa} implies that $V_{\rho,loc}( \{Q_{i,t}\}_{t>0})$ is bounded from $L^p(\mathbb R^n,\gamma_{-1})$ into itself.

We conclude that $V_{\rho}( \{\mathcal C_{i,t}\}_{t>0})$ is bounded from $L^p(\mathbb R^n,\gamma_{-1})$ into itself.

\subsection{An auxiliary result} We prove a higher dimension version of the result established in \cite[Theorem 1.2]{HLM}, for every the operators $Q_{i,t}$, $i=1,...,n$ and $t>0$, we have  defined  in previous subsection.

\begin{prop}\label{Prop3.1}
Le $X$ be a Banach space and $1<p<\infty$.
\begin{enumerate}
    \item[(i)] Let $2\leq q<\rho<\infty$ and $X$ be of $q$-martingale cotype with UMD property. Then, for every $i=1,...,n$, the operator $V_\rho^X(\{Q_{i,t}\}_{t>0})$ is bounded from $L^p_X(\mathbb R^n,dx)$ into $L^p(\mathbb R^n,dx)$.
    \item[(ii)] Let $2<\rho<\infty$. Suppose that, for every $i=1,...,n$, $V_\rho^X(\{Q_{i,t}\}_{t>0})$ is bounded from $L^p_X(\mathbb R^n,dx)$ into $L^p(\mathbb R^n,dx)$. Then, $X$ has $\rho$-martingale cotype and the UMD property.
\end{enumerate}
\end{prop}
\begin{proof} $\;$ \newline

\noindent{\bf (i)} According to \cite[Theorem 1.2]{HLM} the operator $V_\rho(\{H_\epsilon\}_{\epsilon >0})$ is bounded from $L^p_X(\mathbb R^n,dx)$ into $L^p(\mathbb R^n,dx)$. The arguments in the proof of \cite[Theorem 2.1]{CJRW2} lead to the corresponding boundedness of $V_\rho(\{\mathcal R_{i,\epsilon}\}_{\epsilon >0})$ for $i=1,...,n$.

For every $t>0$, we define

$$A_t(f)(x)=\frac{1}{|B(0,t)|}\int_{B(0,t)}f(x+y)dy,\;\;\;\;x\in\mathbb R^n.$$
These operators are usually named central differential operators in $\mathbb R^n$.

By \cite[Theorem 3.1]{HM2}, since $X$ has $q$-martingale cotype with $2\leq q<\rho$, the operator $V_\rho^X(\{A_t\}_{t >0})$ is bounded from $L^p_X(\mathbb R^n,dx)$ into $L^p(\mathbb R^n,dx)$. By using \cite[Lemma 4.2 and Remark 4.3]{HM2} we obtain that $V_\rho(\{P_t\}_{t>0})$ is bounded from $L^p_X(\mathbb R^n,dx)$ into $L^p(\mathbb R^n,dx)$.

For every $i=1,...,n$, we have that

$$Q_{i,t}(f)=P_t(\mathcal R_i(f)),\;\;\;\;f\in L^p(\mathbb R^n,dx)\otimes X\;\mbox{and}\;t>0.$$
Since $X$ has the UMD property we conclude that $V_\rho^X(\{Q_{i,t}\}_{t >0})$ is bounded from $L^p_X(\mathbb R^n,dx)$ into $L^p(\mathbb R^n,dx)$, for $i=1,...,n$. \newline

\noindent{\bf (ii)} Assume that $V_\rho^X(\{Q_{i,t}\}_{t >0})$ is bounded from $L^p_X(\mathbb R^n,dx)$ into $L^p(\mathbb R^n,dx)$, for every $i=1,...,n$. It is well-known that, for every $f\in L^p(\mathbb R^n,dx)\otimes X$, and $i=1,...,n$

$$\mathcal R_i(f)(x)=\lim_{t\rightarrow 0}Q_{i,t}(f)(x),\;\;\;\;\mbox{for almost all}\;x\in\mathbb R^n.$$
It follows that, for every $f\in L^p(\mathbb R^n,dx)\otimes X$,

$$\|\mathcal R_i(f)\|_{L^p_X(\mathbb R^n,dx)}\leq \|V_\rho^X(\{Q_{i,t}\}_{t >0})\|_{L^p_X(\mathbb R^n,dx)}+\|Q_{1,i}(f)\|_{L^p_X(\mathbb R^n,dx)}\leq C \|f\|_{L^p_X(\mathbb R^n,dx)},\;\;\;i=1,...,n.$$
We concluded that $X$ has the UMD property.

In order to prove that $X$ has $\rho$-martingale cotype we use the following strategy. Firstly we transfer our problem from $\mathbb R^n\otimes X$ to $\mathbb T^n\otimes X$, being $\mathbb T^n$ the torus $n$-dimensional (\cite{BV}), next to $\mathbb T^1\otimes X$ (\cite{Jod}) and then we use the results obtained in \cite{HLM}.

For every $f\in L^p(\mathbb R^n,dx)\otimes X$, $\displaystyle \sum_{i=1}^n\mathcal R_i^2(f)=nf$. Since $X$ is UMD this last equality also holds for every $f\in L^p_X(\mathbb R^n,dx)$. Then, we obtain

$$P_t(f)=\frac{1}{n}\sum_{i=1}^n Q_{t,i}(\mathcal R_i(f)),\;\;\;\;  f\in L^p_X(\mathbb R^n,dx).$$
That leads straightforwardly to the boundedness of $V_p^X(\{P_t\}_{t >0})$ from $L^p_X(\mathbb R^n,dx)$ into $L^p(\mathbb R^n,dx)$ and then, $V_\rho^X(\{P_t\}_{t >0})$  is bounded from $L^q_X(\mathbb R^n,dx)$ into $L^q(\mathbb R^n,dx)$, for every $1<q<\infty$.

Let $a>0$. We say that a function $g:(0,a]\rightarrow X$ is in $F_{\rho ,a}^X$ when

$$ \|g\|_{F_{\rho ,a}^X}=\sup_{\begin{array}{c} 0<t_1<...<t_k\leq a \\ k\in\mathbb N \end{array}} \left(\sum_{j=1}^{k-1}\|g(t_{j+1})-g(t_j)\|^\rho \right)^{1/\rho}<\infty.$$
$(F_{\rho ,a}^X, \|\cdot\|_{F_{\rho ,a}^X})$ is a Banach space (modulus constant functions).

We consider the function

$$\begin{array} {ccccccc} g_a: & \mathbb R^n & \rightarrow & F_{\rho ,a}^X & & &  \\
& x & \rightarrow & g(x): & (0,a] & \longrightarrow \mathbb R  &  \\
 & & & & t & \rightarrow & [g(x)](t)=e^{-|x|t}.
\end{array}$$
We are going to see that $g_a$ is continuous. Indeed, let $x_0\in(0,a]$. For every $y\in\mathbb R^n$ and $0<t_1<...<t_k\leq a$, $k\in\mathbb N$, we have that

\begin{align*}
  & \left(  \sum_{j=1}^{k-1}\left|[g_a(x)](t_{j+1})-[g_a(x)](t_{j})-\left([g_a(y)](t_{j+1})-[g_a(y)](t_{j})\right)\right|^\rho\right)^{1/\rho} \\
   & \leq \sum_{j=1}^{k-1}\int_{t_j}^{t_{j+1}}\left|\frac{\partial}{\partial t}\left([g_a(x)](t)-[g_a(y)](t)\right)\right|dt \\
    & \leq \int_0^a \left||x|e^{-|x|t}-|y|e^{-|y|t}\right|dt.
\end{align*}
Then,

$$\|g_a(x)-g_a(y)\|_{F_{\rho ,a}^{\mathbb R}}\leq\int_0^a \left||x|e^{-|x|t}-|y|e^{-|y|t}\right|dt.$$
The dominated convergence theorem allows us to conclude that $g$ is continuosu in $X$.

We denote by $\{\mathbb P_t\}_{t>o}$ the Poisson semigroup over $\mathbb T^n=\left[-\frac{1}{2},\frac{1}{2}\right]^n$. That is,

$$\mathbb P_t(f)(x)=\sum_{k\in\mathbb Z^n}e^{-t|k|}c_k(f)e_k(x),\;\;\;\;x\in \left[-\frac{1}{2},\frac{1}{2}\right]^n.$$
Here, for every $k\in\mathbb Z^n$, $e_k(x)=e^{-2\pi i\langle x,k\rangle}$, $x\in\mathbb T^n$, and $c_k(f)=\int_{\mathbb T^n}f(x)e_k(x)dx$.

Let $1<q<\infty$. Our next objective is to see that $V_\rho^X(\{\mathbb P_t\}_{t>0})$ is a bounded operator from $L^q_X(\mathbb T^n,dx)$ into $L^q(\mathbb T^n,dx)$. In order to do so we use a transference argument inspired in some ideas in \cite{BV}.

Note firstly that by taking into account monotone convergence theorem

$$\|V_\rho^X(\{\mathbb P_t\}_{t>0})(f)\|_{L^q_X(\mathbb T^n,dx)}=\lim_{a\rightarrow\infty}\left\|\|\mathbb P_t(f)\|_{F_{\rho ,a}^X}\right\|_{L^q(\mathbb T^n,dx)}$$
Hence our objective will be obtained proving that there exists $C>0$ such that

$$\left\|\|\mathbb P_t(f)\|_{F_{\rho ,a}^X}\right\|_{L^q(\mathbb T^n,dx)}\leq C\|f\|_{L^q_X(\mathbb T^n,dx)},\;\;\;\;f\in L^q_X(\mathbb T^n,dx)\;\mbox{and}\;a>0.$$

Let $a>0$. We consider the function $g_a$ as above.

We can write $P_t$ as a Fourier multiplier in the following way

$$P_t(f)=\left(e^{-t|y|}\hat{f}(y)\right)^\vee,\;\;\;\;t>0,$$
where $\wedge$ and $\vee$ represent the Fourier transform and its inverse, respectively.

We have that there exists $C>0$ such that

$$\left\|\|P_t(f)\|_{F_{\rho ,a}^X}\right\|_{L^q(\mathbb R^n,dx)}\leq C\|f\|_{L^q_X(\mathbb R^n,dx)},\;\;\;\;f\in L^q_X(\mathbb R^n,dx)\;\mbox{and}\;a>0.$$
We now apply \cite[Theorem 2.4]{BV} in its unweighted form. Actually, there only scalar functions are considered, but the proof in \cite[pages from 531 to 534]{BV} can be rewritten when the functions take values in a Banach space $X$. Note that if $\mathcal P_X$ denotes the set of trigonometric polynomial with coefficients in $X$, it is a dense subspace of $L^q_X(\mathbb T^n,dx)$.

We conclude that

$$\left\|\|\mathbb P_t(f)\|_{F_{\rho ,a}^X}\right\|_{L^q(\mathbb T^n,dx)}\leq C\|f\|_{L^q_X(\mathbb T^n,dx)},\;\;\;\;f\in L^q_X(\mathbb T^n,dx)\;\mbox{and}\;a>0.$$

Since $C$ does not depend on $a>0$, we obtain

$$\|V_\rho^X(\{\mathbb P_t\}_{t>0})(f)\|_{L^q_X(\mathbb T^n,dx)}\leq C\|f\|_{L^q_X(\mathbb T^n,dx)},\;\;\;\;f\in L^q_X(\mathbb T^n,dx).$$

In the sequel we denote by $\{\mathbb P_t^{(k)}\}_{t>0}$ the Poisson semigroup in $\mathbb T^k$,  $k\in\mathbb N$. We are going to see that

$$\|V_\rho^X(\{\mathbb P_t^{(n-1)}\}_{t>0})(f)\|_{L^q_X(\mathbb T^{n-1},dx)}\leq C\|f\|_{L^q_X(\mathbb T^{n-1},dx)},\;\;\;\;f\in L^q_X(\mathbb T^{n-1},dx).$$

In order to prove this inequality we follows some ideas developed in \cite[p. 226]{Jod} understanding them in a vectorial setting.

We consider the extension, $E$, and the restriction, $R$, operators defined by

$$E(h)(x_1,...,x_n)=h(x_1,...,x_{n-1}),\;\;\;\;(x_1,...,x_n)\in\mathbb T^n,$$
and

$$R(f)(x_1,...,x_{n-1})=\int_{-1/2}^{1/2}f(x_1,...,x_{n-1},z)dz,\;\;\;\;(x_1,...,x_{n-1})\in\mathbb T^{n-1}.$$
Assume that $f\in L^q_X(\mathbb T^{n-1},dx)$. We have that

$$\mathbb P_t^{(n)}(E(f))(x)=\sum_{k\in\mathbb N^n}e^{-t|k|}c_k^{(n)}(E(f))e_k(x),\;\;\;\;x\in\mathbb T^n\;\mbox{and}\;t>0.$$
We can write

$$c_k^{(n)}(E(f))=\int_{\mathbb T^n}f(y_1,...,y_{n-1})\prod_{j=1}^n e^{-2\pi y_jk_j}dy_1...dy_{n-1}dy_n=\left\{\begin{array} {lll} 0, & k_n\neq 0 & \\ & & ,k\in\mathbb N^n. \\ c_{\tilde{k}}^{(n-1)}, & k_n=0, &
\end{array}\right.$$
For every $k=(k_1,...,k_n)\in\mathbb Z^n$, we write $\tilde{k}=(k_1,...,k_{n-1})$. $c_{\tilde{k}}^{(n-1)}(f)$ has the obvious meaning for every $\tilde{k}\in\mathbb Z^{n-1}$.

We get

$$\mathbb P_t^{(n)}(E(f))(x)=\sum_{k\in\mathbb N^{n-1}}e^{-t|k|}c_{k}^{(n-1)}(f)e_k(x),\;\;\;\;x\in\mathbb T^n.$$

Then

$$R(\mathbb P_t^{(n)}(E(f)))(x)=\mathbb P_t^{(n-1)}(E(f))(x),\;\;\;\;x\in\mathbb T^{n-1}.$$

By using Jensen inequality we get

\begin{align*}
  & \left\|V_\rho^X(\{\mathbb P_t^{(n-1)}\}_{t>0})(f)\right\|_{L^q(\mathbb T^{n-1},dx)}^q\leq \int_{\mathbb T^{n-1}}\int_{-1/2}^{1/2}\left|V_\rho^X(\{\mathbb P_t^{(n)}\}_{t>0})(E(f))(x_1,...,x_n)\right|^qdx_ndx_1...dx_{n-1} \\
    & \leq C \int_{\mathbb T^{n}}\|E(f)(x)\|^qdx = C\int_{\mathbb T^{n-1}}\|f(x)\|^qdx.
\end{align*}
By iterating this argument we prove that $V_\rho^X(\{\mathbb P_t^{(1)}\}_{t>0})$ is bounded from  $ L^q_X(\mathbb T^{1},dx)$ into $ L^q(\mathbb T^{1},dx)$. Then (see \cite[Section 2]{HLM} and \cite[(2.7)]{HLM}),  $X$ has the $\rho$-martingale cotype.

Thus the proof of the proposition is finished.

\end{proof}

\subsection{Proof of Theorem \ref{th1.5}}
According to Proposition \ref{Prop3.1} this proof will be completed if we establish that the following two properties are equivalent for every $i=1,...,n$.

\begin{enumerate}
    \item[(a)] $V_\rho^X(\{\mathcal C_{i,t}^{\mathcal A}\}_{t>0}\})$ is bounded from $L_X^p(\mathbb R^n,\gamma_{-1})$ into $L^p(\mathbb R^n,\gamma_{-1})$.

    \item[(b)] $V_\rho^X(\{Q_{i,t}\}_{t>0}\})$ is bounded from $L_X^p(\mathbb R^n,\gamma_{-1})$ into $L^p(\mathbb R^n,\gamma_{-1})$.
\end{enumerate}

Let $i=1,...,n$. By proceeding as in the proof of Theorem \ref{th1.4} we can prove that $(b)\Rightarrow (a)$. We now prove $(a)\Rightarrow (b)$. Assume that $(a)$ holds. We have that

$$V_{\rho,loc}^X(\{Q_{i,t}^{\mathcal A}\}_{t>0}\})(f)\leq V_{\rho,loc}^X(\{Q_{i,t}-\mathcal C_{i,t}^{\mathcal A}\}_{t>0}\})(f)+V_{\rho,loc}^X(\{\mathcal C_{i,t}^{\mathcal A}\}_{t>0}\})(f).$$

We recall that

$$\mathcal C_{i,t}^{\mathcal A}(x,y)=\frac{1}{\pi}\int_0^\infty\frac{e^{-t^2/4u}}{\sqrt{u}}\partial_{x_i}T_u^{\mathcal A}(x,y)du,\;\;\;\;x,y\in\mathbb R^n.$$
Then,

\begin{align*}
  & V_\rho^X(\{\mathcal C_{i,t}^{\mathcal A}(x,y)\}_{t>0}\})\leq  C\int_0^\infty\frac{\partial_{x_i} T_u^{\mathcal A}(x,y)}{\sqrt{u}}du\leq C\int_0^\infty\frac{|x_i-e^{-t}y_i|}{\sqrt{t}}\frac{e^{-|x-ye^{-t}|^2/(1-e^{-2t})}}{(1-e^{-2t})^{n/2+1}}e^{-nt}dt \\
    & \leq C\int_0^\infty\frac{e^{-c|x-ye^{-t}|^2/(1-e^{-2t})}}{\sqrt{t}(1-e^{-2t})^{\frac{n+1}{2}}}e^{-nt}dt \leq C\left(\int_{m(x)}^\infty\frac{1}{t^{n/2+1}}dt+\int_0^{m(x)}\frac{e^{-c|x-y|^2/t}}{t^{n/2+1}}dt\right)\\
    & \leq C \left(\frac{1}{m(x)^{n/2}}+\int_0^\infty\frac{e^{-c|x-y|^2/t}}{t^{n/2+1}}dt\right)\leq \frac{C}{|x-y|^n},\;\;\;\;(x,y)\in N_1.
\end{align*}
By arguing as in the proof of \cite[Proposition 3.2.7]{Sa} we deduce that $V_{\rho,loc}^X(\{\mathcal C_{i,t}^{\mathcal A}\}_{t>0}\})$ is bounded from $L_X^p(\mathbb R^n,\gamma_{-1})$ into $L^p(\mathbb R^n,\gamma_{-1})$. By using \cite[Proposition 3.2.5]{Sa} we get that $V_{\rho,loc}^X(\{\mathcal C_{i,t}^{\mathcal A}\}_{t>0}\})$ is bounded from $L_X^p(\mathbb R^n,dx)$ into $L^p(\mathbb R^n,dx)$.

As in the proof of Theorem \ref{th1.4} we can see that $V_{\rho,loc}^X(\{\mathcal C_{i,t}^{\mathcal A}-Q_{i,t}\}_{t>0}\})$ is bounded from $L_X^p(\mathbb R^n,dx)$ into $L^p(\mathbb R^n,dx)$.

For every $\lambda >0$ we define $f_\lambda(x)=f(\lambda x)$, $x\in\mathbb R^n$. We have that

$$Q_{i,t}(f_\lambda)(\frac{x}{\lambda})=Q_{i,\lambda t}(f)(x),\:\:\:\:x\in\mathbb R^n,\;\;\lambda,t>0.$$
It follows that

$$V_{\rho}^X(\{Q_{i,\lambda t}\}_{t>0}\})(f)(x)=V_{\rho}^X(\{Q_{i,t}\}_{t>0}\})(f_\lambda)(\frac{x}{\lambda}),\;\;\;\; x\in\mathbb R^n,\;\mbox{and}\;\lambda>0.$$
Assume that $f\in L_X^p(\mathbb R^n,dx)$ with compact support and $r >0$. Then there exists $\lambda >0$ such that $(x/\lambda,y)\in N_1$ provided that $y\in \supp f_\lambda$ and $|x|< r$. Then, $V_{\rho}^X(\{Q_{i,t}\}_{t>0}\})(f)(x)=V_{\rho,loc}^X(\{Q_{i,t}\}_{t>0}\})(f_\lambda)(\frac{x}{\lambda})$, $x\in\mathbb R^n$.

We have that (see \cite[p. 21]{HTV})

\begin{align*}
  & \int_{B(0,r)}|V_\rho^X(\{Q_{i,t}\}_{t>0}\})(f)(x)|^pdx  =  \int_{B(0,r/\lambda)}|V_{\rho,loc}^X(\{Q_{i,t}\}_{t>0}\})(f_\lambda)(x)|^pdx \lambda^n \\
    & \leq C\int_{\mathbb R^n}\|f_\lambda(x)\|^pdx\lambda^n=C\|f\|^p_{L_X^p(\mathbb R^n,dx)}.
\end{align*}
We get by letting $r\rightarrow\infty$

$$\|V_{\rho}^X(\{Q_{i,t}\}_{t>0}\})(f)\|_{L_X^p(\mathbb R^n,dx)}\leq C\|f\|_{L_X^p(\mathbb R^n,dx)}.$$
Thus, $(b)$ is proved and the proof is complete.

\subsection{Proof of Corollary \ref{cor1.6}}
According to Theorem \ref{th1.4}, for every $f\in L_X^p(\mathbb R^n,\gamma_{-1})$ there exists the limit $\displaystyle\lim_{t\rightarrow 0^+}\mathcal C^{\mathcal A}_{i,t}(f)(x)$, for almost all $x\in\mathbb R^n$. It also holds, for every $f\in L^p(\mathbb R^n,dx)$ that there exists the limit $\displaystyle\lim_{t\rightarrow 0^+}Q_{i,t}(f)(x)=\mathcal R_i(f)(x)$, for almost all $x\in\mathbb R^n$.

By $C_c^\infty(\mathbb R^n)$ we denote the space of smooth functions with compact support in $\mathbb R^n$. Let $f\in C_c^\infty(\mathbb R^n)$. For every $t>0$,

$$\mathcal C_{i,t}^{\mathcal A}(f)(x)-Q_{i,t}(f)(x)=\int_{\mathbb R^n}f(y)(\mathcal C_{i,t}^{\mathcal A}(x,y)-Q_{i,t}(x,y))dy,\;\;\;\;x\in\mathbb R^n,$$
being

$$\mathcal C_{i,t}^{\mathcal A}(x,y)-Q_{i,t}(x,y)=\frac{1}{\sqrt{\pi}}\int_0^\infty\frac{e^{-t^2/4u}}{\sqrt{u}}(\partial_{x_i}T_u^{\mathcal A}(x,y)-\partial_{x_i}T_u(x,y))du,\;\;\;\;x,y\in\mathbb R^n.$$
We choose $\delta=4/3$. According to (\ref{2.30}) we have that

$$|\mathcal C_{i,t}^{\mathcal A}(x,y)-Q_{i,t}(x,y)|\leq C\frac{(1+|x|)^{1/2}}{|x-y|^{n-1/2}},\;\;\;\;(x,y)\in N_\delta\;\mbox{and}\;t>0.$$
By (\ref{2.29}) we get

\begin{align*}
    |\mathcal C_{i,t}^{\mathcal A}(x,y)|\leq & C\left\{\begin{array}{lll}
       e^{-\eta|x|^2},  &  \mbox{if}\;\langle x,y\rangle\leq 0 & \\
         & & (x,y)\notin N_\delta, \\
        |x+y|^n exp\left(-\frac{\eta}{2}|x+y||x-y|-\frac{\eta}{2}(|x|^2-|y|^2)\right) & \mbox{if}\;\langle x,y\rangle \geq 0 &
    \end{array}  \right.
\end{align*}
where $\eta=3/4$.

Furthermore, we have that

\begin{align*}
    |Q_{i,t}(x,y)|\leq & C \int_0^\infty\frac{|x_i-y_i|}{t^{\frac{n+3}{2}}}e^{-|x-y|^2/2t}dt\leq C \int_0^\infty\frac{1}{t^{\frac{n+2}{2}}}e^{-c|x-y|^2/2t}dt \\
    & \leq C\frac{1}{|x-y|^n}\leq C(1+|x|)^n,\;\;\;\;(x,y)\in N_\delta^c.
\end{align*}
By using the above estimates we conclude that

$$|\mathcal C_{i,t}^{\mathcal A}(x,y)-Q_{i,t}(x,y)|\leq H(x,y),\;\;\;\;(x,y)\in \mathbb R^n\;\mbox{and}\;t>0,$$
where $H$ is a measurable function on $\mathbb R^n\times\mathbb R^n$ such that

$$\int_{\mathbb R^n}|f(y)|H(x,y)dy<\infty,\;\;\mbox{for every}\;\;x\in\mathbb R^n.$$

Then, by using dominated convergence theorem, we get

$$\lim_{t\rightarrow 0^+}(\mathcal C_{i,t}^{\mathcal A}(f)(x)-Q_{i,t}(f)(x))=\int_{\mathcal R^n}f(y)(\mathcal R_i^{\mathcal A}(x,y)-\mathcal R_i(x,y))dy,\;\;\;\;x\in\mathbb R^n,$$
where

$$\mathcal R_i^{\mathcal A}(x,y)=\frac{1}{\sqrt{\pi}}\int_0^\infty\partial_{x_i}T_u^{\mathcal A}(x,y)\frac{du}{\sqrt{u}},\;\;\;\;x,y\in\mathbb R^n,\;x\neq y.$$
and

$$\mathcal R_i(x,y)=\frac{1}{\sqrt{\pi}}\int_0^\infty\partial_{x_i}T_u(x,y)\frac{du}{\sqrt{u}},\;\;\;\;x,y\in\mathbb R^n,\;x\neq y.$$

Note that $|\mathcal R_i^{\mathcal A}(x,y)-\mathcal R_i(x,y)|\leq H(x,y)$, $x,y\in\mathbb R^n$. Then,

$$\int_{\mathbb R^n} f(y)(\mathcal R_i^{\mathcal A}(x,y)-\mathcal R_i(x,y))dy=\lim_{\epsilon\rightarrow 0^+}\int_{|x-y|>\epsilon} f(y)(\mathcal R_i^{\mathcal A}(x,y)-\mathcal R_i(x,y))dy,\;\;\;\;x\in\mathbb R^n.$$
Since $\lim_{t\rightarrow 0^+}Q_{i,t}(f)(x)=\mathcal R_i(f)(x)$, for almost all $x\in\mathbb R^n$, we conclude that

$$\lim_{\epsilon\rightarrow 0^+}\int_{|x-y|>\epsilon} f(y)\mathcal R_i^{\mathcal A}(x,y)dy=\lim_{t\rightarrow 0^+}\mathcal C_{i,t}^{\mathcal A}(f)(x),\;\;\mbox{for almost all}\;\;x\in\mathbb R^n.$$
Since $V_\rho(\{\mathcal C_{i,t}\}_{t>0})$ is bounded from $L^p(\mathbb R^n,\gamma_{-1})$ into itself, the operator $L$ defined by

$$L(f)(x)=\lim_{t\rightarrow 0^+}\mathcal C_{i,t}^{\mathcal A}(f)(x),$$
is bounded from $L^p(\mathbb R^n,\gamma_{-1})$ into itself. The Riesz transform $\mathcal R_i^{\mathcal A}$ is also bounded from $L^p(\mathbb R^n,\gamma_{-1})$ into itself. Since $C_c^\infty(\mathbb R^N)$ is a dense subspace of $L^p(\mathbb R^n,\gamma_{-1})$, it follows that, for every $f\in L^p(\mathbb R^n,\gamma_{-1})$,

$$\mathcal R_i^{\mathcal A}(f)(x)=\lim_{t\rightarrow 0^+}\mathcal C_{i,t}^{\mathcal A}(f)(x),\;\;\mbox{for almost all}\;\;x\in\mathbb R^n.$$

\section{Proof of Theorem \ref{th1.7}}
\subsection{An auxiliary result}
In this subsection we establish an extension of \cite[Theorem 1.1]{CT} that will  be useful to prove Theorem \ref{th1.7}.

Assume that $\{t_j\}_{j\in\mathbb Z}$ is an increasing sequence in $(0,\infty)$ such that $\displaystyle\lim_{j\rightarrow -\infty}t_j=0$ and $\{v_j\}_{j\in\mathbb Z}$ is a bounded sequence of complex numbers. For every $N=(N_1,N_2)\in\mathbb Z^2$, $N_1<N_2$ and $\alpha >0$, we define

$$T_N^\alpha(f)=\sum_{j=N_1}^{N_2}v_j(t^\alpha\partial_t^\alpha T_t(f)_{|t=t_{j+1}}-t^\alpha\partial_t^\alpha T_t(f)_{|t=t_j}).$$

We also consider the maximal operator

$$T_*^\alpha(f)=\sup_{\begin{array}{c} N=(N_1,N_2)\in\mathbb Z^2 \\ N_1<N_2 \end{array}}|T_N^\alpha(f)|.$$

The operators $P_N^\alpha$ and $P_*^\alpha$ are defined as above by replacing the heat semigroup $\{T_t\}_{t>0}$ by the Poisson semigroup $\{P_t\}_{t>0}$.

\begin{prop}\label{prop4.1}
Assume that $\{v_j\}_{j\in\mathbb Z}$ is a bounded sequence of complex numbers and $\{t_j\}_{j\in\mathbb Z}$ is a $\lambda$-lacunary with $\lambda >1$ in $(0,\infty)$.
Then the operator $W_*^\alpha$ (respectively $P_*^\alpha$) is bounded from $L^p(\mathbb R^n,dx)$ into itself, for every $1<p<\infty$  and from $L^1(\mathbb R^n,dx)$ into $L^{1,\infty}(\mathbb R^n,dx)$ when $\alpha\geq 1$ or $\alpha =0$ (respectively, $\alpha\geq 0$)
\end{prop}

In order to prove the auxiliary result we show in this Proposition, we use some ideas developed in the proof of \cite[Theorem 1.1]{CT}, where it is established that $T_*^0$ is bounded from $L^p(\mathbb R^n,dx)$ into itself, for every $1<p<\infty$  and from $L^1(\mathbb R^n,dx)$ into $L^{1,\infty}(\mathbb R^n,dx)$.We need to demonstrate a series of properties which are similar than the ones we can find in \cite{CT}. In our case we can not use them directly because we are dealing with fractional derivatives, so an appropriate proof for each of them is needed.

According to \cite[Lemma 2.3]{CT} we can assume that $\displaystyle 1<\lambda<\frac{t_{j+1}}{t_j}<\lambda^2$, $j\in\mathbb Z$.
\begin{enumerate}

    \item[(P1)] There exists $C>0$ such that, for every $m.k\in\mathbb Z$,
    $$\left|\sum_{j=m}^k v_j\left(t^\alpha\partial_t^\alpha T_t(x,y)_{|t=t_{j+1}}-t^\alpha\partial_t^\alpha T_t(x,y)_{|t=t_j}\right)\right|\leq\frac{C}{t_m^{n/2}},\;\;\;\;x,y\in\mathbb R^n.$$
    \begin{proof}
   Let $m,k\in\mathbb Z$, $m<k$, and $x,y\in\mathbb R^n$. We have that

  $$\left|\sum_{j=m}^k v_j\left(t^\alpha\partial_t^\alpha T_t(x,y)_{|t=t_{j+1}}-t^\alpha\partial_t^\alpha T_t(x,y)_{|t=t_j}\right)\right|\leq\|v\|_{\ell^\infty}\int_{t_m}^{t_{k+1}}|\partial_t(t^\alpha\partial_t^\alpha T_t(x,y))|dt.$$
 Let $\ell\in\mathbb N$ such that $\ell -1\leq\alpha <\ell$. By taking into account that $\partial_t\partial_t^\alpha T_t(z)=\partial_t^{\alpha +1}T_t(z)$,  $z\in \mathbb R^n$ and $t>0$, we deduce that, for every $x,y\in\mathbb R^n$ and $t>0$,

 \begin{align*}
    & \partial_t(t^\alpha\partial_t^\alpha T_t(x,y))= \partial_t\left(\frac{t^\alpha}{\Gamma(\ell -\alpha)}\int_0^\infty \partial_u^\ell T_u(x,y)_{|u=t+s}s^{\ell -\alpha -1}ds\right) \\
    & =\frac{\alpha t^{\alpha-1}}{\Gamma(\ell -\alpha)}\int_0^\infty \partial_u^\ell T_u(x,y)_{|u=t+s}s^{\ell -\alpha -1}ds  +\frac{t^\alpha}{\Gamma(\ell -\alpha)}\int_0^\infty \partial_u^{\ell +1}T_u(x,y)_{|u=t+s}s^{\ell -\alpha -1}ds.
\end{align*}
We get

\begin{align*}
    & \int_{t_m}^{t_{k+1}} t^{\alpha-1}\int_0^\infty|\partial_u^\ell T_u(x,y)_{|u=t+s}|s^{\ell -\alpha -1}dsdt\leq C\int_{t_m}^\infty t^{\alpha-1}\int_0^\infty\frac{e^{-c|x-y|^2/(s+t)}}{(s+t)^{n/2+\ell}}s^{\ell -\alpha -1}dsdt \\
    & \leq C \int_{t_m}^\infty \frac{e^{-c|x-y|^2/u}}{u^{n/2+\ell}}\int_{t_m}^u t^{\alpha -1}(u-t)^{\ell -\alpha -1}dtdu\leq C \int_{t_m}^\infty \frac{e^{-c|x-y|^2/u}}{u^{n/2+\ell}}\int_0^u t^{\alpha -1}(u-t)^{\ell -\alpha -1}dtdu \\
    & \leq C\int_{t_m}^\infty \frac{e^{-c|x-y|^2/u}}{u^{n/2+1}}du \leq \frac{C}{t_m^{n/2}},\;\;\;\;x,y\in\mathbb R^n,
\end{align*}
and

$$\int_{t_m}^{t_{k+1}} t^\alpha\int_0^\infty |\partial_u^{\ell+1} T_u(x,y)_{|u=t+s}|s^{\ell -\alpha -1}dsdt\leq \frac{C}{t_m^{n/2}},\;\;\;\;x,y\in\mathbb R^n.$$
Thus we conclude the desire inequality.
    \end{proof}

    \item[(P2)] There exists $C>0$ such that, for every $k\geq m\geq -\ell +1$ and $x,y\in\mathbb R^n$ verifying that $|x-y|\geq \frac{1}{2}t_k^{1/2}$,

$$\left|\sum_{j=-\ell}^{m-1} v_j\left(t^\alpha\partial_t^\alpha T_t (x,y)_{|t=t_{j+1}}-t^\alpha\partial_t^\alpha T_t(x,y)_{|t=t_j}\right)\right|\leq C\frac{\lambda^{-(k-m+1)}}{t_k^{n/2}}.$$
\begin{proof}
Suppose that $k\geq m\geq -\ell +1$ and $x,y\in\mathbb R^n$ such that $|x-y|\geq \frac{1}{2}t_k^{1/2}$. We have that

\begin{align*}
& \left|\sum_{j=-\ell}^{m-1} v_j\left(t^\alpha\partial_t^\alpha T_t(x,y)_{|t=t_{j+1}}-t^\alpha\partial_t^\alpha T_t(x,y)_{|t=t_j}\right)\right|\leq\|v\|_{\ell^\infty}\int_{t_{-\ell}}^{t_m}|\partial_t(t^\alpha\partial_t^\alpha T_t(x,y))|dt \\
 & \leq C\left(\int_{t_{-\ell}}^{t_m}|t^\alpha\partial_t^{\alpha +1} T_t(x,y)|dt+\int_{t_{-\ell}}^{t_m}|t^{\alpha -1}\partial_t^\alpha T_t(x,y))|dt\right),\;\;x,y\in\mathbb R^n.
 \end{align*}

We also get

\begin{align}\label{4.1}
    |\partial_t^\alpha T_t(z)| & \leq C\int_0^\infty|\partial_u^m T_u(z)_{|u=t+s}|s^{m-\alpha -1}ds \nonumber \\
    & \leq \frac{C}{t^{\frac{n}{2}+\alpha}}\int_1^\infty \frac{e^{-c|z|^2/tu}}{u^{m+n/2}}(u-1)^{m-\alpha -1}du,\;\;\;z\in\mathbb R^n\;\mbox{and}\;t>0.
\end{align}
Then,

 \begin{align*}
     \int_{t_{-\ell}}^{t_m}|\partial_t(t^\alpha\partial_t^\alpha T_t(x,y))|dt & \leq C\int_{t_{-\ell}}^{t_m}\frac{1}{t^{\frac{n}{2}+1}}\int_1^\infty\frac{e^{-t_k/tu}}{u^{m+\frac{n}{2}}}(u-1)^{m-\alpha -1}dudt \\
   & \leq C\frac{1}{t_k^{\frac{n}{2}}}\int_{t_k/t_m}^{t_k/t_{-\ell}}s^{\frac{n}{2}-1}\int_1^\infty\frac{e^{-s/u}}{u^{m+\frac{n}{2}}}(u-1)^{m-\alpha -1}duds \\
   & \leq C\frac{1}{t_k^{\frac{n}{2}}}\frac{t_m}{t_k}\int_{t_k/t_m}^{t_k/t_{-\ell}}s^{\frac{n}{2}}\int_1^\infty\frac{e^{-s/u}}{u^{m+\frac{n}{2}}}(u-1)^{m-\alpha -1}duds \\
    & \leq C\frac{1}{t_k^{\frac{n}{2}}}\frac{t_m}{t_k}\int_1^\infty\frac{(u-1)^{m-\alpha -1}}{u^{m+\frac{n}{2}}}\int_1^\infty s^{\frac{n}{2}}e^{-s/u}dsdu \\
    & \leq C\frac{1}{t_k^{\frac{n}{2}}}\frac{t_m}{t_k}\leq C\frac{\lambda^{-(k-m+1)}}{t_k^{\frac{n}{2}}},
\end{align*}
being $\alpha >1$. When $\alpha =1$, we also have that

 \begin{align*}
   \int_{t_{-\ell}}^{t_m}|\partial_t(t\partial_t T_t(x,y))|dt & \leq C\int_{t_{-\ell}}^{t_m}\frac{e^{-c|x-y|^2/t}}{t^{\frac{n}{2}+1}}dt \\
   & \leq C\frac{1}{t_k^{\frac{n}{2}}}\int_{t_k/t_m}^{t_k/t_{-\ell}}u^{\frac{n}{2}-1}e^{-cu}du \\
   & \leq C\frac{1}{t_k^{\frac{n}{2}}}\frac{t_m}{t_k}\leq C\frac{\lambda^{-(k-m+1)}}{t_k^{\frac{n}{2}}}.
\end{align*}

\end{proof}

If $N=(N_1,N_2)\in\mathbb Z^2$ such that $N_1<N_2$, the integral kernel of the operator  $T_N^\alpha$ is given by

$$T_N^\alpha(x,y)=\sum_{j=N_1}^{N_2} v_j\left(t^\alpha\partial_t^\alpha T_t (x,y)_{|t=t_{j+1}}-t^\alpha\partial_t^\alpha T_t(x,y)_{|t=t_j}\right),\;\;\;\;x,y\in\mathbb R^n,$$
and we are going to see that it is a standard Calder\'on-Zygmund kernel uniformly in $N$.

    \item[(P3)] There exists $C>0$ such that, for every $N=(N_1,N_2)\in\mathbb Z^2$ with $N_1<N_2$,

    $$|T_N^\alpha(x,y)|\leq\frac{C}{|x-y|^n},\;\;\;\;x,y\in\mathbb R^n,\;\;x\neq y.$$
\begin{proof}
  Let  $N=(N_1,N_2)\in\mathbb Z^2$, $N_1<N_2$. By (\ref{4.1}) we have that

  \begin{align*}
     |T_N^\alpha(x,y)| & \leq C\|v\|_{\ell^\infty}\int_{t_{N_1}}^{t_{N_2+1}}|\partial_t(t^\alpha\partial_t^\alpha T_t(x,y))|dt \\
    &\leq C\|v\|_{\ell^\infty}\int_{t_{N_1}}^{t_{N_2+1}}\frac{1}{t^{\frac{n}{2}+1}}\int_1^\infty\frac{e^{-c|x-y|^2/tu}}{u^{m+\frac{n}{2}}}(u-1)^{m-\alpha -1}dudt \\
   & \leq \frac{C}{|x-y|^n}\|v\|_{\ell^\infty}\int_0^\infty v^{\frac{n}{2}-1}\int_1^\infty\frac{e^{-cv/u}}{u^{m+\frac{n}{2}}}(u-1)^{m-\alpha -1}dudv \\
   & \leq \frac{C}{|x-y|^n}\|v\|_{\ell^\infty}\int_1^\infty \frac{(u-1)^{m-\alpha -1}}{u^m}du\int_0^\infty e^{-cv}v^{\frac{n}{2}-1}dv \\
    & \leq \frac{C}{|x-y|^n},\;\;\;\;x,y\in\mathbb R^n,\;\;x\neq y.
\end{align*}
\end{proof}

    \item[(P4)] There exists $C>0$ such that, for every $N=(N_1,N_2)\in\mathbb Z^2$ with $N_1<N_2$,

    $$|T_N^\alpha(x,y)-T_N^\alpha(x,z)|+|T_N^\alpha(y,x)-T_N^\alpha(z,x)|\leq C\frac{|y-z|}{|x-y|^{n+1}},\;\;\;\;|x-y|>2|y-z|.$$
\begin{proof}
  Let  $N=(N_1,N_2)\in\mathbb Z^2$, $N_1<N_2$. Suppose that $x,y,z\in\mathbb R^n$ such that $|x-y|>2|y-z|$. We can write

  \begin{align*}
 |T_N^\alpha(x,y)- & T_N^\alpha(x,z)|\leq C\|v\|_{\ell^\infty}\left(\int_{t_{N_1}}^{t_{N_2+1}}\alpha t^{\alpha -1}|\partial_t^\alpha T_t(x,y)-\partial_t^\alpha T_t(x,z)|dt\right. \\
    & \left. +\int_{t_{N_1}}^{t_{N_2+1}} t^\alpha|\partial_t^{\alpha +1} T_t(x,y)-\partial_t^{\alpha +1} T_t(x,z)|dt\right), \;\;\;\;x,y\in\mathbb R^n.
\end{align*}
Let $i=1,..,n$. We have that, as in (\ref{4.1}),

\begin{align*}
    & \left|\int_{t_{N_1}}^{t_{N_2+1}}t^{\alpha -1}\partial_{z_i}\partial_t^{\alpha}T_t(z)dt\right|\leq C\int_{t_{N_1}}^{t_{N_2+1}}\frac{1}{t^{\frac{n+1}{2}+1}}\int_1^\infty\frac{e^{-c|z|^2/tu}}{u^{m+\frac{n+1}{2}}}(u-1)^{m-\alpha -1}dudt \\
   & \leq \frac{C}{|z|^{n+1}}\int_0^\infty v^{\frac{n-1}{2}}\int_1^\infty\frac{e^{-cv/u}}{u^{m+\frac{n+1}{2}}}(u-1)^{m-\alpha -1}dudv \\
   & \leq \frac{C}{|z|^{n+1}}\int_1^\infty \frac{(u-1)^{m-\alpha -1}}{u^m}du\int_0^\infty e^{-cv}v^{\frac{n-1}{2}}dv.
\end{align*}
We deduce that there exists $C$  not depending on $N$ such that

$$|T_N^\alpha(x,y)-  T_N^\alpha(x,z)|\leq C\frac{|y-z|}{|x-y|^{n+1}}.$$
In a similar way we can get the same estimate for the second term in the sum.
\end{proof}
$\;$ \newline

Finally we prove that $\{T_N^\alpha\}_{N=(N_1,N_2)\in\mathbb Z^2}$ is a bounded set in the space of the bounded operator in $L^2(\mathbb R^n,dx)$.

    \item[(P5)] There exists $C>0$ such that, for every $N=(N_1,N_2)\in\mathbb Z^2$, $N_1<N_2$,

    $$\|T_N^\alpha(f)\|_{L^2(\mathbb R^n,dx)}\leq C\|f\|_{L^2(\mathbb R^n,dx)},\;\;\;\;f\in L^2(\mathbb R^n,dx).$$
\begin{proof}
    Let $N=(N_1,N_2)\in\mathbb Z^2$, $N_1<N_2$. Plancharel's equality leads to

\begin{align*}
    & \|T_N^\alpha(f)\|_{L^2(\mathbb R^n,dx)}^2=\int_{\mathbb R^n}|\sum_{j=N_1}^{N_2} v_j\widehat{(t^\alpha\partial_t^\alpha T_t(f))}(x)_{|t=t_{j+1}}-\widehat{(t^\alpha\partial_t^\alpha T_t(f))}(x)_{|t=t_j}|^2 dx \\
   & \leq C\int_{\mathbb R^n}|\sum_{j=N_1}^{N_2} ((t^\alpha\partial_t^\alpha e^{-t|x|^2})_{|t=t_{j+1}}-(t^\alpha\partial_t^\alpha e^{-t|x|^2})_{|t=t_j})\hat{f}(x)|^2 dx \\
   & = C \int_{\mathbb R^n}|\sum_{j=N_1}^{N_2} \int_{t_j}^{t_{j+1}}\partial_t(t^\alpha\partial_t^\alpha e^{-t|x|^2})dt|^2|\hat{f}(x)|^2 dx \\
   & \leq C \int_{\mathbb R^n}|\hat{f}(x)|^2\left(\int_0^\infty|t^{\alpha -1}\partial_t^\alpha e^{-t|x|^2}|dt+\int_0^\infty|t^\alpha\partial_t^{\alpha +1}e^{-t|x|^2}|dt\right)^2dx \\
   & \leq C \int_{\mathbb R^n}|\hat{f}(x)|^2\left(\int_0^\infty|t^{\alpha -1}|x|^{2\alpha} e^{-t|x|^2}|dt+\int_0^\infty|t^\alpha|x|^{2(\alpha +1)}e^{-t|x|^2}|dt\right)^2dx \\
   & =C\int_{\mathbb R^n}|\hat{f}|^2(x)dx=C\|f\|^2_{L^2(\mathbb R^n,dx)}.
\end{align*}
We have used that $\partial_t^\alpha(e^{-t|x|^2})=(-1)^m|x|^{2\alpha}e^{-t|x|^2}$, $x\in\mathbb R^n$ and $t>0$.
\end{proof}
\end{enumerate}

This last property can be seen in a more general setting by using spectral theory (see \cite[p. 12]{CT}). By using Calder\'on-Zygmund theory we deduce that, for every $1<r<\infty$, there exists $C>0$ such that, for every $(N_1,N_2)\in\mathbb Z^2$, $N_1<N_2$,

$$\|T_N^\alpha(f)\|_{L^r(\mathbb R^n,dx)}\leq C\|f\|_{L^r(\mathbb R^n,dx)},\;\;\;\;f\in L^r(\mathbb R^n,dx).$$
For every $K\in\mathbb N$ we define the maximal operator $T_{K,*}^\alpha$ as follows

$$T_{K,*}^\alpha(f)=\sup_{-K\leq N_1<N_2\leq K}|T_{(N_1,N_2)}^\alpha(f)|.$$
We have establish all the ingredients that we need to prove, following the strategy developed in the proof of \cite[Theorem 3.11]{CT}, that for every $1<q<\infty$ there exists $C>0$ such that, for every $K\in\mathbb N$,

$$T_{K,*}^\alpha(f)\leq C(\mathcal M(T_{(-K,K)}^\alpha(f))+\mathcal M_q(f)),$$
where $\mathcal M$ represents the Hardy-Littlewood function and $\mathcal M_q(f)=(\mathcal M(|f|^q))^{1/q}$.

Suppose now that $1<p<\infty$. We choose $1<q<p$. Since $\mathcal M$ and $\mathcal M_q$ are bounded from $L^p(\mathbb R^n,dx)$ into itself, there exists $C>0$ such that, for every $K\in\mathbb N$,

$$\|T_{K,*}^\alpha(f)\|_{L^p(\mathbb R^n,dx)}\leq C\|f\|_{L^p(\mathbb R^n,dx)},\;\;\;\;f\in L^p(\mathbb R^n,dx).$$
By letting $K\rightarrow +\infty$, Fatou's lemma leads to

$$\|T_*^\alpha(f)\|_{L^p(\mathbb R^n,dx)}\leq C\|f\|_{L^p(\mathbb R^n,dx)},\;\;\;\;f\in L^p(\mathbb R^n,dx).$$
In order to see that $T_*^\alpha$ is bounded from $L^1(\mathbb R^n,dx)$ into $L^{1,\infty}(\mathbb R^n,dx)$ we use vector valued Calder\'on-Zygmund theory. Let $K\in\mathbb N$. We can write $T_{K,*}^\alpha(f)=\|T_{(\cdot,\cdot)}^\alpha(f)\|_K$, where

$$\|(a_{(m_1,m_2)})_{-K\leq m_1<m_2\leq K}\|_K=\sup_{-K\leq m_1<m_2\leq K}|a_{(m_1,m_2)}|,\;\;\;\;(a_{(m_1,m_2)})_{-K\leq m_1<m_2\leq K}\in\mathbb R^{\beta_K},$$
where $\beta_k=\#\{(m_1,m_2)\;:\;-K\leq m_1<m_2\leq K\}$.

By taking into account the properties we have proved we can deduce from the $(\mathbb R^{\beta_K},\|\;\|_K)$-valued Calder\'on-Zygmund theorem that there exists $C>0$ such that, for every $k\in\mathbb N$,

$$|\{x\in\mathbb R^n\;:\:T_{K,*}^\alpha(f)(x)>\lambda\}|\leq C\frac{\|f\|_{L^1(\mathbb R^n,dx)}}{\lambda},\;\;\;\;f\in L^1(\mathbb R^n,dx).$$
By letting $K\rightarrow +\infty$ we conclude that $T_*^\alpha$ is bounded from $L^1(\mathbb R^n,dx)$ into $L^{1,\infty}(\mathbb R^n,dx)$.

We now study $P_*^\alpha$. By using subordination formula and \cite[Lemma 4]{BCCFR} we get

\begin{align*}
  |\partial_t^\alpha P_t(z)| & \leq C\int_0^\infty\frac{|\partial_t^\alpha[te^{-t^2/4u}]|}{u^{3/2}}T_u(z)du\leq C\int_0^\infty\frac{e^{-t^2/8u}}{u^{\frac{\alpha +2+n}{2}}}e^{-|z|^2/2u} du\\
   & \leq C\frac{1}{(t^2+|z|^2)^{\frac{\alpha +n}{2}}},\;\;\;\;z\in\mathbb R^n\;\mbox{and}\;t>0.
\end{align*}
By proceeding as above we can prove the following properties
\begin{enumerate}
    \item[(P1)] For every $m,k\in\mathbb Z$, such that $m<k$,
    $$\left|\sum_{j=m}^k v_j\left(t^\alpha\partial_t^\alpha P_t(x-y)_{|t=t_{j+1}}-t^\alpha\partial_t^\alpha P_t(x-y)_{|t=t_j}\right)\right|\leq\frac{C}{t_m^{n}}.$$

    \item[(P2)] For every $k\geq m\geq -\ell +1$ and $x,y\in\mathbb R^n$ verifying that $|x-y|\geq \frac{1}{2}t_k$,

$$\left|\sum_{j=-\ell}^{m-1} v_j\left(t^\alpha\partial_t^\alpha P_t (x-y)_{|t=t_{j+1}}-t^\alpha\partial_t^\alpha P_t(x-y)_{|t=t_j}\right)\right|\leq C\frac{\lambda^{m-k}}{t_k^{n}}.$$
    The operators $P_N^\alpha$, $N=(N_1,N_2)\in\mathbb Z^2$, $N_1<N_2$, are defined in the obvious way.
    \item[(P3)]
    $$|P_N^\alpha(x,y)|\leq\frac{C}{|x-y|^n},\;\;\;\;x,y\in\mathbb R^n,\;\;x\neq y.$$

    \item[(P4)]
    $$|P_N^\alpha(x,y)-P_N^\alpha(x,z)|+|P_N^\alpha(y,x)-P_N^\alpha(z,x)|\leq C\frac{|y-z|}{|x-y|^{n+1}},\;\;\;\;|x-y|>2|y-z|.$$

    \item[(P5)]
    $$\|P_N^\alpha(f)\|_{L^2(\mathbb R^n,dx)}\leq C\|f\|_{L^2(\mathbb R^n,dx)},\;\;\;\;f\in L^2(\mathbb R^n,dx).$$
\end{enumerate}
The constant $C$ in (P1) and (P2) does not depend on $m,k$ and $\ell$, and in (P3), (P4) and (P5) does not depend on $N$.

\subsection{Proof of Theorem \ref{th1.7}}
We define the local and global operators as usual. We can write, with $\delta >0$,

\begin{align*}
    & T_{*,glob}^{\mathcal A,\alpha}(f)(x)=\sup_{\begin{array}{c} N=(N_1,N_2)\in\mathbb Z^2 \\ N_1<N_2 \end{array}}\left|\sum_{j=N_1}^{N_2} v_j\left(t^\alpha\partial_t^\alpha T_{t,glob}^{\mathcal A}(f)(x)_{|t=t_{j+1}}-t^\alpha\partial_t^\alpha T_{t,glob}^{\mathcal A}(f)(x)_{|t=t_j}\right)\right|\\
   & \leq \|v\|_{\ell^\infty}\int_0^\infty|\partial_t(t^\alpha\partial_t^\alpha T_{t,glob}^{\mathcal A}(f)(x))|dt \leq \|v\|_{\ell^\infty}\int_{\mathbb R^n}\chi_{N_\delta^c}(x,y)|f(y)|\int_0^\infty|\partial_t(t^\alpha\partial_t^\alpha T_t^{\mathcal A}(x,y))|dt,\;\;\;\;x\in\mathbb R^n,
\end{align*}
and

\begin{align*}
    & |T_{*,loc}^{\mathcal A,\alpha}(f)(x)-T_{*,loc}^{\alpha}(f)(x)| \leq \|v\|_{\ell^\infty}\int_{\mathbb R^n}\chi_{N_\delta}(x,y)|f(y)|\int_0^\infty|\partial_t(t^\alpha\partial_t^\alpha (T_t^{\mathcal A}(x,y)-T_t(x,y)))|dt,\;\;\;\;x\in\mathbb R^n,
\end{align*}
By using Proposition \ref{4.1} and by proceeding as in the proof of Theorem \ref{th1.1} we can prove the properties stated in Theorem \ref{th1.7}.

\begin{rem}
In the Gaussian case, that is, when the Ornstein-Uhlembeck operator is considered, we can establish an analogous of Theorem \ref{th1.7} but we only know how to show that the operator $T_*^{OU,\alpha}$ is bounded from $L^1(\mathbb R^n,dx)$ into $L^{1,\infty}(\mathbb R^n,dx)$,  when $\alpha=0$.
\end{rem}


\end{document}